\documentclass{iopjournal}

\usepackage[english]{babel}

\usepackage{placeins}
\usepackage{booktabs}
\usepackage{makecell}
\usepackage{enumerate}

\usepackage{amssymb}
\usepackage{amsmath}
\usepackage{amsthm}
\usepackage{mathtools}
\usepackage{tikz}       
\usepackage{nicematrix}
\usepackage{orcidlink}

\DeclareMathOperator*{\argmax}{arg\,max}
\DeclareMathOperator*{\argmin}{arg\,min}
\DeclareMathOperator*{\RayTran}{\mathcal{X}}
\newcommand{\vecOp}[1]{\boldsymbol{#1}}
\newcommand{\matOp}[1]{\mathsf{#1}}
\newcommand{\Mlines}{\mathbb{M}}
\newcommand{\DynLines}{\mathbb{W}}
\newcommand{\Mdomain}{\mathbb{D}}
\newcommand{\be}{\begin{equation}}
\newcommand{\ee}{\end{equation}}
\newcommand{\bea}{\begin{equation*}}
\newcommand{\eea}{\end{equation*}}
\newcommand{\SO}{\operatorname{SO}}
\newcommand{\SE}{\operatorname{SE}}
\newcommand{\Aff}{\operatorname{Aff}}
\newcommand{\GL}{\operatorname{GL}}
\newcommand{\R}{\mathbb{R}}
\newcommand{\I}{\mathbb{I}}

\allowdisplaybreaks
\numberwithin{equation}{section}

\newtheorem{theorem}{Theorem}[section]
\newtheorem{proposition}[theorem]{Proposition}

\newtheorem{corollary}[theorem]{Corollary}
\newtheorem{definition}[theorem]{Definition}

\newtheorem{remark}{Remark}

\begin{document}

\articletype{Paper}

\title{Motion-Enabled Tomography of Gaussian-Mixture Particles under Unknown Rigid-Body Motion with Parabolic Trajectories}

\author{Daniel Burrows$^{1,*}$\orcidlink{0009-0001-6210-5583}, Ozan {\"O}ktem$^2$, and Can Evren Yarman$^{2,3}$}

\affil{$^1$Institute for Mathematical Innovation, University of Bath, Bath, UK}
\affil{$^2$Department of Mathematics, KTH Royal Institute of Technology, Stockholm, Sweden}
\affil{$^3$Schlumberger Riboud Product Center, Clamart, France}
\affil{$^*$Author to whom any correspondence should be addressed.}

\email{dwb26@bath.ac.uk}


\begin{abstract}
Reconstructing objects with unknown position or orientation arises in applications spanning geology to cryo-ET. When the relative motion between the object and the sensing apparatus provides sufficient angular coverage, computerized tomography offers a means of reconstruction. For such scenarios, we propose a parametric spatiotemporal model applied to Gaussian mixture models (GMMs), in which each constituent Gaussian is parameterized by its own angular velocity, parabolic trajectory, and morphology. GMMs are a suitable representation because they (i) admit accurate approximations in object space and (ii) have a closed-form expression under the ray transform, enabling efficient forward predictions and analysis in data space. We show that the projected modes of the Gaussians are determined by their trajectories alone (exactly for isotropic Gaussians and approximately otherwise), so the projectile motion can be decoupled from the rotation and morphology and the full optimization problem reduced to two lower-dimensional sub-problems. The method is formulated in Euclidean space of arbitrary dimension and is evaluated on noisy 2D simulations with a single fixed source, where it achieves lower median reconstruction error than two baselines that fit the projection data directly. Its accuracy degrades as the particles increasingly overlap in the projections; we discuss the resulting identifiability limits and the extensions needed for higher dimensions, particles composed of several Gaussians, and physical data.
\end{abstract}

\section{Introduction}
Tomography refers to imaging technologies that enable visualization of the internal structure of objects in a non-destructive manner \cite{natterer2001mathematics,kak2001principles,herman2009fundamentals}. 
The basic idea is to repeatedly probe the object with waves or particles from different directions, measure how they are altered after passing through or interacting with the object, and use these measurements to reconstruct a cross-sectional or three-dimensional image of its interior.
The mathematical foundations of tomographic reconstruction, including the underlying integral transforms and the ill-posedness of the associated inverse problems, are developed in detail in the classical monographs \cite{natterer2001mathematics,kak2001principles,natterer2001mathematical,herman2009fundamentals}.
In many applications, it is however desirable to recover a model of the object. 
A ``model'' refers here to a representation of the object that goes beyond the traditional image format, which typically consists of an array of pixel or voxel values.
Some applications involve tomographic imaging of objects that move during the data acquisition; if the motion is unknown, then it also needs to be estimated.

An example of tomographic imaging that involves both model building and motion estimation is cryo-EM single-particle analysis, where tomographic imaging principles are used to reconstruct 3D molecular structures from a series of 2D transmission electron microscope images. 
Assuming homogeneous particles, the reconstruction problem amounts to recovering the atomic structure of a protein (model) from a collection of 2D electron microscopy images. 
Each 2D image represents tomographic projection data of the protein in an unknown 3D orientation, so the process can be understood as tomographic imaging of a protein that undergoes unknown rotations (motion) during data acquisition.

Other examples of such imaging problems arise in industrial tomographic imaging where it is crucial to acquire data quickly while keeping the instrumentation as simple and cost-efficient as possible. One way to reduce costs is to use a fixed source–detector setup and allow the object to move through the field of view. Ideally, this motion is controlled; however, in low-cost scenarios it may be uncontrolled, e.g., a freely falling object undergoing rotation. This calls for joint motion estimation and reconstruction.
Additionally, one can also include model building to represent the objects being imaged in a more appropriate way.
The use case we have as the underlying motivation for the work in this paper comes from industrial tomography.
It deals with computerized tomography (CT) imaging for subsurface characterization.
The instrumental setup leads to tomographic reconstruction of particles that undergo unknown motions.
We are primarily interested in recovering the morphological shape of the particles, which we model by a Gaussian mixture model (GMM).
Sections~\ref{sec:SubSurfChar}--\ref{sec:SubSurfCharCT} provide further details and Section~\ref{sec:SubSurfCharChall} outlines the challenges associated with this use case.

The common approach for tomographic reconstruction with a downstream post-processing step, like model building, is sequential: first reconstruct an image of the object and its internal structure, then use this image as input to a separate post-processing step that builds the desired model. 
Alternatively, one can integrate the reconstruction and post-processing steps, i.e., jointly perform reconstruction and model building.
One can adopt a similar sequential approach for reconstruction with an upstream pre-processing step, like motion estimation.
Just as with post-processing, one can integrate the motion estimation and reconstruction steps, i.e., perform motion estimation jointly with image reconstruction.

\emph{The approach considered in this paper is to jointly perform the following three tasks: motion estimation, tomographic reconstruction, and model building.} 
Such a joint approach may appear more complex than a sequential one, but it also makes it possible to use more efficient regularization methods.
This is needed as tomographic reconstruction of an object that undergoes an unknown motion is a severely ill-posed inverse problem.
A key advantage that comes with a joint formulation that also includes model building is that the latter can act as a form of regularization for the motion estimation and reconstruction. 
However, this benefit depends on adopting a carefully chosen representation.
In our framework, the object is modeled using a GMM, where it is expressed as a linear combination of Gaussian components. 
This is a natural model in many use cases, e.g., for representing proteins where atoms are represented by Gaussians \cite{TEMsim2011} or in a multi-scale setting where Gaussians represent residues \cite{Bafna:2025aa}.
GMMs have found widespread use in the above-mentioned cryo-EM single-particle analysis \cite{Chen2023GMMCryoEM, qu2025gem, chen2025cryosplat} and in electron tomography \cite{Singh2025GaussianET, Zhang2026GaussianSTEM}.
Another important application area is computer vision and robotics, where Gaussian-based scene representations play a central role in many successful approaches to 3D scene reconstruction \cite{Halacheva2025GaussianVLM, Sun2024EmbodiedGS, Kong2024Survey3DGSRobotics}. A notable example is 3D Gaussian splatting \cite{Kerbl2023GaussianSplatting}, which, like neural radiance fields (NeRFs), recovers a 3D scene from 2D images captured from multiple viewpoints, but represents it explicitly as a mixture of anisotropic Gaussians; see \cite{Gao2025NeRFReview, Yao2026NeRF} for reviews and \cite{Bouaziz2025GaussianNeuralFields} for a related Gaussian neural field representation.
On a final note, we also mention that the joint approach taken in this paper is part of an active research area within inverse problems on developing theory and algorithms for jointly solving an ill-posed inverse problem (reconstruction) with a downstream post-processing task \cite{Louis:2011aa,Oktem:Adler:2022aa}.
Similarly, it also aligns with the current drive towards multimodal learning, which is a type of deep learning that integrates and jointly processes multiple different data types (modalities), like text, audio, images, or video \cite{Liang:2024aa,Yin:2024aa,Yuan:2025aa}.

\paragraph{Relation to prior work on GMM tomography.}
The present framework builds on, and substantially extends, the joint Gaussian dictionary learning and tomographic reconstruction method of \cite{zickert2022joint} (which shares two of the present authors). That work reconstructs a \emph{static} object from a fixed and \emph{known} multi-view acquisition geometry: it recovers a sparse GMM representation from sparse-view, limited-angle parallel-beam data via a greedy continuous matching pursuit, initialized by a Gaussian filtered back-projection built from the closed-form Riesz potential of a Gaussian. Our contribution differs in three essential respects. First, we introduce an explicit spatiotemporal model in which \emph{each} Gaussian carries its own unknown rigid motion---a parabolic trajectory together with a constant angular velocity---so that the object is dynamic and the effective acquisition geometry is unknown and time-varying rather than fixed and known. Second, we operate in a minimal geometry, a single fixed source and detector panel, in which the angular coverage required for reconstruction is generated by the motion itself rather than by acquiring many static views. Third, we derive the closed-form ray transform of a \emph{moving} anisotropic Gaussian (Proposition~\ref{prop: ray of stationary gaussian}) and establish a motion--morphology decoupling (Proposition~\ref{prop: mode of ray of ansitropic gaussian} and Corollary~\ref{cor: modes of object and its particles}), whereby the trajectories are recovered from the projection-domain modes independently of the shape and rotation, leading to the two-stage estimator of Section~\ref{sec:IP}. Consequently, whereas \cite{zickert2022joint} estimates only static morphology under a known geometry, we jointly estimate motion \emph{and} morphology under an unknown one. The two approaches nonetheless share their analytical foundation: the ray transform maps a GMM to a GMM with closed-form projections and analytic gradients, and both rely on the Gaussian mode-localization bounds of \cite{zickert2022gaussian}; indeed the static counterpart of our Corollary~\ref{cor: modes of object and its particles} is \cite[Theorem~3.4]{zickert2022joint}.

\subsection{Subsurface characterization}\label{sec:SubSurfChar}
One of the main tasks in the reservoir characterization process is to determine petrophysical properties of the excavated rocks (porosity, saturation, permeability) early in the process. This enables one to characterize the reserves as well as their productivity and deliverability. Although logs generally provide a good estimate of porosity and saturation along the well, permeability is harder to assess especially in carbonates. Drill cuttings can provide lithological and other additional information on the petrophysical properties of the reservoir \cite{Lenormand2007Cuttings}.

Lithology interpretation refers to the process of identifying and interpreting the lithology of the geological layers from surface and subsurface drilling measurements. This plays a foundational role in subsurface exploration and drilling operations, whether for geothermal energy or conventional oil and gas resources. It involves the identification and characterization of rock types and their properties, which directly influence drilling strategy, reservoir evaluation, and resource extraction efficiency. Accurate lithological interpretation enables engineers and geoscientists to optimize well placement, anticipate drilling hazards, and assess the viability of subsurface formations.

Lithology interpretation is particularly critical in geothermal drilling for locating heat-bearing formations and understanding the thermal and hydraulic behavior of the subsurface. It helps to identify permeable zones, estimate temperature gradients, and guide well design to maximize energy recovery. In conventional oil and gas drilling, lithology interpretation informs reservoir characterization, petrophysical analysis, and structural mapping. It supports the identification of productive zones, seal integrity, and hydrocarbon potential, while also aiding in the mitigation of drilling risks such as overpressured formations or unstable shales \cite{lai2023toward, lithology_rock_types, sun2025enhanced}.

One of the most widely used methods for lithology identification is visual analysis of drilled cuttings. Drilled cuttings are often angular, fractured, and irregular due to the mechanical action of drilling. Geologists characterize them through a combination of visual inspection, physical testing, and analytical techniques. Cuttings examined under the microscopes are characterized by their grain size and shape, color and texture, mineral composition, cementation and porosity. Roundness, sphericity, aspect ratio, convexity, and solidity are among the shape descriptors.

While cost-effective and accessible, the above approach is labor-intensive due to the frequent manual collection and preparation of cuttings, fragmented nature and collected volume of cuttings, potential contamination, and the need for expert interpretation. Geologists must visually inspect samples, often under microscopes, and integrate findings with log data and regional geological models. Despite advancements in automated mineralogy and machine learning, manual interpretation remains essential, especially in complex geological settings \cite{yamada2024liobia,morton1993development,caja2019image}.

\subsection{Tomographic imaging in subsurface characterization}\label{sec:SubSurfCharCT}
Motivated by reducing the aforementioned labor intensive effort in collection, preparation, and interpretation process, we conceptualize a system for online tomographic imaging of the cuttings as they are being discharged into a collection pit. 

CT imaging has been successfully used to get structural information of materials in a non-destructive fashion.
Its usage to image micro-cores is almost a standard procedure in the energy industry and is referred to as ``digital rock'' 
\cite{Rassenfoss2011DigitalRocks,Madonna2012DRP,Andra2013DRP1,Andra2013DRP2,Berg2017IndustrialDRP,Sadeghnejad2021DRPReview}.
In digital rock, microscopic images obtained from CT scans are combined with numerical simulations to analyze the physical properties of rocks and to infer petrophysical properties of the reservoir. This is limited to small numbers of core samples. If upscaled to the whole well with high throughput, it could help determine geomechanical properties of the subsurface  \cite{Cong2023ThinInterbedded} that could aid well construction and complement well integrity measurements. A similar approach was also investigated for cement-based materials \cite{Chung2019MicroCTCement}.

Incident energy dependency of the attenuation of the object can be used as a spectroscopic method \cite{Sanchez2017EDXASTomo}. When combined with the structural information, this could enable mineral characterization \cite{Wang2020MicroCTGeomet}. X-ray attenuation
depends on the density of the material and its atomic composition, which vary significantly between different rock types. For example, carbonates (e.g., limestone, dolomite) have moderate attenuation due to calcium and magnesium content, shales have higher attenuation if rich in iron or organic matter. Sandstones have lower attenuation unless cemented with heavy minerals, and igneous rocks have high attenuation due to dense minerals like feldspar, pyroxene, and olivine 
\cite{Rahmani2016MultiscaleShaleCuttings}.

To make the setting concrete, we envision the following acquisition scenario. During drilling, rock cuttings are continuously carried to the surface by the circulating drilling mud and separated from the fluid at the shale shaker. Rather than being collected and inspected manually, the separated cuttings are routed to a fixed discharge point---for example, the end of the shaker screen, a conveyor, or a chute---from which they fall freely into a collection pit. 
The fixed imaging gantry is placed to surround this fall path, so that each cutting is projected onto the detector during the short interval in which it traverses the field of view. This compact geometry is chosen deliberately for cost, mechanical simplicity, robustness, and the high throughput required for online operation, in contrast to conventional CT scanners that rotate the source--detector assembly (or rotate the sample on a stage) in order to acquire many views.

A defining feature of this scenario is that the motion of each cutting is neither controlled nor directly measured. Every fragment leaves the discharge point with its own launch velocity, direction, and spin, set by the mechanics of the shaker and by particle-to-particle differences in size, shape, and mass. During its free flight each cutting therefore undergoes (i)~a projectile, approximately parabolic, translation under gravity and (ii)~tumbling about its own axes; both differ from particle to particle and are unknown a priori. Instrumenting the system to sense this per-particle motion directly would defeat the purpose of a low-cost, high-throughput design. Consequently, the rotation and forward movement of each particle must be treated as unknown latent variables and estimated jointly with the reconstruction---precisely the single-source computerized tomography of moving, Gaussian-modeled particles with unknown motion that this paper addresses. Because the discharge point is fixed by the apparatus, the location at which a cutting enters free flight (its ejection point) is approximately known even though its velocity and spin are not. We further assume that the feed to the discharge point is metered so that only a few cuttings are in the field of view at any one time. The number of particles in view is then a design parameter that trades throughput against separability, as in dynamic image analysis of moving particles, whose standard methods assume that the particles in the images are individually separated \cite{ISO13322-2}. We return to both assumptions in the numerical study.

\subsection{Challenges}\label{sec:SubSurfCharChall}
There are hardware and software challenges associated with developing a tomographic imaging system that can characterize the structural and attenuation properties of cuttings.

One is to have a cost-effective acquisition system, e.g., by  considering a gantry with fixed source and detector elements.
Performing tomographic imaging in this setting relies on letting the objects being imaged rotate as they move through the gantry.
Another is to design an associated reconstruction algorithm that is suited to deal with the high throughput of the system while delivering the highest image quality with least computational load for real-time applications.

Tomographic imaging of objects in motion has been a subject of investigation with fixed acquisition geometries and multiple sources \cite{Mudde2008TimeResolvingCT,Graas2024TimeResolvedFluidizedBeds}. The inherent ill-posed nature of tomography problems makes it impossible to completely recover static objects with a fixed single-source acquisition geometry. However, we assume that discharged objects undergo parabolic and rotational motion while they traverse the field of view of the imaging system. The rotational motion, in particular, continuously reorients each particle and thereby supplies the range of viewing directions that a single static source cannot provide on its own, so that the unknown motion is not only a nuisance to be estimated but also the very mechanism that makes reconstruction possible. This assumption turns a single-view problem into one in which the views are acquired over time, at the cost of having to recover the dynamic properties of each object jointly with its reconstruction.
In this study we will focus on the dynamic tomographic reconstruction problem for a given detector geometry and leave the optimal experimental design problem
\cite{fathi2025bi} of using a minimal number of detectors for a future study. 

The associated software challenge is to investigate whether it is possible to image discharged cuttings as they undergo parabolic motion. A motivation for representing the particles as a GMM is that it has the capacity to approximate general objects up to an arbitrary degree of accuracy \cite{zickert2022gaussian}. Another is that the ray transform of this representation is a lower dimensional GMM which can be analytically computed. In this paper we approximate each cutting by a single Gaussian, so that discharge of multiple cuttings constitutes a dynamic GMM with one component per cutting. Characterizing the shape of a drilled cutting as a Gaussian is not typically done in geology, but it provides convenience from a mathematical modeling or image analysis perspective. A single Gaussian captures the position, size, orientation, aspect ratio and integrated attenuation of a cutting, but not finer shape descriptors such as angularity or convexity (Section~\ref{sec:SubSurfChar}). We therefore regard the single-Gaussian model as a feasibility study of the motion-estimation problem: it isolates the question of whether the unknown ballistic and rotational motion of several particles can be recovered from single-view data. Representing each cutting by several Gaussians that share one rigid motion is the natural next step (Section~\ref{sec:conclusion}), but it is not treated here.

\subsection{Related work}
Tomographic imaging of objects that evolve during acquisition has been studied extensively, and a comprehensive survey of dynamic reconstruction methods is given in \cite{PaperF}. A large class of approaches models temporal evolution through image deformations. From this perspective, the proposed GMM-CT framework may be viewed as a finite-dimensional deformable-template model. Specifically, taking a template $f_0$ made up of copies of a canonical template $\rho_0\in X$ for a suitable function space $X$ (see \eqref{eq:ParticleModel}),
 \[ f_0=(\rho_0,\ldots,\rho_0)\in X^N, \] 
a deformation operator 
\[ W_{\theta_t}(f_0) = \sum_n \alpha_n \matOp{A}[\matOp{U}_n,\vecOp{\Theta}_n,\vecOp{\eta}_n](t)\rho_0,\quad \alpha_n\in\R^+, \] 
where $\matOp{A}[\matOp{U}_n,\vecOp{\Theta}_n,\vecOp{\eta}_n]$ is the affine transformation operator with morphology ($\matOp{U}_n$) and motion parameters ($\vecOp{\Theta}_n,\vecOp{\eta}_n$) (see Section \ref{sec:ParamDomFwdOp} for detailed definitions of notations), together with regularization restricting admissible motions to the parametric trajectories introduced in Section~\ref{sec:MotionModel} and particle morphologies to the Gaussian orbit; our formulation falls within the general variational framework of \cite{PaperF}. 

Despite this conceptual connection, existing deformable-image approaches are not designed for the particle setting considered here. Optical-flow-based joint reconstruction \cite{PaperD}, LDDMM-based formulations, metamorphosis models, and learned registration techniques (see \cite{PaperF} and the references therein) operate on dense image sequences defined on fixed spatial grids and represent temporal evolution by a deformation field acting on an image intensity distribution. In the present setting this raises several difficulties: (i) the smoothness regularization of a single deformation field couples nearby particles that move independently; (ii) particles that occlude one another, or that emerge from a common ejection point, are not well described by a diffeomorphic deformation of a single template image; (iii) particles entering and leaving the field of view are difficult to represent without introducing singularities; and (iv) rigid-body rotations induce affine velocity fields, which standard optical-flow regularizers such as the total variation of the flow penalize. Furthermore, these methods are typically demonstrated with several projections per time frame (e.g.\ gated data), whereas here only a single projection is available at each time instant, so all spatial information must be supplied by the motion and object models.

The most closely related work is found in sparse particle-based and measure-theoretic formulations for dynamic inverse problems. Alberti et al.~\cite{PaperA} and Holler et al.~\cite{PaperB} represent moving objects as point masses evolving in phase space and recover them by solving sparse inverse problems over measures. These approaches provide powerful recovery guarantees and stability results, but rely on highly simplified object models. In particular, particles are represented by Dirac masses with linear trajectories, so morphology, orientation, particle extent, and rotational dynamics are excluded from the reconstruction model. Extending these formulations to anisotropic particles undergoing projectile motion and rigid-body rotation would require embedding shape, orientation, and higher-order dynamics into the phase-space representation, substantially increasing both the dimensionality and complexity of the reconstruction problem. Moreover, the existing recovery theory is derived under assumptions that do not directly generalize to such models. Bredies et al.~\cite{PaperE} recover point masses moving along general curves, regularized by the Benamou--Brenier kinetic energy, with a generalized conditional gradient method. This is closest to our setting in that it outputs particles together with their trajectories, but its atoms are again point masses without extent or orientation, and the kinetic-energy regularizer favors constant-velocity paths rather than ballistic ones.

Schmitzer et al.~\cite{PaperC} instead reconstruct a time-dependent density on a grid, coupled across time by a continuity equation and regularized by the Benamou--Brenier action. Such dynamic optimal-transport formulations model conservation and transport naturally and lead to convex variational problems. However, they reconstruct density fields rather than discrete particles and therefore do not directly recover object-level quantities such as particle trajectories, orientations, or morphologies. Particle identities are not explicitly preserved, and an additional model-building stage would be required to extract the individual particle parameters that constitute the primary output of our framework.

Variational methods that jointly estimate images and motion fields are also related. Burger et al.~\cite{PaperD}, for example, couple tomographic reconstruction with optical-flow estimation. While such approaches can in principle accommodate tomographic forward operators, they remain high-dimensional, requiring reconstruction of a complete spatiotemporal image sequence before object extraction can be performed. This introduces two additional limitations: first, the motion field is estimated indirectly through image reconstruction error rather than directly through particle parameters; second, model construction remains a separate inverse problem. In contrast, the motion parameters in GMM-CT constitute the unknowns themselves. 

In summary, existing dynamic tomography methods typically reconstruct either (i) temporally evolving voxelized images, (ii) transported density measures, or (iii) moving point masses. To the best of our knowledge, no existing framework simultaneously accommodates anisotropic particle morphology, independent rigid-body rotations, projectile motion, object-level parameter estimation, and single-view temporal acquisition. GMM-CT addresses this gap by representing each particle as an affine transformation of a canonical Gaussian template evolving under a parameterized rigid-body motion model. This representation yields closed-form projection formulas, analytical gradients, and a finite-dimensional optimization problem in which trajectory estimation, morphology recovery, and model construction are integrated within a single reconstruction framework.

\section{The mathematical setting}\label{Sect:Setting}
Consider performing tomography on an object that undergoes motion. 
The starting point for the mathematical formalization is to represent the object with a real valued function $f \colon \R^d \times [0,t_{\max}] \to \R$. Data is similarly given as (noisy) samples of a function $g \colon \Mlines \times [0,t_{\max}] \to \R$ where $\Mlines$ is some manifold of lines in $\R^d$ and 
\begin{align}
\label{eq:ray_sinogram}
  g(\ell,t) := \RayTran\bigl(f(\,\cdot\,,t)\bigr)(\ell) := \int_{\vecOp{x}\in\ell} f(\vecOp{x},t)\, d\vecOp{x} \quad\text{for $\ell \in \Mlines$,}
\end{align}
with $\RayTran$ denoting the ray transform ($d\vecOp{x}$ is the 1-dimensional Lebesgue measure on the line $\ell$).
The ray transform $\RayTran$---also referred to as the X-ray transform---is a classical object in integral geometry and the mathematical model underlying transmission (X-ray) computerized tomography \cite{natterer2001mathematics,kak2001principles,helgason1999radon}. It maps a function to its integrals along lines; in two dimensions ($d=2$) it coincides with the Radon transform, whereas for $d \geq 3$ the two differ, the Radon transform integrating over hyperplanes rather than lines \cite{natterer2001mathematics,helgason1999radon}. Its analytic and stability properties, together with inversion formulae for the classical parallel-, fan-, and cone-beam acquisition geometries, are well documented \cite{natterer2001mathematics,herman2009fundamentals,natterer2001mathematical}. In the present work we parameterize the lines by source--detector pairs, corresponding to the divergent-beam geometry of our fixed-source setup (Section~\ref{sec:ParamFwdModel}); the problem we consider departs from these classical settings in that the object moves during acquisition and the motion itself is unknown.
\begin{definition}[Spatiotemporal imaging]
\label{prob:InvProb1}
Recover $t \mapsto f(\,\cdot\,,t)$ on $\R^d$ from time-series $t \mapsto g(\,\cdot\,,t)$ on $\Mlines$ where $g \colon \Mlines \times [0,t_{\max}] \to \R$ is given as in \eqref{eq:ray_sinogram}.
\end{definition}

The data collected by the experimental source-detector setup corresponds to sampling on the manifold $\Mlines$ of lines in $\R^d$.
Our experimental setup has a \emph{fixed} point-like source (source position) with fixed detectors that are represented as a planar rectangular region contained in a $(d-1)$-dimensional plane (detector plane). We assume the source is centered w.r.t.\@ the detector, i.e., projecting the source position onto the detector plane coincides with the midpoint of the rectangular region that makes up the detector.

The above inverse problem is clearly under-determined as it involves recovering a real-valued function on $\R^d$ from data given as a real-valued function on $\Mlines$, which is the family of lines intersecting a $(d-1)$-dimensional rectangular region (detector) and that go through a point (source position).
We therefore need to make use of further assumptions.

In the application we have in mind, we can assume that the object consists of $N$ rigid isolated particles that move independently of each other. 
Hence, the temporal variation of the object corresponds to $N$ independent motion trajectories, one for each particle. 
Since each particle undergoes a time-dependent rigid body motion, we can re-interpret this as viewing each particle from directions that vary with time, i.e., as tomographic imaging with an unknown acquisition geometry.

To formalize the above, the function $f \colon \R^d\times[0, t_\textnormal{max}] \to \R$ representing the moving object can be written as
\begin{align}
\label{eq:f}
f(\vecOp{x},t) = \sum_{n=1}^N f_n(\vecOp{x},t) \quad\text{for $(\vecOp{x},t) \in \R^d\times [0, t_{\max}]$}    
\end{align}
where $f_n \colon \R^d \to \R$ represents the $n$-th particle.
The assumption that each particle undergoes a time-dependent rigid body motion is then encoded as
\begin{align}
      f_n(\vecOp{x},t) := \bigl( \matOp{M}_n(t).\rho_n \bigr)(\vecOp{x}),
  \label{eq:f_n}
\end{align}  
where $\rho_n \colon \R^d \to \R$ represents the particle structure and the particle-specific rigid body motion is represented by the action $(\matOp{M},\rho) \mapsto \matOp{M}.\rho$ of $\matOp{M} \in \SE(d)$ on the set of $\R$-valued functions on $\R^d$.
The latter is given as 
\begin{equation} 
(\matOp{M}.\rho)(x) := \rho(\matOp{M}.x) =
   \rho\bigl( \matOp{R}^{-1}(\vecOp{x}-\vecOp{\tau}) \bigr)
   \quad\text{for $\matOp{M}=(\matOp{R},\vecOp{\tau}) \in \SE(d) = \SO(d) \ltimes \R^d$.}
\end{equation}
Inserting \eqref{eq:f_n} into \eqref{eq:f} yields the following spatiotemporal model for the object:
\begin{align}
\label{eq:f_generative_model}
f(\vecOp{x},t) 
= \sum_{n=1}^N \bigr(\matOp{M}_n(t).\rho_n\bigr)(\vecOp{x}) 
= \sum_{n=1}^N \rho_n\Bigl(\matOp{R}_n(t)^{-1}\bigl(\vecOp{x}-\vecOp{\tau}_n(t)\bigr)\Bigr),
\end{align}
where each curve $t \mapsto \matOp{M}_n(t) = \bigl(\matOp{R}_n(t),\vecOp{\tau}_n(t)\bigr) \in \SE(d)$ models the motion of the $n$-th particle.
Inserting the motion model \eqref{eq:f_generative_model} into \eqref{eq:ray_sinogram} yields the following model for the data:
\begin{equation}
g(\ell,t) 
 := \RayTran\biggl(\sum_{n=1}^N \matOp{M}_n(t).\rho_n\biggr)(\ell) 
 = \sum_{n=1}^N \RayTran\Bigl(\matOp{M}_n(t).\rho_n\Bigr)(\ell),  
\label{eq:data_model_motion}
\end{equation}
where the last equality follows from the linearity of the ray transform. 

As a generative model for the morphology component of the particles, we assume they are all dilations of a \emph{single} template $\rho_0$, i.e. morphology of particles are in the $\GL(d)$ orbit of $\rho_0$:
\begin{align}
    \GL(d) \rho_0 := \bigl\{\matOp{U}_n^{-1}.\rho_0\, |\, \matOp{U}_n\in \GL(d) \bigr\}
    = \bigl\{\vecOp{x} \mapsto \rho_0(\matOp{U}_n\vecOp{x} )\, |\, \matOp{U}_n\in \GL(d) \bigr\},
\end{align}
where $\GL(d)$ is the general linear group on $\R^d$. 
\begin{remark}\label{rmrk:U_SNND}
    Note that $\matOp{U}^{-1}.\rho_0 \colon \R^d \to \R$ depends only on the symmetric non-negative definite matrix $\matOp{U}^\top \matOp{U}$, so we can without loss of generality assume that $\matOp{U}\in\R^{d\times d}$ is upper triangular  with non-negative diagonal entries, $\matOp{U}\in \R^{d_\matOp{U}} \subset \R^{d\times d}$, where $d_\matOp{U} = d(d+1)/2$.
\end{remark}
For example, if morphology component is given by an isotropic Gaussian $\rho_0(\vecOp{x}) = e^{-\vecOp{x}^{\!\top}\! \vecOp{x}}$, then
\begin{equation}\label{eq:ParticleModel}
  \rho_n(\vecOp{x}) := \alpha_n (\matOp{U}_n^{-1}.\rho_0)(\vecOp{x})
  =
  \alpha_n \rho_0(\matOp{U}_n\vecOp{x}),
\end{equation}
where $\alpha_n > 0$ is an additional attenuation coefficient that we will frequently include when referencing morphology.

Next, combination of \eqref{eq:f}, \eqref{eq:f_n}, and \eqref{eq:ParticleModel} constitutes a GMM for $\vecOp{x} \mapsto f(\vecOp{x},t)$. A motivation for this representation is that it has the capacity to approximate general objects up to an arbitrary degree of accuracy \cite{zickert2022gaussian}. Another motivation is that this representation is mapped to a lower dimensional GMM by the ray transform that can be analytically computed, as will be shown in Proposition~\ref{prop: ray of stationary gaussian}.

Substituting \eqref{eq:ParticleModel} into \eqref{eq:f_generative_model} and \eqref{eq:data_model_motion} yields the following generative model for the object and data, respectively,
\begin{align}
\label{eq:f_generative_model_GMM}
f(\vecOp{x},t) = \sum_{n=1}^N \alpha_n \bigl(\matOp{M}_n(t).\matOp{U}_n^{-1}.\rho_0 \bigr)(\vecOp{x}) 
\quad\text{for $\vecOp{x} \in \R^d$ and $t \in [0, t_{\max}]$.}
\end{align}
and
\begin{equation}
g(\ell,t) 
:= \RayTran\biggl(\sum_{n=1}^N \alpha_n \matOp{M}_n(t).\matOp{U}_n^{-1}.\rho_0\biggr)(\ell) 
= \sum_{n=1}^N \alpha_n \RayTran\Bigl(\matOp{M}_n(t).\matOp{U}_n^{-1}.\rho_0\Bigr)(\ell)
  \quad\text{$\ell \in \Mlines$, $t \in [0, t_{\max}]$.}
\label{eq:data_model_GMM}
\end{equation}

\emph{The inverse problem is now to determine $N$ morphologies $(\alpha_n, \matOp{U}_n) \in \R_+ \times \R^{d_\matOp{U}}$ and $N$ curves $t\mapsto \matOp{M}_n(t)$ from a single time series $t\mapsto g(\,\cdot\,,t)$ of ray transform data on $\Mlines$.}
Each $\matOp{U}_n$ generates the structures for the particles through \eqref{eq:ParticleModel} and $t\mapsto \matOp{M}_n(t)$ is a nuisance parameter representing the particle specific motions. We expect that this inverse problem is better posed than the original one stated after \eqref{eq:ray_sinogram}.

We next introduce computationally feasible parameterizations of the motion model (Section~\ref{sec:MotionModel}) and the ray transform (Section~\ref{sec:ParamFwdModel}).

\section{Parameterizing particle motions}\label{sec:MotionModel}
An element $\matOp{M}\in \SE(d)$ can be parameterized as a combination of translation and rotation which in matrix form is expressible as
\begin{align}
    \matOp{M} = \begin{bmatrix} 
        \matOp{R}_{\vecOp{\Omega}} & \vecOp{\tau} \\[0.5em]
        \vecOp{0}^{\top} & 1
    \end{bmatrix}.
\end{align}
In the above, $\vecOp{\tau}\in \R^d$ is the  translation, $\vecOp{0}\in \R^d$ is the zero vector, and $\matOp{R}_{\vecOp{\Omega}} \in \SO(d)$ is parameterized by $\vecOp{\Omega} \in \R^{d_\matOp{R}}$ with $d_\matOp{R}:=d(d-1)/2$. 
One way to parameterize $\matOp{R}_{\vecOp{\Omega}}$ is as the product
\begin{align}
    \matOp{R}_{\vecOp{\Omega}} := \prod_{k=1}^{d_\matOp{R}} \matOp{R}_k(\omega_k), \quad \text{for $\omega_k\in\R$}
\end{align}
where
\[
\matOp{R}_k(\omega_k) :=\!\!\!\!\!\!\!\!\!\!
\begin{bNiceMatrix}[
  first-row,
  code-for-first-row = \scriptstyle,
  last-row,  
  code-for-last-row = \scriptstyle,
  first-col, 
  code-for-first-col = \scriptstyle,  
  last-col, 
  code-for-last-col = \scriptstyle,
  margin,
  cell-space-limits=3pt]
  &  d_1-1 & &  d_2-d_1-1 & &  &
  \\
  d_1-1 & \Block[borders={bottom,top,right,left,tikz=dotted}]{1-1}{} \matOp{I}_1 & & & & &
  \\
  &  &  \cos\omega_k & & -\sin\omega_k & &
  \\
  &  &  & \Block[borders={bottom,top,right,left,tikz=dotted}]{1-1}{} \matOp{I}_2 & & & d_2-d_1-1
  \\
  & & \sin\omega_k &  & \cos\omega_k & &
  \\  
  &  &  &  &  & \Block[borders={bottom,top,right,left,tikz=dotted}]{1-1}{} \matOp{I}_3 & d_R-d_2
  \\
  &  &  &  &  &  d_R-d_2&
\end{bNiceMatrix}
\]
for $k = (d_1-1)(d-d_1) + (d_2-d_1)$.
Here $d_1<d_2\le d$ are the row and column indices for specific elements in $\matOp{R}_k(\omega_k)$ and 
\begin{itemize}
\item $\matOp{I}_1$ is the $(d_1-1)\times(d_1-1)$ identity matrix,
\item $\matOp{I}_2$ is the $(d_2-d_1-1)\times (d_2-d_1-1)$ identity matrix,
\item $\matOp{I}_3$ is the $\left(d_\matOp{R}-d_2\right)\times \left(d_\matOp{R}-d_2\right)$ identity matrix,
\end{itemize}
Besides the elements shown, all other elements of $\matOp{R}_k(\omega_k)$ are zero. 

One can now model the curve $\matOp{M}_n(t)\in \SE(d)$ by a time-dependent translation $\vecOp{\tau}_n(t) \in \R^d$ and rotation $\matOp{R}_{\vecOp{\Omega}_n(t)} \in \SO(d)$ given by
\begin{align}
  \vecOp{\tau}_n(t)&= \vecOp{\mu}_n + t\, \vecOp{v}_n + \frac{t^2}{2} \vecOp{a}_n \quad\text{for some $\vecOp{\mu}_n, \vecOp{v}_n, \vecOp{a}_n \in \R^d$}
  \label{eq:TranslationModel}
  \\
  \vecOp{\Omega}_n(t) &:= t \vecOp{\Theta}_n
  , \quad\text{for $\vecOp{\Theta}_n \in \R^{d_\matOp{R}}$}
  \label{eq:RotationModel}   
\end{align}
where
\[
  \matOp{R}_{\vecOp{\Omega}_n(t)} := \prod_{k=1}^{d_\matOp{R}}\begin{bmatrix}
 \matOp{I}_1 &     &  &  &  \\
  & \cos(t\theta_{n,k}) &  & -\sin(t\theta_{n,k}) &  \\
  &  & \matOp{I}_2 &  &  \\
  & \sin(t\theta_{n,k}) &  & \cos(t\theta_{n,k}) & \\
  &  &  &  & \matOp{I}_3
  \end{bmatrix}
\]
and with $\theta_{n,k}$ denoting the $k$-th component of $\vecOp{\Theta}_n$.

\section{Parameterizing the domain of the forward operator}
\label{sec:ParamDomFwdOp}
The particle motion parameterization in Section~\ref{sec:MotionModel} allows us to rewrite the data model in \eqref{eq:data_model_GMM}. Let us employ the following notation:
\begin{itemize}
\item $(\alpha,\matOp{U}) \in \R_+ \times \R^{d_\matOp{U}}$ encodes the amplitude/attenuation and shape/morphology of the particle.
\item $\vecOp{\Theta} \in \R^{d_\matOp{R}}$ encodes the rotational motion $t\mapsto\matOp{R}_{t\vecOp{\Theta}}\in \SO(d)$ of the particle.
\item $\vecOp{\eta}:=(\vecOp{\mu},\vecOp{v}, \vecOp{a})\in \R^d \times \R^d \times \R^d$ encodes the projectile motion curve $\vecOp{C}[\vecOp{\eta}] \colon [0,t_{\max}] \to \R^d$ of the particle defined as
\[
\vecOp{C}[\vecOp{\eta}](t):= \vecOp{\mu} + t \vecOp{v} + \frac{t^2}{2}\vecOp{a}.
\]
Here, $\vecOp{\Theta}$ is the angular velocity, $\vecOp{\mu}$ is the spatial location at time $t=0$, $\vecOp{v}$ is the initial velocity, and $\vecOp{a}$ is the initial acceleration.
\end{itemize}
Let $\Aff(d) = \GL(d)\ltimes \R^d$ denote the affine group, and 
$\matOp{A}[\matOp{U},\vecOp{\Theta},\vecOp{\eta}] (t) = (\matOp{R}_{t \vecOp{\Theta}} \matOp{U}^{-1}, \vecOp{C}[\vecOp{\eta}](t))\in \Aff(d)$
capture the morphology and motion of a particle. 
Denoting the $n$-th particle by
\begin{align}\label{eq:NthParticle}
\matOp{M}_n(t).\rho_n  = \matOp{A}[\matOp{I},\vecOp{\Theta}_n,\vecOp{\eta}_n](t). \rho_n = \alpha_n \matOp{A}[\matOp{U}_n,\vecOp{\Theta}_n,\vecOp{\eta}_n](t). \rho_0,
\end{align}
we see that the morphological components of the particles undergoing motion are in the affine group orbit of $\rho_0$: 
\begin{align}
\Aff(d) \rho_0 := \bigl\{\matOp{A}.\rho_0\, |\, \matOp{A}\in \Aff(d) \bigr\} = \bigl\{\vecOp{x}\mapsto \rho_0\bigl(\matOp{U}^{-1}(\vecOp{x}-\vecOp{\mu})\bigr)\, |\, \matOp{A}=(\matOp{U},\vecOp{\mu})\in \Aff(d)\bigr\}.
\end{align}

Substituting \eqref{eq:NthParticle} into \eqref{eq:data_model_GMM} and using the above notation allows us to express the inverse problem stated after \eqref{eq:data_model_GMM} as the task to recover 
\[ \matOp{\Gamma} :=(\vecOp{\gamma}_1,\ldots,\vecOp{\gamma}_N)
\quad\text{with}\quad 
\vecOp{\gamma}_n:=\bigl((\alpha_n,\matOp{U}_n),\vecOp{\Theta}_n,\vecOp{\eta}_n\bigr)\in (\R_+ \times \R^{d_\matOp{U}}) \times \R^{d_\matOp{R}} \times (\R^d  \times \R^d \times \R^d)\]
from data 
\begin{equation}\label{eq:data_model2}
  g(\ell,t) = F[\matOp{\Gamma}](\ell,t)
  \quad\text{for $(\ell,t) \in \Mlines \times [0,t_{\max}]$} 
\end{equation}
where $F[\matOp{\Gamma}] \colon \Mlines \times [0,t_{\max}] \to \R$ (forward operator) is given as   
\begin{align}
\label{eq:forward_operator}
    F[\matOp{\Gamma}](\ell,t) := 
    \sum_{n=1}^N \alpha_n \RayTran\bigl(\matOp{A}[\matOp{U}_n,\vecOp{\Theta}_n, \vecOp{\eta}_n](t). \rho_0\bigr)(\ell).
\end{align}

\begin{remark}[Parameterization invariances and identifiability]\label{rem:invariances}
The forward operator \eqref{eq:forward_operator} is invariant under two gauge symmetries of the particle parameters $\vecOp{\gamma}_n = ((\alpha_n,\matOp{U}_n),\vecOp{\Theta}_n,\vecOp{\eta}_n)$ so distinct parameters can generate identical data.
\begin{enumerate}[(i)]
\item \emph{Template symmetry (left-orthogonal gauge).} As $\rho_0$ is isotropic, $\rho_0(\matOp{O}\vecOp{y}) = \rho_0(\vecOp{y})$ for every $\matOp{O}\in \operatorname{O}(d)$, so $\matOp{U}_n$ enters only through $\matOp{U}_n^\top\matOp{U}_n$ and $\matOp{U}_n \mapsto \matOp{O}\matOp{U}_n$ leaves $F[\matOp{\Gamma}]$ unchanged. This is the ambiguity already noted in Remark~\ref{rmrk:U_SNND}; we fix it by taking $\matOp{U}_n$ upper triangular with non-negative diagonal (a Cholesky gauge), so that only the shape $\matOp{U}_n^\top\matOp{U}_n$ is regarded as identifiable.
\item \emph{Static-rotation gauge (morphology--rotation coupling).} The $n$-th particle at time $t$ is the Gaussian $\vecOp{x}\mapsto \alpha_n\rho_0\bigl(\matOp{U}_n \matOp{R}_{t\vecOp{\Theta}_n}^{-1}(\vecOp{x}-\vecOp{C}[\vecOp{\eta}_n](t))\bigr)$, which depends on morphology and rotational motion only through the product $\matOp{U}_n\matOp{R}_{t\vecOp{\Theta}_n}^{-1}$. For any constant $\matOp{Q}\in\SO(d)$, $\matOp{U}_n\matOp{R}_{t\vecOp{\Theta}_n}^{-1} = (\matOp{U}_n\matOp{Q})(\matOp{R}_{t\vecOp{\Theta}_n}\matOp{Q})^{-1}$, so applying a static rotation $\matOp{Q}$ to the morphology and its inverse to the motion leaves the data unchanged. The constant-angular-velocity model \eqref{eq:RotationModel}, which pins the initial orientation to $\matOp{R}_{\vecOp{0}}=\matOp{I}$, fixes this gauge (it forces $\matOp{Q}=\matOp{I}$) and thereby defines $\matOp{U}_n$ in the body frame of the particle at $t=0$; equivalently, $\matOp{U}_n^\top\matOp{U}_n$ is the shape at $t=0$.
\end{enumerate}
Two consequences for identifiability follow. First, the data determine, for each $t$, only the gauge-invariant quantities: the center $\vecOp{C}[\vecOp{\eta}_n](t)$ and the lab-frame precision matrix $\matOp{R}_{t\vecOp{\Theta}_n}(\matOp{U}_n^\top\matOp{U}_n)\matOp{R}_{t\vecOp{\Theta}_n}^{-1}$; the morphology $\matOp{U}_n$ and the orientation are pinned down only after the gauge choices above. Second, since the rotational motion enters the observation solely through its action on an anisotropic shape, $\vecOp{\Theta}_n$ is identifiable only insofar as $\matOp{U}_n^\top\matOp{U}_n$ is anisotropic: for an isotropic particle $\matOp{R}_{t\vecOp{\Theta}_n}(\matOp{U}_n^\top\matOp{U}_n)\matOp{R}_{t\vecOp{\Theta}_n}^{-1}=\matOp{U}_n^\top\matOp{U}_n$ for every $\vecOp{\Theta}_n$, so the angular velocity is unobservable, and more generally the components of $\vecOp{\Theta}_n$ lying in the stabilizer of $\matOp{U}_n^\top\matOp{U}_n$ (e.g.\ rotations about a symmetry axis of a spheroidal particle) leave the data unchanged.

For reconstruction, the two continuous gauges (i)--(ii) are flat directions of the data-misfit; fixing them via the upper-triangular convention for $\matOp{U}_n$ and $\matOp{R}_{\vecOp{0}}=\matOp{I}$ removes these flat directions and renders the finite-dimensional problem \eqref{eq:data_model2} locally well-posed (isolated minimizers) away from the degenerate near-isotropic configurations, where the recovered $\vecOp{\Theta}_n$ is inevitably ill-conditioned. Reconstruction quality is therefore best assessed on the gauge-invariant quantities (the center trajectory $\vecOp{C}[\vecOp{\eta}_n]$ and the shape evolution) or after aligning estimates modulo the stabilizer. This also motivates the decoupled scheme \eqref{eq:ParticleTrajec}--\eqref{eq:ParticleMorpho}: the center trajectory is gauge-invariant and is estimated first, whereas the gauge-entangled morphology and rotation are recovered jointly from the full projection data.
\end{remark}

To summarize, the original inverse problem stated right after \eqref{eq:ray_sinogram} was to recover the time-dependent structure of the object from a time series of single projections, which is clearly severely ill-posed. To address the latter, we assume the object consists of particles that move independently from each other as expressed in \eqref{eq:f_n}.
The corresponding inverse problem, which is stated after \eqref{eq:data_model_GMM}, is still under-determined as it involves recovering particle structures (which are elements in an infinite-dimensional space) and their motions (which are encoded as an element in an infinite-dimensional space).
For this, we use a generative model for the particle structures \eqref{eq:ParticleModel} and a parabolic motion model \eqref{eq:TranslationModel} with rotations \eqref{eq:RotationModel}.
All of these are specified by finitely many parameters, so the corresponding inverse problem considering \eqref{eq:data_model2} is not necessarily under-determined. A drawback is that the corresponding forward operator, which was linear in \eqref{eq:ray_sinogram}, is now non-linear.

The above dealt with introducing coordinates on the domain of the forward operator, i.e., to parameterize time-dependent, $\R$-valued functions on $\R^d$.
The final step necessary for implementation is to introduce coordinates on the range of the forward operator, which are time-dependent, $\R$-valued functions on $\R^d $ (time-dependent sinograms).
This is done in Section~\ref{sec:ParamFwdModel}.

\section{Parameterizing the range of the forward operator}\label{sec:ParamFwdModel}
The starting point is to introduce coordinates on the manifold of lines $\Mlines$.
A line $\ell \in \Mlines$ can be parameterized by a source-detector pair of points $(\vecOp{s},\vecOp{r}) \in \R^d \times \R^d$ as $\xi \mapsto \vecOp{s} + (\widehat{\vecOp{r} - \vecOp{s}})\xi$. 
Then the ray transform in \eqref{eq:ray_sinogram} can be expressed in these coordinates as follows.
\begin{definition}[Ray transform]\label{def: ray transform}
    The ray transform of a function $\rho \colon \R^d \to \R$ evaluated at a line given by the source--detector pair $(\vecOp{s},\vecOp{r})\in\R^d\times\R^d$ is the line integral of $\rho$ over that line: 
    \begin{align}
        \RayTran\bigl(f(\,\cdot\,,t)\bigr)[\vecOp{s},\vecOp{r}] := \int_\R f\big(\vecOp{s} + (\widehat{\vecOp{r} - \vecOp{s}})\xi,t\big)d\xi,\label{eq: ray transform}
    \end{align}
    where $\widehat{\vecOp{x}} := \vecOp{x}/\|\vecOp{x}\|$ for any $\vecOp{x}\in\R^d\backslash\{0\}$, and $\|\,\cdot\,\|$ is the $d$-dimensional Euclidean norm.
\end{definition}

\begin{proposition}[Ray transform of an anisotropic Gaussian]\label{prop: ray of stationary gaussian}
    Let $\matOp{A}=(\matOp{U}^{-1},\vecOp{\mu})\in\Aff(d)$. 
    Then the ray transform of the anisotropic Gaussian function $\matOp{A}.\rho_0$ is
    \be
        \RayTran(\matOp{A}.\rho_0)[\vecOp{s},\vecOp{r}] = \dfrac{\sqrt{\pi}}{\bigl\|{\matOp{U}(\widehat{\vecOp{r} - \vecOp{s}})}\bigr\|}\exp\biggl(\dfrac{\bigl\langle \matOp{U}(\vecOp{r} - \vecOp{s}), \matOp{U}(\vecOp{s} - \vecOp{\mu})\bigr\rangle^2}{\bigl\|\matOp{U}(\vecOp{r} - \vecOp{s})\bigr\|^2} - \bigl\|\matOp{U}(\vecOp{s} - \vecOp{\mu})\bigr\|^2\biggr)
        \label{eq: ray of Gaussian}
    \ee
    for $(\vecOp{s},\vecOp{r})\in\R^d\times\R^d$, where $(\matOp{A}.\rho)(x) := \rho(\matOp{A}.x) = \rho(\matOp{U}(\vecOp{x} - \vecOp{\mu}))$.
\end{proposition}

\begin{proof}
    Write $\vecOp{\delta} = \vecOp{r} - \vecOp{s}$ for brevity. Then
    \begin{align*}
        \RayTran( \matOp{A}.\rho_0)[\vecOp{s},\vecOp{r}] 
        &= \int_\R (\matOp{A}.\rho_0)\big(\vecOp{s} + \xi\widehat{\vecOp{\delta}}\big)d\xi
        \\
        &= \int_\R\exp\Bigl(-\bigl\|\matOp{U}(\vecOp{s} + \xi\widehat{\vecOp{\delta}} - \vecOp{\mu})\bigr\|^2\Bigr)d\xi
        \\
        &= \int_\R\exp\Bigl(-\bigl\|\xi\matOp{U}\widehat{\vecOp{\delta}} + \matOp{U}(\vecOp{s} - \vecOp{\mu})\bigr\|^2\Bigr)d\xi.
    \end{align*}
    \begin{align*}
    \RayTran( \matOp{A}.\rho_0)[\vecOp{s},\vecOp{r}]
        &= \int_\R\exp\Bigl(-\|\matOp{U}\widehat{\vecOp{\delta}}\|^2\xi^2 - 2\langle\matOp{U}\widehat{\vecOp{\delta}}, \matOp{U}(\vecOp{s} - \vecOp{\mu})\rangle\xi - \|\matOp{U}(\vecOp{s} - \vecOp{\mu})\|^2 \Bigr)d\xi
        \\
        &= \exp\Bigl(- \|\matOp{U}(\vecOp{s} - \vecOp{\mu})\|^2\Bigr)\int_\mathbb{R}\exp\Bigl(-\|\matOp{U}\widehat{\vecOp{\delta}}\|^2\xi^2 - 2\langle\matOp{U}\widehat{\vecOp{\delta}}, \matOp{U}(\vecOp{s} - \vecOp{\mu})\rangle\xi\Bigr)d\xi.
    \end{align*}
    Applying the identity $\int_\R e^{-at^2 + bt}dt = \sqrt{\pi / a}\, e^{b^2/4a}$ (valid for $a > 0$) with $a = \|\matOp{U}\widehat{\vecOp{\delta}}\|^2$ and $b = -2\langle\matOp{U}\widehat{\vecOp{\delta}}, \matOp{U}(\vecOp{s} - \vecOp{\mu})\rangle$:
    \begin{align*}
    \RayTran( \matOp{A}.\rho_0)[\vecOp{s},\vecOp{r}]
        = \exp\Bigl(-\|\matOp{U}(\vecOp{s} - \vecOp{\mu})\|^2\Bigr)\cdot\dfrac{\sqrt{\pi}}{\|\matOp{U}\widehat{\vecOp{\delta}}\|}\exp\biggl(\dfrac{\bigl\langle\matOp{U}\widehat{\vecOp{\delta}}, \matOp{U}(\vecOp{s} - \vecOp{\mu})\bigr\rangle^2}{\|\matOp{U}\widehat{\vecOp{\delta}}\|^2}\biggr).
    \end{align*}
    Since $\langle \matOp{U}\widehat{\vecOp{\delta}}, \matOp{U}(\vecOp{s} - \vecOp{\mu}) \rangle / \|\matOp{U}\widehat{\delta}\|^2 = \langle \matOp{U}\vecOp{\delta}, \matOp{U}(\vecOp{s} - \vecOp{\mu}) \rangle / \|\matOp{U}\vecOp{\delta}\|^2$, substituting $\vecOp{\delta} = \vecOp{r} - \vecOp{s}$ back gives \eqref{eq: ray of Gaussian}.    
\end{proof}

\begin{remark}[Beyond the Gaussian template]\label{rem:general_template}
The closed form \eqref{eq: ray of Gaussian} is a convenience rather than a structural requirement of the method, which relies on the Gaussian template only through two properties. First, since every particle is an affine image of a \emph{single} template $\rho_0$ and the ray transform intertwines with the affine group, the forward operator \eqref{eq:forward_operator} depends on $\rho_0$ solely through the projection of that one template, reparameterized per line. It therefore suffices to know, or to precompute once, the ray transform of $\rho_0$: for a radially symmetric template this is a one-dimensional radial profile (the Abel transform of the radial profile \cite{bracewell2000fourier}) that can be stored as a lookup table and evaluated per ray in $\mathcal{O}(1)$ by interpolation, while anisotropic templates require a low-dimensional table indexed additionally by orientation. Consequently, any template admitting a closed-form or cheaply tabulated line integral fits the framework unchanged; standard choices in tomography include Kaiser--Bessel ``blob'' functions \cite{lewitt1992alternatives,matej1996practical} and B-spline/box-spline elements \cite{entezari2012box}, both of which yield analytic line integrals, as do generalized Gaussians and uniform ellipsoids. Because the per-line map is a smooth reparameterization of a fixed template projection, exact (or lookup-table) gradients remain available, so the optimization of Section~\ref{sec:IP} is unaffected. Second, the trajectory stage uses that the projection is unimodal with its mode at the projected center (Proposition~\ref{prop: mode of ray of ansitropic gaussian}); this holds exactly for any radially symmetric template whose profile decreases away from its center, approximately for anisotropic templates (cf.\ Remark~\ref{rem:anisotropic_mode_shift}), and for a general template the mode can be replaced by the projected centroid, which is exactly affine-covariant in parallel-beam geometry (and approximately so in divergent-beam geometry, where the projection is perspective rather than affine) and thus preserves the two-stage decoupling. The Gaussian is singled out because it simultaneously offers density in object space \cite{zickert2022gaussian}, a closed-form affine-covariant projection, and unimodality; for a different template family one should retain a family that is likewise dense (e.g.\ blobs or splines). The main additional cost of a general anisotropic template is the higher-dimensional projection table, and the exact mode-localization results then hold only approximately.
\end{remark}

The identity in \eqref{eq: ray of Gaussian} allows us to express the forward operator in \eqref{eq:forward_operator} when the lines in the sinogram are parameterized by source--detector pairs:
\begin{align}
\label{eq:param_fwd_opt}
F[\matOp{\Gamma}](\vecOp{s},\vecOp{r}, t)
& := \sum_{n=1}^N \alpha_n \RayTran\Bigl(\matOp{A}[\matOp{U}_n,\vecOp{\Theta}_n, \vecOp{\eta}_n](t).\rho_0 \Bigr)[\vecOp{s},\vecOp{r}].
\end{align}

In an experimental setup, we only measure finitely many lines at finitely many time samples. A straightforward formalization of this setting is to consider finitely many triplets indexed by the multi-index $\vecOp{m}\in\mathbb{I}$, where $\mathbb{I} = \{1, \ldots, M_s\}\times\{1, \ldots, M_r\}\times\{1, \ldots, M_t\}$:
\begin{equation}\label{eq:MfldLines}
    \DynLines_0 = \{(\vecOp{s}_{m_s},\vecOp{r}_{m_r}, t_{m_t}) \}_{\vecOp{m}\in\mathbb{I}} \subset \DynLines := \Mlines \times [0, t_{\max}], 
\end{equation}
each encoding the line between a source--detector pair $(\vecOp{s}_m,\vecOp{r}_m)$ at some time point $t_m$.

Recall from Section~\ref{sec:ParamDomFwdOp} that the domain of $F$ consists of elements $\matOp{\Gamma} :=(\vecOp{\gamma}_1,\ldots,\vecOp{\gamma}_N)$ of the form
\begin{equation}\label{eq:SmallGamma}
\vecOp{\gamma}_n:=\bigl((\alpha_n,\matOp{U}_n),\vecOp{\Theta}_n,\vecOp{\eta}_n\bigr)\in (\R_+ \times \R^{d_\matOp{U}}) \times \R^{d_\matOp{R}} \times (\R^d  \times \R^d \times \R^d) = \mathbb{D}.
\end{equation}
Each element in $\Mdomain$ specifies a particle morphology and dynamics. Note that, by Remark~\ref{rmrk:U_SNND} the dimension of $\matOp{U_n}$ is $d(d+1)/2$, so the dimension of $\Mdomain$ is $1+d(d+1)/2+d+d(d-1)/2+d+d = d^2+3d+1$. Consequently, the dimension of the domain is $(d_\matOp{U}+ d_\matOp{R}+ 3d + 1)\times N = (d^2+3d+1)\times N$, with $N$ denoting the number of particles that make up the object. On the other hand the dimension of the measured data is $M$.
For a fixed number of sources, $M_{\vecOp{s}}$, and detector elements $M_{\vecOp{r}}$, one has $M = M_{\vecOp{s}} \times M_{\vecOp{r}} \times M_t$, where $M_t$ is the number of time samples. As a result, $M_t$ should be chosen to at least satisfy 
\begin{align}\label{eq:Criteria}
    M \ge (d^2+3d+1)\times N \implies
    M_{\vecOp{s}} \times M_{\vecOp{r}} \times M_t \ge (d^2+3d+1) \times N
\end{align}
Note however that the above inequality does not guarantee that the digitized ray transform in \eqref{eq: ray transform} restricted to the manifold \eqref{eq:MfldLines} is invertible.
To see this, consider a planar setting, i.e. $d=2$, with a single source position and a line detector of 512 detector elements. Then \eqref{eq:Criteria} becomes
\begin{align*}
    M_t \ge 11 N/512 > 0
\end{align*}
which for $N=1$ means that $M_t = 1$ suffices, which is reconstruction from a single view. However, for $M_t =1$ the inversion problem reduces to the standard tomography problem and in that setting, it is well known that an arbitrary object cannot be recovered from a single view.
It is thus clear that one needs additional requirements to ensure invertibility. In addition, further requirements will be needed to ensure the inversion is mildly ill-posed or even better, well-posed.

\section{The inverse problem}\label{sec:IP}
We now revisit the spatiotemporal inverse problem in Definition~\ref{prob:InvProb1}. 
Rephrasing it in the setting defined by the specific parameterizations introduced in Sections~\ref{sec:MotionModel}--\ref{sec:ParamDomFwdOp} gives us the following inverse problem:
\begin{definition}[The inverse problem]\label{prob:InvProb2}
    Recover the parameters governing the motion and morphology for $N$ distinct particles $\matOp{\Gamma} :=(\vecOp{\gamma}_1,\ldots,\vecOp{\gamma}_N)$ with $\vecOp{\gamma}_i \in \mathbb{D}$ from data $\{g[\vecOp{m}]\}_{\vecOp{m}\in\mathbb{I}}\subset\mathbb{R}$, where
    \bea
        g[\vecOp{m}]\approx F[\matOp{\Gamma}](\vecOp{s}_{m_s}, \vecOp{r}_{m_r}, t_{m_t})
    \eea
    for each triplet $(\vecOp{s}_{m_s}, \vecOp{r}_{m_r}, t_{m_t})$ indexed by the multi-index $\vecOp{m} = (m_s, m_r, m_t)\in\mathbb{I}$. The index set $\mathbb{I} = \{1, \ldots, M_s\}\times\{1, \ldots, M_r\}\times\{1, \ldots, M_t\}$ defines the finite sampling of the manifold $\mathbb{M}\times[0, t_\textnormal{max}]$, and the forward operator $F$ models the observation process as defined in \eqref{eq:param_fwd_opt}.
\end{definition}
Least-squares type solutions of the inverse problem in Definition~\ref{prob:InvProb2} are obtained by solving
\begin{align}
    \widehat{\matOp{\Gamma}} \in \arg\min_{\matOp{\Gamma}}\, \matOp{L}\Bigl(\left\{g[\vecOp{m}]\right\}_{\vecOp{m}\in\mathbb{I}},\bigl\{F[\matOp{\Gamma}](\vecOp{s}_{m_s},\vecOp{r}_{m_r}, t_{m_t})\bigr\}_{\vecOp{m}\in\I} \Bigr) \label{eq:inverse_problem_Gamma}
\end{align}
where $\matOp{L} \colon \R^{|\I|} \times \R^{|\I|} \to \R$ is some suitably chosen loss function and $|\I| = M_s\times M_r\times M_t$. For example, choosing $\matOp{L}$ as a $p$-norm yields
\begin{align}
    \widehat{\matOp{\Gamma}} \in \arg\min_{\matOp{\Gamma}}\, \Bigl(\sum_{\vecOp{m}\in\I} \bigl| g[\vecOp{m}] - F[\matOp{\Gamma}](\vecOp{s}_{m_s},\vecOp{r}_{m_r}, t_{m_t})\bigr|^p\Bigr)^{1/p} \label{eq:inverse_problem_Gamma2}.
\end{align}

The spatiotemporal inverse problem in the general setting of Definition~\ref{prob:InvProb1} is ill-posed, where inversion of the ray transform is known to be unstable.
This means in particular that least-squares solutions are unsuitable, since such solution methods are unstable w.r.t.\@ perturbations/errors in data.
However, the specific variant of the inverse problem in Definition~\ref{prob:InvProb2} may come with better stability properties (the specific parameterization introduced in Sections~\ref{sec:MotionModel}--\ref{sec:ParamDomFwdOp} will here implicitly act as a regularization).
This is also why a least-squares solution, like the one in \eqref{eq:inverse_problem_Gamma2} for $p=2$, may work without any additional regularizer.

Besides stability there is the issue of uniqueness. The least-squares optimization problem is non-convex, meaning that it can have several locally-optimal solutions. 
In addition, gradient-based schemes for solving it could get trapped in undesirable local minima (see subsequent results in Tables~\ref{tab:gmm_ct_ablation_1}, \ref{tab:gmm_ct_ablation_2}), so initialization becomes important.
One approach for mitigating the risk of getting trapped in undesirable local minima is to split the updating of the dynamic motion parameters from the static morphological parameters.

\subsection{Alternating recovery of motion and morphology}
The control variable (the object we seek to recover) can be split as $\matOp{\Gamma}=[\matOp{\Gamma}_D,\matOp{\Gamma}_S]$ where $\matOp{\Gamma}_D = [\matOp{\Gamma}_P,\matOp{\Gamma}_R]$ represents the dynamic motion parameters with $\matOp{\Gamma}_P$ and $\matOp{\Gamma}_R$ denoting the projectile and rotational components respectively. Hence
\[
    \matOp{\Gamma}_P := \{\vecOp{\eta}_n\}_{n=1}^N \in [\R^d \times \R^d \times \R^d]^N
    \quad\text{and}\quad
    \matOp{\Gamma}_R := \{\vecOp{\Theta}_n\}_{n=1}^N \in [\R^{d_\matOp{R}} ]^N
    \text{ with $d_\matOp{R}:=d(d-1)/2$.}
\]
Furthermore, $\matOp{\Gamma}_S$ represents the static morphological parameters, so
\[
    \matOp{\Gamma}_S := \{(\alpha_n,\matOp{U}_n)\}_{n=1}^N \in [\R_+ \times \R^{d_\matOp{U}}]^N
    \text{ with $d_\matOp{U}:=d (d+1)/2$.}
\] 
One can now adopt an alternating approach for computing an approximation to the least-squares solution that alternates updates for the projectile motion parameters $\matOp{\Gamma}_P$ and the rotational and static morphological parameters $[\matOp{\Gamma}_R, \matOp{\Gamma}_S]$. 
Translating this into formulas gives the following alternating scheme:
\begin{align}
    \widehat{\matOp{\Gamma}}^i_P &\in \arg\min_{\matOp{\Gamma}_P} \matOp{L}_P\Bigl(\left\{g[\vecOp{m}]\right\}_{\vecOp{m}\in\I},\bigl\{F[\matOp{\Gamma}_P,\widehat{\matOp{\Gamma}}^{i-1}_R,\widehat{\matOp{\Gamma}}^{i-1}_S](\vecOp{s}_{m_s},\vecOp{r}_{m_r}, t_{m_t})\bigr\}_{\vecOp{m}\in\I} \Bigr) \label{eq:ParticleTrajec}
    \\
    [\widehat{\matOp{\Gamma}}^{i}_R,\widehat{\matOp{\Gamma}}^i_S] &\in \arg\min_{[\matOp{\Gamma}_R,\matOp{\Gamma}_S]} \matOp{L}_S\Bigl(\left\{g[\vecOp{m}]\right\}_{\vecOp{m}\in\I},\bigl\{F[\widehat{\matOp{\Gamma}}^i_P,\matOp{\Gamma}_R,\matOp{\Gamma}_S](\vecOp{s}_{m_s},\vecOp{r}_{m_r}, t_{m_t})\bigr\}_{\vecOp{m}\in\I} \Bigr) 
    \label{eq:ParticleMorpho}
\end{align}
where $\matOp{L}_P \colon \R^{|\I|} \times \R^{|\I|} \to \R$ and $\matOp{L}_S \colon \R^{|\I|} \times \R^{|\I|} \to \R$ are suitably chosen loss functions. 
Such an approach may not necessarily converge to minimizers of the objectives, however, it could be sufficient for certain applications. Furthermore, the result could constitute a good initialization for future approaches.
A key step will be to select $\matOp{L}_P$ and $\matOp{L}_S$. A natural criterion is that these are appropriate quantifiers of similarity of motion and morphological parameters. 
Another aspect is computational. As we shall see next, an appropriate choice of $\matOp{L}_P$ will allow us to decouple the two optimization problems \eqref{eq:ParticleTrajec} and \eqref{eq:ParticleMorpho}.
This means the alternating approach in \eqref{eq:ParticleTrajec}--\eqref{eq:ParticleMorpho} can be replaced by a sequential one.

\subsection{Sequential recovery of motion and morphology}
It is desirable to choose $\matOp{L}_P$ so that the objective in \eqref{eq:ParticleTrajec} is independent of the morphological parameters $\widehat{\matOp{\Gamma}}^{i-1}_S \in [\R \times \R^{d_\matOp{U}}]^N$. 
The alternating scheme in \eqref{eq:ParticleTrajec}--\eqref{eq:ParticleMorpho} that updates projectile motion and rotation/morphological parameters in an alternating manner can then be replaced by a sequential scheme. Under this scheme, the projectile motion parameters are independently estimated in a preprocessing step by solving \eqref{eq:ParticleTrajec}. 

To choose such an $\matOp{L}_P$, we begin by characterizing how the ray transform maps the modes of a function \eqref{eq:f_generative_model_GMM} representing the object (the modes are the points where the function attains local maxima).
Since the ray transform is linear, it is enough to characterize how the ray transform maps the modes of an anisotropic Gaussian.
\begin{proposition}[Mode of ray transform of an anisotropic Gaussian]\label{prop: mode of ray of ansitropic gaussian}
Let $\rho_0$ be an isotropic Gaussian, $\matOp{A}=(\matOp{U}^{-1},\vecOp{\mu})\in\Aff(d)$ with $\matOp{U}$ invertible, and let $\vecOp{s}$ be a fixed source. Regard $\vecOp{r} \mapsto \RayTran(\matOp{A}.\rho_0)[\vecOp{s},\vecOp{r}]$ as a function of the detector location, equivalently of the ray direction $\vecOp{\omega} := \widehat{\vecOp{r} - \vecOp{s}}$. Then:
\begin{enumerate}[(i)]
    \item The Gaussian (exponential) factor of $\RayTran(\matOp{A}.\rho_0)[\vecOp{s},\vecOp{r}]$ attains its maximum exactly along the ray that is co-linear with the source $\vecOp{s}$ and the mode $\vecOp{\mu}$, i.e.\ for $\vecOp{r} = \vecOp{s} - \lambda(\vecOp{s}-\vecOp{\mu})$, $\lambda\in\R$. This holds for every invertible $\matOp{U}$.
    \item If the morphology is isotropic, $\matOp{U}^\top\matOp{U} = c^2\matOp{I}$ for some $c>0$ (so that $\matOp{A}.\rho_0$ is again isotropic), then the algebraic prefactor is constant in $\vecOp{\omega}$ and the mode of $\vecOp{r} \mapsto \RayTran(\matOp{A}.\rho_0)[\vecOp{s},\vecOp{r}]$ itself is attained along this co-linear ray. In particular the mode is co-linear with $\vecOp{s}$ and $\vecOp{\mu}$ and independent of the (scalar) morphology $c$.
\end{enumerate}
\end{proposition}
\begin{proof}
    Since the line determined by the source--detector pair $(\vecOp{s},\vecOp{r})$ depends only on the direction $\vecOp{\omega} = \widehat{\vecOp{r} - \vecOp{s}}$, the value $\RayTran(\matOp{A}.\rho_0)[\vecOp{s},\vecOp{r}]$ depends on $\vecOp{r}$ only through $\vecOp{\omega}$; maximizing over detector locations is therefore equivalent to maximizing over admissible unit directions $\vecOp{\omega}$. Using $\widehat{\matOp{U}(\vecOp{s}-\vecOp{r})} = -\widehat{\matOp{U}\vecOp{\omega}}$, expression \eqref{eq: ray of Gaussian} becomes
    \be
        \RayTran(\matOp{A}.\rho_0)[\vecOp{s},\vecOp{r}] = 
        \underbrace{\dfrac{\sqrt{\pi}}{\bigl\|{\matOp{U}\vecOp{\omega}}\bigr\|}}_{=:\,P(\vecOp{\omega})}\;
        \underbrace{\exp\left(- \bigl\|\matOp{U}(\vecOp{s} - \vecOp{\mu})\bigr\|^2
        \left[1 - {\bigl\langle 
        \widehat{\matOp{U}\vecOp{\omega}},\, \widehat{\matOp{U}(\vecOp{s} - \vecOp{\mu})} \bigr\rangle^2}\right]\right)}_{=:\,E(\vecOp{\omega})}.
        \label{eq: ray of Gaussian alt}
    \ee
    Both factors are positive. Writing $\bar\sigma \ge \underline{\sigma} > 0$ for the largest and smallest singular values of $\matOp{U}$, we have $\underline{\sigma} \le \|\matOp{U}\vecOp{\omega}\| \le \bar\sigma$, so that
    \be
        \bar\sigma^{-1}\sqrt{\pi}\,E(\vecOp{\omega}) \;\le\; \RayTran(\matOp{A}.\rho_0)[\vecOp{s},\vecOp{r}] \;\le\; \underline{\sigma}^{-1}\sqrt{\pi}\,E(\vecOp{\omega}).
        \label{eq:ray_sandwich}
    \ee
    These two-sided bounds show only that the projection is comparable in magnitude to the Gaussian factor $E$. Because the enclosing constants differ, and $P$ genuinely depends on $\vecOp{\omega}$ unless $\matOp{U}$ is isotropic, \eqref{eq:ray_sandwich} does \emph{not} by itself imply that the projection and $E$ share a maximizer; the prefactor $P$ can move the maximizer. We therefore analyze the two factors directly.

    \emph{(i)}\ By Cauchy--Schwarz, $0 \le \bigl\langle \widehat{\matOp{U}\vecOp{\omega}},\, \widehat{\matOp{U}(\vecOp{s}-\vecOp{\mu})} \bigr\rangle^2 \le 1$, and since $t \mapsto e^{-\|\matOp{U}(\vecOp{s}-\vecOp{\mu})\|^2(1-t)}$ is strictly increasing, $E$ is maximized exactly when
    \be
        {\bigl\langle 
        \widehat{\matOp{U}\vecOp{\omega}},\, \widehat{\matOp{U}(\vecOp{s} - \vecOp{\mu})} \bigr\rangle^2} = 1
        \iff \matOp{U}\vecOp{\omega} \parallel \matOp{U}(\vecOp{s}-\vecOp{\mu}).
        \label{eq:source_receiver_mu_colinear}
    \ee
    As $\matOp{U}$ is invertible, \eqref{eq:source_receiver_mu_colinear} is equivalent to $\vecOp{\omega} \parallel (\vecOp{s}-\vecOp{\mu})$, i.e.\ the ray is co-linear with $\vecOp{s}$ and $\vecOp{\mu}$; equivalently $\vecOp{r} = \vecOp{s} - \lambda(\vecOp{s}-\vecOp{\mu})$ for some $\lambda\in\R$. This is independent of $\matOp{U}$.

    \emph{(ii)}\ If $\matOp{U}^\top\matOp{U} = c^2\matOp{I}$ then $\|\matOp{U}\vecOp{\omega}\| = c$ for every unit $\vecOp{\omega}$, so $P(\vecOp{\omega}) = \sqrt{\pi}/c$ is constant and the bounds in \eqref{eq:ray_sandwich} coincide, giving $\RayTran(\matOp{A}.\rho_0)[\vecOp{s},\vecOp{r}] = (\sqrt{\pi}/c)\,E(\vecOp{\omega})$. It therefore follows that $\argmax_{\vecOp{r}} \RayTran(\matOp{A}.\rho_0)[\vecOp{s},\vecOp{r}] = \argmax_{\vecOp{\omega}} E(\vecOp{\omega})$, which by (i) is the co-linear ray; it is independent of the scalar morphology $c$.
\end{proof}

\begin{remark}\label{rem:anisotropic_mode_shift}
    For a genuinely anisotropic morphology $\matOp{U}$ the prefactor $P(\vecOp{\omega}) = \sqrt{\pi}/\|\matOp{U}\vecOp{\omega}\|$ is not constant, and the maximizer of the full projection $P\cdot E$ is shifted away from the co-linear ray of part~(i) toward the elongation axis of the projected Gaussian. Since the exponent of $E$ has curvature of order $\|\matOp{U}(\vecOp{s}-\vecOp{\mu})\|^2$ at its maximizer whereas $\log P$ has $\vecOp{\omega}$-derivative of order one, the induced angular shift is of order $\|\matOp{U}(\vecOp{s}-\vecOp{\mu})\|^{-2}$ and vanishes as the projected Gaussian becomes concentrated. Accordingly, the mode-based trajectory objective \eqref{eq:ParticleTrajec} is used with the isotropic surrogate $\matOp{U}_n = \sigma_n\matOp{I}$, for which part~(ii) applies exactly, while the anisotropic morphology is recovered separately in the subsequent projection-domain stage \eqref{eq:ParticleMorpho}.
\end{remark}

Since the ray transform is linear, Proposition~\ref{prop: ray of stationary gaussian} shows that, for a fixed source $\vecOp{s}$ and time $t$, the projection
\[
  \vecOp{r} \mapsto F[\matOp{\Gamma}](\vecOp{s},\vecOp{r},t) = \sum_{n=1}^N \alpha_n\,\RayTran\bigl(\matOp{A}[\matOp{U}_n,\vecOp{\Theta}_n,\vecOp{\eta}_n](t).\rho_0\bigr)[\vecOp{s},\vecOp{r}]
\]
is a positive combination of $N$ unimodal components on the $(d-1)$-dimensional detector manifold. By \eqref{eq: ray of Gaussian alt}, the $n$-th component is the product of a strictly positive, smooth weight $P(\vecOp{\omega})$ that is bounded above and below (constant for the isotropic morphology used in the trajectory stage, by Proposition~\ref{prop: mode of ray of ansitropic gaussian}(ii), and inducing the mild shift quantified in Remark~\ref{rem:anisotropic_mode_shift} for anisotropic morphology) and the factor $E(\vecOp{\omega})$, with mode
\[
  \widehat{\vecOp{r}}_n(t) := \argmax_{\vecOp{r}} \RayTran\bigl(\matOp{A}[\matOp{U}_n,\vecOp{\Theta}_n,\vecOp{\eta}_n](t).\rho_0\bigr)[\vecOp{s},\vecOp{r}]
\]
given by Proposition~\ref{prop: mode of ray of ansitropic gaussian}. In parallel-beam geometry, where the detector coordinate is an affine function of the offset of the line, each component is exactly an anisotropic Gaussian on the detector. In the fan-beam geometry considered here this holds only to second order about $\widehat{\vecOp{r}}_n(t)$, since the exponent of $E$ is quadratic in the sine of the angle between $\matOp{U}\vecOp{\omega}$ and $\matOp{U}(\vecOp{s}-\vecOp{\mu})$ rather than in the detector coordinate. We denote by $\bigl(\Sigma^{\RayTran}_n(t)\bigr)^2$ the covariance of this second-order Gaussian approximation of the $n$-th component (in the convention of \cite{zickert2022gaussian}), and write $\sigma^{\RayTran}_{n,\max}(t) \ge \sigma^{\RayTran}_{n,\min}(t) > 0$ for the largest and smallest eigenvalues of $\Sigma^{\RayTran}_n(t)$. For the resulting Gaussian mixture, the quantitative bound of \cite{zickert2022gaussian} makes precise the sense in which the modes of the mixture stay close to the individual component modes.

\begin{corollary}\label{cor: modes of object and its particles}
    Assume that the trajectories $t \mapsto \vecOp{C}[\vecOp{\eta}_n](t)$ of the $n=1,\ldots,N$ particles lie between the fixed source $\vecOp{s}$ and any of the detectors $\vecOp{r}$. Fix a time $t$ and consider the detector-domain Gaussian mixture obtained by replacing each component of $\vecOp{r} \mapsto F[\matOp{\Gamma}](\vecOp{s},\vecOp{r},t)$ by its second-order Gaussian approximation described above (this replacement is exact in parallel-beam geometry). Then every mode $\vecOp{r}'$ of this mixture lies in the neighborhood
    \be
        \mathcal{N}(t) := \bigcup_{n=1}^N \overline{\mathcal{B}}\bigl(\widehat{\vecOp{r}}_n(t),\,\varrho_n(t)\bigr),
        \qquad
        \varrho_n(t) := \sqrt{d-1}\,\dfrac{\bigl(\sigma^{\RayTran}_{n,\max}(t)\bigr)^2}{\sigma^{\RayTran}_{n,\min}(t)},
        \label{eq:mode_neighborhood}
    \ee
    of the component modes $\widehat{\vecOp{r}}_n(t)$ of $\vecOp{r} \mapsto \RayTran\bigl(\matOp{A}[\matOp{U}_n,\vecOp{\Theta}_n,\vecOp{\eta}_n](t).\rho_0 \bigr)[\vecOp{s},\vecOp{r}]$, where $\overline{\mathcal{B}}(\vecOp{c},\varrho)$ denotes the closed Euclidean ball of radius $\varrho$ centered at $\vecOp{c}$ and $d-1$ is the dimension of the detector manifold. Equivalently, each mode of the mixture lies within Euclidean distance $\varrho_n(t)$ of $\widehat{\vecOp{r}}_n(t)$ for at least one component $n$. In the isotropic case $\Sigma^{\RayTran}_n(t) = \sigma^{\RayTran}_n(t)\,\matOp{I}$ the radius reduces to $\varrho_n(t) = \sqrt{d-1}\,\sigma^{\RayTran}_n(t)$, and this bound cannot be improved by a global constant.
    Furthermore, the component modes $\widehat{\vecOp{r}}_n(t)$ trace the set of curves
    \be
    \left\{ t \mapsto (\vecOp{r},t) \text{ such that }
      \dfrac{\vecOp{s}-\vecOp{C}[\vecOp{\eta}_n](t)}{\bigl\|\vecOp{s}-\vecOp{C}[\vecOp{\eta}_n](t)\bigr\|} \cdot
      \dfrac{\vecOp{s}-\vecOp{r}}{\|\vecOp{s}-\vecOp{r}\|} = 1
    \right\}_{n=1}^N.
    \ee
\end{corollary}
\begin{proof}
    Under the stated approximation, $\vecOp{r} \mapsto F[\matOp{\Gamma}](\vecOp{s},\vecOp{r},t)$ is a Gaussian mixture on the detector manifold with positive weights $\alpha_n$, component means $\widehat{\vecOp{r}}_n(t)$, and covariances $\bigl(\Sigma^{\RayTran}_n(t)\bigr)^2$. By Proposition~\ref{prop: mode of ray of ansitropic gaussian} the mean of the $n$-th component is the mode $\widehat{\vecOp{r}}_n(t)$ (co-linear with $\vecOp{s}$ and $\vecOp{C}[\vecOp{\eta}_n](t)$, and morphology-independent for the isotropic surrogate of Proposition~\ref{prop: mode of ray of ansitropic gaussian}(ii); see Remark~\ref{rem:anisotropic_mode_shift}), which is co-linear with the source $\vecOp{s}$ and the trajectory point $\vecOp{C}[\vecOp{\eta}_n](t)$: indeed $\widehat{(\vecOp{s}-\vecOp{C}[\vecOp{\eta}_n](t))} \cdot \widehat{(\vecOp{s}-\vecOp{r})} = 1$ holds iff these unit vectors coincide, i.e.\ iff $\vecOp{s}$, $\vecOp{C}[\vecOp{\eta}_n](t)$ and $\vecOp{r}$ are co-linear with the particle lying between the source and the detector (as assumed), which is precisely the stated curve characterization. Applying \cite[Theorem~4.4]{zickert2022gaussian} to this mixture, any local maximum $\vecOp{r}'$ satisfies $\|\vecOp{r}' - \widehat{\vecOp{r}}_n(t)\| \le \sqrt{d-1}\,\bigl(\sigma^{\RayTran}_{n,\max}(t)\bigr)^2/\sigma^{\RayTran}_{n,\min}(t) = \varrho_n(t)$ for at least one index $n$, i.e.\ $\vecOp{r}' \in \mathcal{N}(t)$. The isotropic special case is \cite[Corollary~4.5]{zickert2022gaussian}, and the impossibility of improving the bound by a global constant is \cite[Proposition~4.6]{zickert2022gaussian}. The analogous statement for the object $f(\,\cdot\,,t)$ itself (as a Gaussian mixture on $\R^d$) recovers the result of \cite[Theorem~3.4]{zickert2022joint}.
\end{proof}

For the fan-beam data of Section~\ref{sec:num_example}, the localization bound \eqref{eq:mode_neighborhood} therefore holds only approximately. It is used only to justify matching observed modes to predicted modes; the predicted mode locations themselves (Propositions~\ref{prop: mode of ray of ansitropic gaussian} and~\ref{prop:traj_to_modes_map}) are exact for isotropic Gaussians.

If one chooses $\matOp{L}_P$ such that the objective in \eqref{eq:ParticleTrajec} corresponds to a distance between local maxima of $\{g_m\}_{m=1}^M$ and modes of 
$\RayTran\bigl(\matOp{A}[\matOp{U}_n,\vecOp{\Theta}_n, \vecOp{\eta}_n](t). \rho_0\bigr)$
then Proposition~\ref{prop: mode of ray of ansitropic gaussian} and Corollary~\ref{cor: modes of object and its particles} imply that the minimizer $\widehat{\matOp{\Gamma}}^i_P$ in \eqref{eq:ParticleTrajec} is independent of $\widehat{\matOp{\Gamma}}^{i-1}_S$. In particular, $\widehat{\matOp{\Gamma}}^i_P$ will not depend on the morphology, so consequently there is no need for iterating over $i$ in \eqref{eq:ParticleTrajec}--\eqref{eq:ParticleMorpho}.
It is therefore advantageous to work with such an $\matOp{L}_P$, as it helps towards estimating the trajectories of the particles $\vecOp{C}[\vecOp{\eta}_n](t)$ without requiring the knowledge of their rotational motion or morphologies. Consequently, recovery of the remaining rotation and morphology requires modification of the optimization problem \eqref{eq:ParticleMorpho} to
\begin{align}
        [\widehat{\matOp{\Gamma}}_R,\widehat{\matOp{\Gamma}}_S]\ &\in \arg\min_{[\matOp{\Gamma}_R,\matOp{\Gamma}_S]} \matOp{L}_S\Bigl(\left\{g[\vecOp{m}]\right\}_{\vecOp{m}\in\I},\bigl\{F[\matOp{\Gamma}_P,\matOp{\Gamma}_R,\matOp{\Gamma}_S](\vecOp{s}_m,\vecOp{r}_m, t_m)\bigr\}_{\vecOp{m}\in\I} \Bigr)\label{eq:ParticleMorphoFixed}
\end{align}
(i.e. dropping the dependency on the sequencing superscript $i$), where all trajectory parameters in $\matOp{\Gamma}_P$ are obtained by solving \eqref{eq:ParticleTrajec} with the aforementioned $\matOp{L}_P$.

In the following we show how to select $\matOp{L}_P$ (given by \eqref{eq:ParticleTrajec_Hausdorff}), in the setting with fixed acquisition geometry and a single source.

\subsection{Case with fixed acquisition geometry and a single source}
Using the multi-index notation $\vecOp{m}=[{m_{\vecOp{r}}},{m_t}]$, we can denote measurements as 
\begin{align}
    g[{\vecOp{m}}] = g[{m_{\vecOp{r}}},{m_t}]  \approx F[\matOp{\Gamma}](\vecOp{s},\vecOp{r}_{m_{\vecOp{r}}}, t_{m_t})
\quad\text{with $(\vecOp{s},\vecOp{r}_{m_{\vecOp{r}}}, t_{m_t}) \in \Mlines \times [0, t_{\max}]$,
}
\end{align} 
where  
\be
\I = \{\vecOp{m}=[{m_{\vecOp{r}}},{m_t}] \text{ such that } m_{\vecOp{r}}=1,\ldots,M_{\vecOp{r}} \text{ and } m_t = 1,\ldots,M_t \}
\ee 
is the set of detector and time indices. 

In this setting, we optimize using the observed modes present in the projection data, and the predicted modes that can be obtained as a function of a trajectory parameter prediction $\vecOp{\eta}$.

\begin{definition}[Observed and predicted projection modes]\label{def:proj_modes}
    Let $\vecOp{M}[g](t)$ denote the points of local maxima of $g[\vecOp{m}]$ with respect to a fixed arrangement of detectors $\vecOp{r}\in \{ \vecOp{r}_{m_{\vecOp{r}}}\}_{{m_{\vecOp{r}}}=1}^{M_{\vecOp{r}}}$ and a fixed set of time points $t\in \{t_{m_t}\}_{m_t=1}^{M_t}$. Furthermore, define 
    \[ {\vecOp{C}}_{\RayTran}(t) := \bigcup_{n=1}^N {\vecOp{C}}_{\RayTran}[\vecOp{\eta}_n](t), \]
    where 
    \be 
        {\vecOp{C}}_{\RayTran}[\vecOp{\eta}_n](t) = \argmin_{\vecOp{r}}
       \left| \dfrac{\vecOp{s}-\vecOp{C}[\vecOp{\eta}_n](t)}{\|\vecOp{s}-\vecOp{C}[\vecOp{\eta}_n](t)\|} \cdot
      \dfrac{\vecOp{s}-\vecOp{r}}{\|\vecOp{s}-\vecOp{r}\|} - 1 \right|
            \label{eq:set_ray_transform_of_single trajectory}
    \ee
    where the minimization in \eqref{eq:set_ray_transform_of_single trajectory} is performed over the detector domain, i.e.\ the subset of $\mathbb{M}$ describing the fixed detector geometry. For such a fixed geometry, the minimizer is the unique detector location at which the ray emanating from the source $\vecOp{s}$ through the particle center $\vecOp{C}[\vecOp{\eta}_n](t)$ meets the detector locus; hence each $\vecOp{C}_{\RayTran}[\vecOp{\eta}_n](t)$ is a single point, and $\vecOp{C}_{\RayTran}(t)$ collects the $N$ resulting modes.
\end{definition}

\begin{remark}[Non-uniqueness of the minimizer in \eqref{eq:set_ray_transform_of_single trajectory}]\label{rem:mode_nonuniqueness}
The functional minimized in \eqref{eq:set_ray_transform_of_single trajectory} depends on the detector location $\vecOp{r}$ only through the \emph{direction} $\widehat{\vecOp{s}-\vecOp{r}}$. Indeed, $\widehat{\vecOp{s}-\vecOp{C}[\vecOp{\eta}_n](t)}\cdot\widehat{\vecOp{s}-\vecOp{r}}\le 1$ with equality if and only if $\widehat{\vecOp{s}-\vecOp{r}}=\widehat{\vecOp{s}-\vecOp{C}[\vecOp{\eta}_n](t)}$, so the minimum value $0$ is attained on the entire half-line
\[
   \mathcal{L}_n(t) := \bigl\{\vecOp{s}+\lambda\,(\vecOp{s}-\vecOp{C}[\vecOp{\eta}_n](t)) : \lambda<0\bigr\}
\]
through the source $\vecOp{s}$ and the particle center $\vecOp{C}[\vecOp{\eta}_n](t)$. The minimizer over all $\vecOp{r}\in\R^d$ is therefore not a single point but this one-parameter ray, which merely reflects that a single projection fixes the viewing direction of the mode but not its range. Restricting $\vecOp{r}$ to a fixed detector geometry, as in Definition~\ref{def:proj_modes}, removes this ambiguity: the ray $\mathcal{L}_n(t)$ meets the detector locus at a single point, so $\vecOp{C}_{\RayTran}[\vecOp{\eta}_n](t)$ is uniquely defined (a discrete detector arrangement returns the nearest detector as its sampled counterpart). The scalar $\lambda$ that locates this intersection along $\mathcal{L}_n(t)$ is given in closed form in Proposition~\ref{prop:traj_to_modes_map}, which selects $\lambda_n$ from a known detector component and thereby yields the unique, differentiable mode $\widehat{\vecOp{r}}[\vecOp{\eta}_n](t)$ used in the trajectory objective.
\end{remark}

Equation \eqref{eq:set_ray_transform_of_single trajectory} provides an expression for the theoretical mode of the $n$-th Gaussian projection, but it is not an explicit map from trajectory parameter space to mode space and is therefore not applicable to gradient descent optimization. Instead, we require the following equivalent result.
\begin{proposition}\label{prop:traj_to_modes_map}
    Let $d \geq 2$, let $\vecOp{\eta}_n \in \matOp{\Gamma}_P$, let $\matOp{U_n} = \xi\matOp{I}$ for some $\xi\in\mathbb{R}^+$, and let $\vecOp{C}[\vecOp{\eta}_n](t)$ be the corresponding $\mathbb{R}^d$-valued trajectory. Assuming the same fixed detector geometry as in Definition \ref{def:proj_modes}, the unique mode of 
    $\alpha_n \RayTran\bigl(\matOp{A}[\matOp{U}_n,\vecOp{\Theta}_n, \vecOp{\eta}_n](t). \rho_0\bigr)$
    is
    \begin{equation}\label{eq:maximizing_rcvr_formula}
        \widehat{\vecOp{r}}[\vecOp{\eta}_n](t) = \vecOp{s} + \lambda_n(t)\bigl(\vecOp{s} - \vecOp{C}[\vecOp{\eta}_n](t)\bigr),
    \end{equation}
    where
    \bea
        \lambda_n(t) = \frac{\vecOp{r}^{(0)} - \vecOp{s}^{(0)}}{\vecOp{s}^{(0)} - \vecOp{C}^{(0)}[\vecOp{\eta}_n](t)}.
    \eea
    Here, for a vector $\vecOp{x}\in\R^d$, the superscript $\vecOp{x}^{(0)}$ denotes a single prescribed coordinate of $\vecOp{x}$, with the \emph{same} index used for $\vecOp{r}^{(0)}$, $\vecOp{s}^{(0)}$, and $\vecOp{C}^{(0)}[\vecOp{\eta}_n](t)$ and chosen so that $\vecOp{s}^{(0)} \neq \vecOp{C}^{(0)}[\vecOp{\eta}_n](t)$ (which is possible whenever $\vecOp{s}\neq\vecOp{C}[\vecOp{\eta}_n](t)$). For a fixed acquisition geometry, $\vecOp{r}^{(0)}$ is the coordinate of the detector location prescribed by the detector locus (for instance the fixed detector-array coordinate). In particular, \eqref{eq:maximizing_rcvr_formula} defines an explicit and smooth map from a trajectory parameterization to the unique mode of its associated Gaussian at time $t$.
\end{proposition}
\begin{proof}
    By the isotropy assumption $\matOp{U_n} = \xi\matOp{I}$, rotation does not manifest in the projections, meaning that one can write $\widetilde{\matOp{R}} = \matOp{U}_n\matOp{R}_{t\vecOp{\Theta}_n}^\top = \xi\matOp{R}_{t\vecOp{\Theta}_n}^\top$. Additionally, the following simplifications apply:
    \[
    \|\widetilde{\matOp{R}}(\vecOp{r} - \vecOp{s})\| = \xi\|\vecOp{r} - \vecOp{s}\|,\qquad
    \bigl\langle\widetilde{\matOp{R}}(\vecOp{r} - \vecOp{s}),\,\widetilde{\matOp{R}}(\vecOp{s} - \vecOp{C}_n(t)\bigr\rangle
    = \xi^2\bigl\langle \vecOp{r} - \vecOp{s},\, \vecOp{s} - \vecOp{C}_n(t)\bigr\rangle,
    \]
    where $\vecOp{C}_n(t) = \vecOp{C}[\vecOp{\eta}_n](t)$. Using these results, the ray transform can be factored as
    $
    \RayTran\bigl(\matOp{A}[\matOp{U}_n,\vecOp{\Theta}_n, \vecOp{\eta}_n](t). \rho_0\bigr)[\vecOp{s}, \vecOp{r}]
    = P_n(\vecOp{s}, t)\,Q_n(\vecOp{s}, \vecOp{r}, t)$,
    with
    \[
    P_n := \frac{\sqrt{\pi}}{\xi}\exp\!\bigl(-\|\widetilde{\matOp{R}}(\vecOp{s} - \vecOp{C}_n(t))\|^2\bigr),\qquad
    Q_n := \exp\!\!\left(\xi^2\frac{\bigl\langle \vecOp{r} - \vecOp{s},\, \vecOp{s} - \vecOp{C}_n(t)\bigr\rangle^2}{\|\vecOp{r} - \vecOp{s}\|^2}\right).
    \]
    Since $P_n$ is $\vecOp{r}$-independent, the theoretical modes at time $t$ are determined by $\nabla_{\vecOp{r}} Q_n = \vecOp{0}$. Computing:
    \[
    \frac{\partial Q_n}{\partial r_i}
    = \frac{2\xi^2Q_n}{\|\vecOp{r} - \vecOp{s}\|^4}\,\bigl\langle \vecOp{r} - \vecOp{s},\, \vecOp{s} - \vecOp{C}_n(t)\bigr\rangle
      \Bigl(\|\vecOp{r} - \vecOp{s}\|^2(\vecOp{s} - \vecOp{C}_n(t))_i - (\vecOp{r} - \vecOp{s})_i\bigl\langle \vecOp{r} - \vecOp{s},\,\vecOp{s} - \vecOp{C}_n(t)\bigr\rangle\Bigr).
    \]
    The prefactor $\langle \vecOp{r} - \vecOp{s}, \vecOp{s} - \vecOp{C}\rangle = 0$ would require $\vecOp{r} - \vecOp{s} \perp \vecOp{s} - \vecOp{C}$, which is geometrically excluded in the acquisition geometry. Hence $\nabla_{\vecOp{r}} Q_n = \vecOp{0}$ reduces to
    \begin{equation}\label{eq:grad_r_eqn}
    \|\vecOp{u}\|^2 \vecOp{v}_n - \langle\vecOp{u},\,\vecOp{v}_n\rangle\, \vecOp{u} = 0,
    \qquad \vecOp{u} := \vecOp{r} - \vecOp{s},\quad \vecOp{v}_n := \vecOp{s} - \vecOp{C}[\vecOp{\eta}_n](t).
    \end{equation}
    For $\vecOp{u}, \vecOp{v}_n \neq \vecOp{0}$, \eqref{eq:grad_r_eqn} holds if and only if $\vecOp{u}$ and $\vecOp{v}_n$ are parallel, giving $\vecOp{r} = \vecOp{s} + \lambda_n(\vecOp{s} - \vecOp{C}[\vecOp{\eta}_n](t))$.
    Fixing $\lambda_n$ via the known component data $\vecOp{s}^{(0)}, \vecOp{r}^{(0)}, \vecOp{C}^{(0)}[\vecOp{\eta}_n](t)$ yields \eqref{eq:maximizing_rcvr_formula}.
\end{proof}

Corollary~\ref{cor: modes of object and its particles} and Proposition~\ref{prop:traj_to_modes_map} collectively imply that the 
projection of trajectory predictions computed by \eqref{eq:maximizing_rcvr_formula} for a GMM object
should be in the vicinity of the modes of the 
observed object.
Based on this, we select $\matOp{L}_P$ by minimizing, over all time samples, the Hausdorff distance between the data-mode set $\vecOp{M}[g](t)$ and the predicted-mode set $\widehat{\mathbb{R}}[\matOp{\Gamma}_P](t):=\bigl\{\widehat{\vecOp{r}}[\vecOp{\eta}_n](t)\bigr\}_{n=1}^N$:
\be
    \matOp{L}_P\Bigl(\left\{g{[\vecOp{m}]}\right\}_{\vecOp{m}\in\I},\bigl\{F[\matOp{\Gamma}_P,\matOp{\Gamma}_R, \matOp{\Gamma}_S](\vecOp{s},\vecOp{r}_{m_{\vecOp{r}}}, t_{m_t})\bigr\}_{\vecOp{m}\in\I} \Bigr) 
    :=
    \sum_{m_t=1}^{M_t}
    d_H\Bigl(\vecOp{M}[g](t_{m_t}), \widehat{\mathbb{R}}[\matOp{\Gamma}_P](t_{m_t}) \Bigr)
    \label{eq:ParticleTrajec_Hausdorff}   
\ee
where $d_H$ is the Hausdorff distance between subsets of $\R^d$:
\begin{equation}
    d_H(\mathbb{A},\mathbb{B}) = \max \left( 
    \sup_{\vecOp{a}\in \mathbb{A}} \ \inf_{\vecOp{b}\in \mathbb{B}} \|a-b \|,\,
    \sup_{\vecOp{b}\in \mathbb{B}} \ \inf_{\vecOp{a}\in \mathbb{A}} \| b-a \| \right)
\quad\text{for $\mathbb{A},\mathbb{B} \subset \R^d$.} \label{eq:Hausdorff}
\end{equation}
Consequently, the optimal trajectory parameters are given by 
\begin{align}
    \matOp{\Gamma}_P^*=\{\vecOp{\eta}^*_n \}_{n=1}^N =  
    \arg\min_{\{\vecOp{\eta}_n \}_{n=1}^N} 
    \sum_{m_t=1}^{M_t}
    d_H\Bigl(\vecOp{M}[g](t_{m_t}), \widehat{\mathbb{R}}[\matOp{\Gamma}_P](t_{m_t}) \Bigr).
    \label{eq:opt_trajectory_parameters}
\end{align}

Figure~\ref{fig:modes} visualizes the pre and post-optimization mode fitting procedure for a two-dimensional numerical example.
Note that computing the optimal trajectory parameters $\{\vecOp{\eta}^*_n \}_{n=1}^N$ in \eqref{eq:opt_trajectory_parameters} amounts to the following: the modes of the data are assigned to the predicted modes in a way that minimizes the total error over all times $t_{m_t}$.
Computing exact Hausdorff distances can be computationally expensive, especially for high-dimensional imprecise data \cite{taha2015efficient,knauer2011directed}, but there are efficient approximate methods \cite{charpiat2005approximations}.  

\begin{remark}[Relation between the stated and the implemented trajectory objective]\label{rem:assignment_vs_hausdorff}
The Hausdorff distance in \eqref{eq:ParticleTrajec_Hausdorff}--\eqref{eq:opt_trajectory_parameters} is the geometric criterion expressing that the predicted-mode set $\widehat{\mathbb{R}}[\matOp{\Gamma}_P](t)$ should coincide with the data-mode set $\vecOp{M}[g](t)$ at every sampled time. It is, however, a nested $\sup$--$\inf$ functional that is nonsmooth in $\matOp{\Gamma}_P$ and, crucially, it does not return the mode-to-particle correspondences required by the analytic Newton refinement associated with Proposition~\ref{prop:traj_to_modes_map}. 
For these reasons the implementation in Section~\ref{subsec:pseudocode} does \emph{not} minimize $d_H$ directly; instead it uses a smooth, directed point-to-set loss 
\be
    \matOp{L}_{\widehat{H}}(\matOp{\Gamma}_P) := \sum_{m_t=1}^{M_t}\sum_{\vecOp{r}\in\vecOp{M}[g](t_{m_t})}\matOp{H}_\delta\biggl(\min_{n\in\{1, \ldots, N\}}\bigl\|\vecOp{r} - \widehat{\vecOp{r}}[\vecOp{\eta}_n](t_{m_t})]\bigr\|\biggr)\label{eq:Hausdorff_surrogate}
\ee
as a surrogate for $d_H$, where $\matOp{H}_\delta$ is the Huber loss \eqref{eq:Huber_loss}. Equation \eqref{eq:Hausdorff_surrogate} differs from \eqref{eq:Hausdorff} through its ``one-sided'' property, in the sense that the data at each time are assigned to their nearest predicted trajectory mode, but not vice versa. By utilizing the Huber loss and considering the objective as a sum across all times rather than $\sup$ and $\inf$ operations, \eqref{eq:Hausdorff_surrogate} has sufficient smoothness and stability properties that make it better suited to quasi-Newton optimization methods than \eqref{eq:Hausdorff}.
\end{remark}

After approximately solving \eqref{eq:ParticleTrajec} by minimizing \eqref{eq:Hausdorff_surrogate}, the trajectories are fixed, and it remains to recover the rotation and morphology parameters as described by \eqref{eq:ParticleMorpho}. 
In the case of trajectory optimization, formulating \eqref{eq:ParticleTrajec} in terms of the modes serves to provide a concise measure of the projectile motion without knowing the morphology or rotational motion, allowing optimization to be performed independently of these quantities.
In contrast, because the $(\alpha_n)$ determine how the projection integral is scaled across all acquisition times, and the $(\matOp{U}_n)$ determine the spatial extent as a function of the viewing angle, the projection data is utilized in its entirety when fitting for the particle morphology $(\alpha_n, \matOp{U}_n)$, which is intertwined with the angular velocity $(\matOp{\Theta}_n)$ through the matrix product $\matOp{U}\matOp{R}_{t\vecOp{\Theta}}^\top$.

Consequently, $\matOp{L}_S$ is chosen to measure the discrepancy between the complete measured and modeled projections. One option is to choose $\matOp{L}_S$ as the $L_p$ norm for $p\ge 1$ or some (smooth) approximation it. We consider a smooth differentiable approximation of the $L_1$ loss -- the Huber loss function \cite{huber1964}:
\be
    \matOp{H}_\delta(u) := \begin{cases}
        \frac{1}{2}u^2 & |u|\le\delta,
        \\
        \delta\bigl(|u| - \frac{\delta}{2}\bigr) & |u| > \delta,
    \end{cases}\label{eq:Huber_loss}
\ee
applied as
\begin{multline}
    \matOp{L}_S\Bigl(\left\{g{[\vecOp{m}]}\right\}_{\vecOp{m}\in\mathbb{I}},\bigl\{F[\matOp{\Gamma}_P,\matOp{\Gamma}_R, \matOp{\Gamma}_S](\vecOp{s},\vecOp{r}_{m_{\vecOp{r}}}, t_{m_t})\bigr\}_{\vecOp{m}\in\mathbb{I}} \Bigr) 
    \\
    := \frac{1}{|\mathbb{I}|}\sum_{\vecOp{m}\in\mathbb{I}}\matOp{H}_\delta\Bigl(g[\vecOp{m}] - F[\matOp{\Gamma}_P, \matOp{\Gamma}_R,\matOp{\Gamma}_S](\vecOp{s}, \vecOp{r}_{m_r}, t_{m_t})\Bigr).\label{eq:morphology_loss}
\end{multline}
Consequently, the general sub-inverse problem \eqref{eq:ParticleMorpho} can be specifically formulated in terms of the loss \eqref{eq:morphology_loss} as that of solving
\begin{align}
[\matOp{\Gamma}_R^*,\matOp{\Gamma}_S^*] = \arg\min_{[\matOp{\Gamma}_R,\matOp{\Gamma}_S]} \frac{1}{|\mathbb{I}|}\sum_{\vecOp{m}\in\mathbb{I}}\matOp{H}_\delta\Bigl(g[\vecOp{m}] - F[\matOp{\Gamma}_P^*, \matOp{\Gamma}_R,\matOp{\Gamma}_S](\vecOp{s}, \vecOp{r}_{m_r}, t_{m_t})\Bigr)
\label{eq:opt_rotation_morphology}
\end{align}
where $\matOp{\Gamma}_P^*$ is the established projectile motion estimate.

\section{Reconstruction algorithm}
\subsection{Initialization}
The initialization methodology reflects the multi-stage nature of the algorithm, with parameters initialized in one manner to solve \eqref{eq:opt_trajectory_parameters} and another for \eqref{eq:opt_rotation_morphology}.

For the trajectory optimization step (i.e. Stage~1), initializations $\vecOp{\eta}_1, \ldots, \vecOp{\eta}_N$ are computed as independent and identically distributed (i.i.d.) samples from a probability distribution and optimized by solving \eqref{eq:opt_trajectory_parameters} via \eqref{eq:Hausdorff_surrogate}. To avoid choosing a solution resulting from an unfavorable initial sample, this procedure is completed $N_{\textnormal{traj}}$ times --- giving $N_\textnormal{traj}$ corresponding optimized solutions --- from which the candidate with the smallest loss value is selected. Such an approach is practically feasible due to the relatively low $\mathcal{O}(NM_td)$ evaluation cost of the GMM trajectory forward model \eqref{eq:maximizing_rcvr_formula}. Accurate trajectory estimation is fundamental to the quality of the overall reconstruction, making these repeated trials highly worthwhile.

While held fixed during trajectory optimization, placeholders for each $(\alpha_n, \matOp{U}_n, \vecOp{\Theta}_n)$ are required. 
By Proposition \ref{prop: mode of ray of ansitropic gaussian}, to ensure that the projection mode of a Gaussian is attained by the ray traveling closest through its center, each Gaussian is configured to be isotropic. The contributions of $\alpha_n > 0$ and $\vecOp{\Theta}_n$ are then redundant, with the latter following from the fact that the rotation of an isotropic Gaussian is unobservable in its sinogram. Exact choices of all of these Stage~1 quantities for a two-dimensional scenario, along with hyperparameters for the probability distributions for the random multi-start routines, are provided in Table~\ref{tab:hyperparameters}.

For Stage~2, a random multi-start strategy is also adopted, this time for the angular velocity parameters, with the morphology and attenuation parameters set deterministically to the same prior values for each trial. Randomization is performed in the rotation parameter since it is the most challenging to optimize for due to the highly non-convex structure of its loss landscape. Unlike in the trajectory setting, a projection mapping that simplifies this behavior to enable gradient-based optimization was not identified. 

Data-informed initialization methods such as grid-search on the rotation space of each Gaussian (with the remaining $N - 1$ fixed), and Non-Negative Least Squares (NNLS) on the attenuation vector $\vecOp{\alpha} = (\alpha_1, \ldots, \alpha_N)^\top$ were explored. However, after repeated randomized experiments, any accuracy gains offered by these initialization methods were negligible compared to pre-selecting prior values, making the associated computational overheads not worthwhile. It should be noted that these experiments were conducted only in two dimensions, and therefore say nothing about the utility of these or other data-informed initialization methods over randomization in higher-dimensional applications, which we leave for future work.

Lastly, we note that we do not initialize for the number of particles $N$ and assume it to be known, which in certain applications such as subsurface characterization would likely not be the case. In scenarios where $N$ particles are present but not fully identifiable, choosing $N$ for reconstruction may not then even provide the optimal fit in data or object space. This is a task in model order selection \cite{burnham2002model, stoica2004model} which we leave for future work.

\subsection{Pseudocode}\label{subsec:pseudocode}
We summarize the tomographic reconstruction algorithm for GMM objects undergoing parabolic motion.
The algorithm, called GMM-CT (\url{https://github.com/dwb26/gmm-ct}), is structured as follows.
\begin{description}
\item[Input:]\ 
\begin{itemize}
\item Dynamic sinogram: Time series of data $g[\vecOp{m}]=g[m_{\vecOp{r}},m_t]$.\vspace{4pt}
\item Source and detector locations and time samples: $(\vecOp{s},\vecOp{r}_{m_{\vecOp{r}}}, t_{m_t}) \in \Mlines \times [0, t_{\max}]$.
\item Number of particles: $N$.\vspace{4pt}
\item Experiment hyperparameters: e.g., initialization hyperparameters, optimization tolerances, and known prior information such as the trajectory ejection coordinates.
\end{itemize}
\vspace{.5em}
\item[Output:] \qquad The GMM structural and motion parameter estimate $\matOp{\Gamma}^* = [\matOp{\Gamma}_P^*, \matOp{\Gamma}_R^*, \matOp{\Gamma}_S^*]$.
\end{description}
The output is computed following a two-stage procedure that approximate \eqref{eq:opt_trajectory_parameters} and \eqref{eq:opt_rotation_morphology} respectively:
\begin{description}
\item[Stage 1:] Estimating projectile motion parameters $\matOp{\Gamma}_P^*$. \vspace{4pt}
\begin{itemize}
    \item Estimate the modes of the projection data $\{\vecOp{M}[g](t_{m_t})\}_{m_t=1}^{M_t}$.\vspace{4pt}
    \item For $n_\textnormal{traj} \textnormal{ in } \{1, \ldots, N_\textnormal{traj}\}$:
    \vspace{0.25em}
    \begin{itemize}
        \item Initialize a projectile motion trajectory parameter $\matOp{\Gamma}_P$.\vspace{4pt}
        \item Compute the modes of the predicted trajectories $\{\widehat{\mathbb{R}}[\matOp{\Gamma}_P](t_{m_t})\}_{m_t=1}^{M_t}$ using Proposition \ref{prop:traj_to_modes_map}.\vspace{4pt}
        \item Estimate the optimal trajectory parameter $\matOp{\Gamma}_P^*$ by minimizing \eqref{eq:Hausdorff_surrogate}.\vspace{4pt}
    \end{itemize}
    \item Choose $\matOp{\Gamma}_P^*$ from trial $n_\textnormal{traj}$ with the smallest loss.
\end{itemize}
\vspace{0.5em}
\item[Stage 2:] Estimating rotational motion and morphological parameters $[\matOp{\Gamma}_R^*,\matOp{\Gamma}_S^*]$.\vspace{4pt}
\begin{itemize}
    \item For $n_\textnormal{morph} \textnormal{ in } \{1, \ldots, N_\textnormal{morph}\}$:
    \vspace{0.25em}
    \begin{itemize}
        \item Initialize the parameters $[\matOp{\Gamma}_R, \matOp{\Gamma}_S]$.\vspace{4pt}
        \item Estimate optimal rotational motion and morphological parameters $[\matOp{\Gamma}_R^*,\matOp{\Gamma}_S^*]$ by minimizing \eqref{eq:morphology_loss} configured with $\matOp{\Gamma}_P^*$.\vspace{4pt}
    \end{itemize}
    \item Choose $[\matOp{\Gamma}_R^*, \matOp{\Gamma}_S^*]$ from trial $n_\textnormal{morph}$ with the smallest loss.
\end{itemize}
\end{description}

\vspace{0.25em}
As discussed in Remark~\ref{rem:assignment_vs_hausdorff}, in Stage~1 we minimize the assignment-based surrogate \eqref{eq:Hausdorff_surrogate} rather than the direct Hausdorff distance \eqref{eq:Hausdorff}. Once data modes are correctly assigned, experimentation has shown that further accuracy gains can be achieved via a Newton root--finder characterizing $\vecOp{\eta}_n$ as solutions of $\nabla_{\vecOp{r}} Q_n = \vecOp{0}$ (with $Q_n$ defined in the proof of Proposition~\ref{prop:traj_to_modes_map}). 

\section{Numerical results}\label{sec:num_example}
To test the performance of GMM-CT, we simulate $N$ experimental particles and motions via parameterizations $\{(\alpha_n, \matOp{U}_n, \vecOp{\Theta}_n, \vecOp{\eta}_n)\}_{n=1}^N$ and generate a sequence of $M_t$ projections according to the forward model \eqref{eq:param_fwd_opt}. Based on this data, accuracy is measured by comparing the reconstructed particles and motions to the simulated versions using appropriate metrics.

Throughout the following, a two-dimensional setting is considered, with sensing apparatus configured to provide a fan-beam acquisition geometry (i.e. a single source placed opposite multiple detectors). The source $\vecOp{s}$ is located at $(-1, -1)$, and the $M_r$ detectors $\{\vecOp{r}_{m_r}\}$ are located equidistantly along a vertical detector geometry in the set $\{4\}\times[-3, 1]\subset\mathbb{R}^2$, as shown in Figure \ref{fig:Acquisition geometry}.

Regardless of whether the particles are simulations or estimates, in alignment with the motivating example of Section \ref{sec:SubSurfCharCT}, they are assumed to be ejected from the same initial location vector $\vecOp{\mu}_0 = (1, 1)^\top$.\footnote{If true ejection points vary (tolerances or an extended discharge region), the model may misattribute position-induced projection differences to velocity/rotation differences, creating systematic trajectory errors that propagate into shape reconstruction. Extension and analysis of the proposed method to more general situations could be a topic of future study.} Additionally, it is assumed that the only acceleration they are subjected to is gravity, implying $\vecOp{a}_0 = (0, -9.81)^\top$ as the acceleration vector for all particles. By virtue of the projectile motion model $\vecOp{C}[\vecOp{\eta_n}](t) = \vecOp{\mu}_0 + t\vecOp{v}_n + \frac{t^2}{2}\vecOp{a}_0$, recovery of each of the trajectories is equivalent to recovering their initial velocity vector $\vecOp{v}_n$. We therefore refer to $\vecOp{\eta}_n$ and $\vecOp{v}_n$ interchangeably. Note that all spatial units are considered in meters (m) and time units in seconds (s).

\begin{figure}[!htbp]
    \centering
    \includegraphics[height=6.5cm]{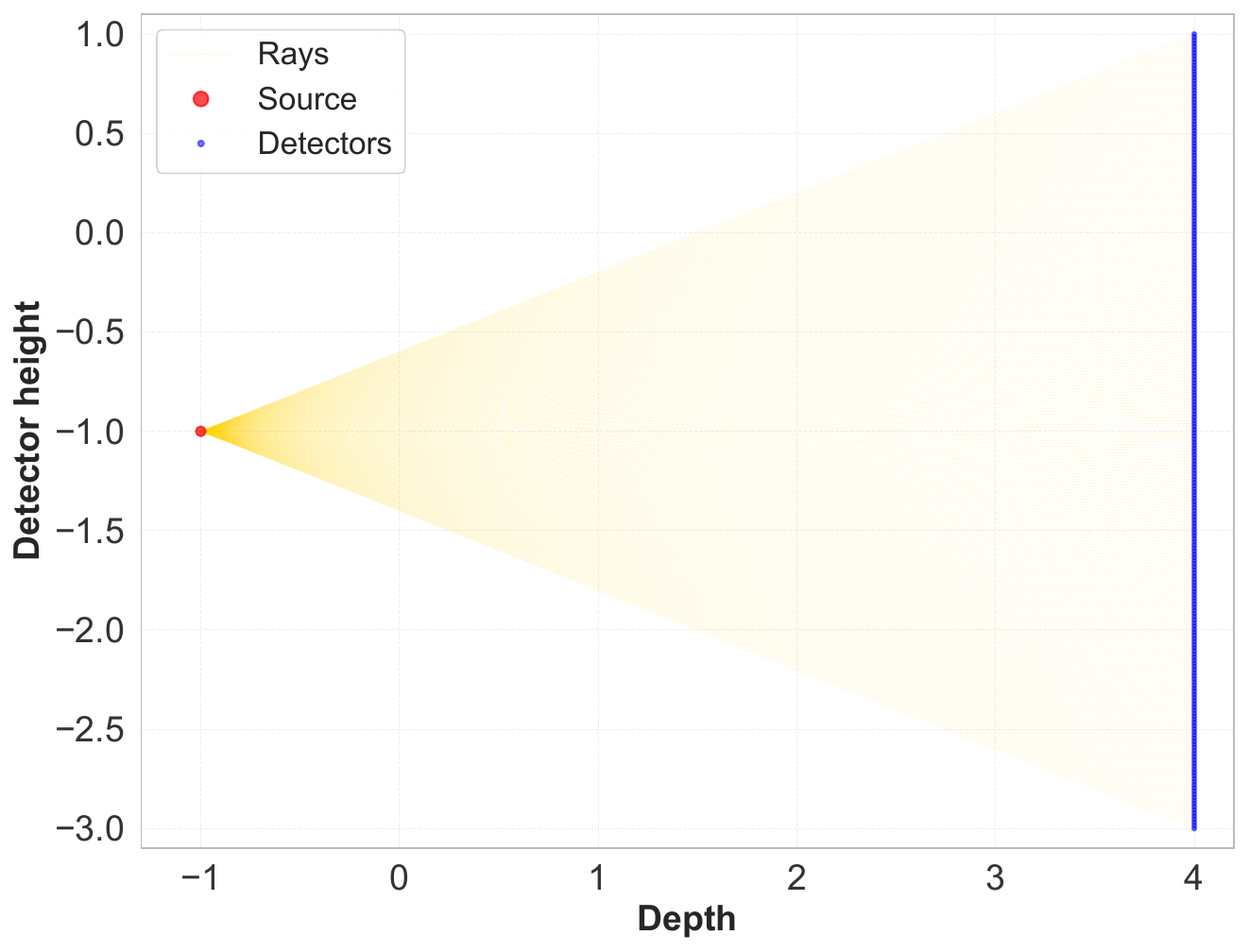}
    \caption{Acquisition geometry, with source located at $(-1, -1)$ and $M_r=256$ equally-spaced detectors placed between $(4, 1)$ and $(4, -3)$.}
    \label{fig:Acquisition geometry}
\end{figure}

\subsection{Data simulation}\label{sec:data_sim}
The simulated particle parameters $\vecOp{\gamma}_n = (\alpha_n, \matOp{U}_n, \Theta_n, \vecOp{v}_n)$ are each sampled for $n\in\{1, \ldots, N\}$ according to:
\be
\begin{dcases}
\alpha_n \sim \mathcal{N}(\mu_{\textnormal{atten}, n}, \sigma_\textnormal{atten}^2) \\[3pt]
\matOp{U}_n \sim \mathcal{P}_{\matOp{U}}(u_\textnormal{min}, u_\textnormal{max}) \\[3pt]
\Theta_n \sim \mathcal{U}(\theta_{\min}, \theta_{\max}) \\[4pt]
\vecOp{v}_{n} \sim \mathcal{P}_{\vecOp{\eta}}(\vecOp{\mu}_{\vecOp{\eta}},  \matOp{\Sigma}_{\vecOp{\eta}}). \\[4pt]
\end{dcases}\label{eq:simulation_model}
\ee
The distributions for $\alpha_n$ and $\Theta_n$ are standard, with specific hyperparameter values provided in Table \ref{tab:hyperparameters}. We note that while statistically it is possible that $\alpha_n < 0$, the ratio of $\mu_{\textnormal{atten}, n}$ to $\sigma_{\textnormal{atten}}$ is so large that in practice a negative $\alpha_n$ has never been sampled in our experiments, with the noise term included simply to introduce a small jitter.

The distributions used to simulate $\matOp{U}_n$ and $\vecOp{v}_n$ are more unconventional.
For morphology, $\mathcal{P}_\matOp{U}$ draws the geometric mean of the diagonal of $\matOp{U}_n$ log-uniformly (dictating the size) and the anisotropy ratio $\max(\textnormal{diag}\,\matOp{U}_n) / \min(\textnormal{diag}\,\matOp{U}_n)$ independently, rejecting proposals for which any entry of $\textnormal{diag}\,\matOp{U}_n$ falls outside a prescribed interval $[u_\textnormal{min}, u_\textnormal{max}]$. This procedure provides heterogeneity in the simulated batch, allowing particles to span a range of sizes and anisotropy ratios.

For velocity, the distribution $\mathcal{P}_{\vecOp{\eta}}$ is the normal distribution, truncated to produce samples with positive horizontal velocity components. Additionally, rejection sampling is used to mitigate overlap between the particle trajectories, which would make the particles indistinguishable in the projections, in which case no tomography method can be expected to succeed. The velocities are drawn sequentially, and a proposal $\vecOp{v}_n$ is rejected if, for some previously accepted $\vecOp{v}_k$, the vertical coordinates $y_n(t)$ and $y_k(t)$ of the two trajectories satisfy $|y_n(t)-y_k(t)| < \delta_{\mathrm{sep}}$ on more than a fraction $\kappa$ of a uniform grid of $128$ times in $[0,t_{\max}]$. We refer to $\delta_{\mathrm{sep}}$ as the \emph{trajectory separation} and to $\kappa$ as the \emph{max overlap fraction}; note that the accepted velocities are not independent. Because all particles share the ejection point $\vecOp{\mu}_0$ and the acceleration $\vecOp{a}_0$, $y_n(t)-y_k(t) = (v_{n,y}-v_{k,y})\,t$, so up to the time discretization the rule enforces a minimum spacing $|v_{n,y}-v_{k,y}| \ge \delta_{\mathrm{sep}}/(\kappa\, t_{\max})$ between vertical launch velocities. With $\delta_{\mathrm{sep}} = 0.2$ and $t_{\max} = 2$ (Table~\ref{tab:hyperparameters}), this spacing is $0.25$ for $\kappa = 0.4$ and $0.5$ for $\kappa = 0.2$.

\begin{figure}[!htbp]
    \centering
    \includegraphics[width=\linewidth]{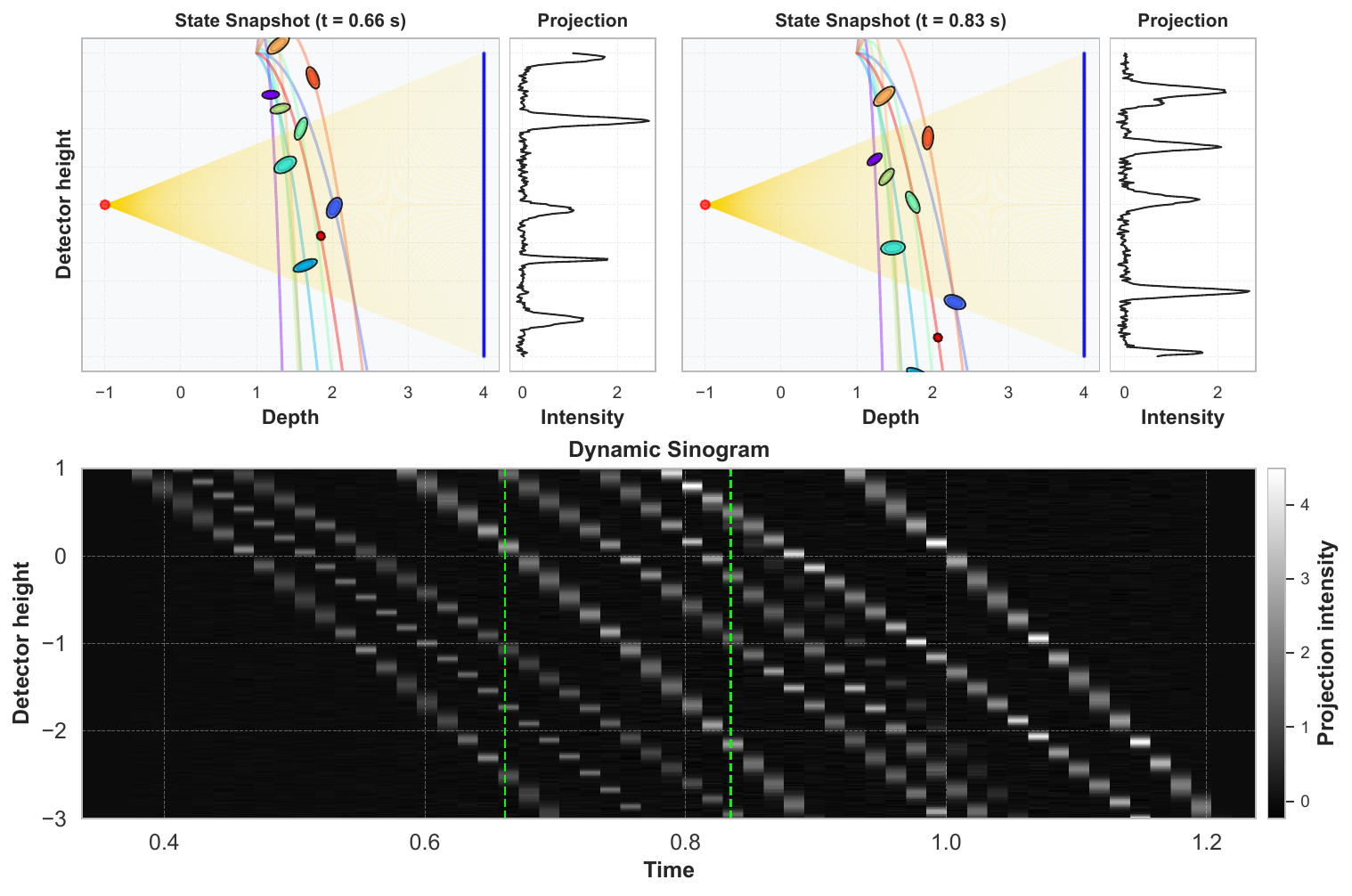}
    \caption{Simulated experimental data for a $9$-particle GMM. The green lines in the bottom row represent the acquisition times in the state-space snapshots in the top row.}
    \label{fig:state-sinogram-N5}
\end{figure}

Once a ground truth parameter $\matOp{\Gamma} = (\vecOp{\gamma}_1, \ldots, \vecOp{\gamma}_N)$ is simulated, noisy projections $g[\vecOp{m}]$ of the resulting GMM at times $t_1, \ldots, t_{M_t}$ are computed according to the forward model evaluations
\bea
    g[\vecOp{m}] = F[\matOp{\Gamma}](\vecOp{s}, \vecOp{r}_{m_r}, t_{m_t}) + \epsilon_{\vecOp{m}}, \qquad \vecOp{m}\in\{1, \ldots, M_r\} \times\{1, \ldots, M_t\}.
\eea
The noise terms are i.i.d. $\epsilon_{\vecOp{m}} \sim \mathcal{N}(0, \sigma_{m_t}^2)$, with a variance set separately for each acquisition time to achieve a prescribed signal-to-noise ratio $\mathrm{SNR}_{\mathrm{dB}}$. Writing
\bea
    P_{m_t} = \frac{1}{M_r}\sum_{m_r=1}^{M_r} F[\matOp{\Gamma}](\vecOp{s}, \vecOp{r}_{m_r}, t_{m_t})^2
\eea
for the mean-square signal power of the $m_t$-th projection, the noise standard deviation is
\bea
    \sigma_{m_t} = \sqrt{P_{m_t}\, 10^{-\mathrm{SNR}_{\mathrm{dB}}/10}},
    \qquad\text{so that}\qquad
    \mathrm{SNR}_{\mathrm{dB}} = 10\log_{10}\frac{P_{m_t}}{\sigma_{m_t}^2}
\eea
is identical for every frame. For example, in an experiment where $\mathrm{SNR}_{\mathrm{dB}} = 20$ is prescribed, the noise RMS $\sigma_{m_t}$ is $10\%$ of the signal RMS in each frame.

For experiment interpretability and to reflect realistic prior knowledge, aspects of the simulation model \eqref{eq:simulation_model} were also used to initialize the estimates. In particular, Stage~1 velocity estimates $\widehat{\vecOp{v}}_n$ were initialized according to $\mathcal{P}_{\vecOp{\eta}}$ in \eqref{eq:simulation_model}, but without the trajectory overlap rejection sampling constraint. In Stage~2, the rotations $\widehat{\Theta}_n$ were initialized as i.i.d. samples from $\mathcal{U}(\theta_\textnormal{min}, \theta_\textnormal{max})$, with $\theta_\textnormal{min}$ and $\theta_\textnormal{max}$ assumed known. All other parameters were initialized as fixed values that were consistent across all trials and experiments.

This protocol has limitations that should be kept in mind when reading the results. The data are generated with the same forward model \eqref{eq:param_fwd_opt}, motion model, and detector and time sampling as are used in the reconstruction; under reasonable assumptions, the ejection point and initial acceleration are known exactly; and the initialization distribution coincides with the simulation distribution apart from the overlap constraint. The experiments below therefore test the estimator under a correctly specified model, and do not assess robustness to model or discretization mismatch.

\subsection{Single reconstruction case study}\label{sec:single_case}
Based on the outlined data simulation process, projections for an $N=9$ heterogeneous example dataset with $\mathrm{SNR}_{\mathrm{dB}} = 20$ are provided in Figure \ref{fig:state-sinogram-N5}, with the simulated GMM state and corresponding projections at two specified acquisition times presented on the top row, and a visualization of the complete sinogram time series on the bottom row.

\begin{figure}[!htbp]
    \centering
    \includegraphics[width=\linewidth]{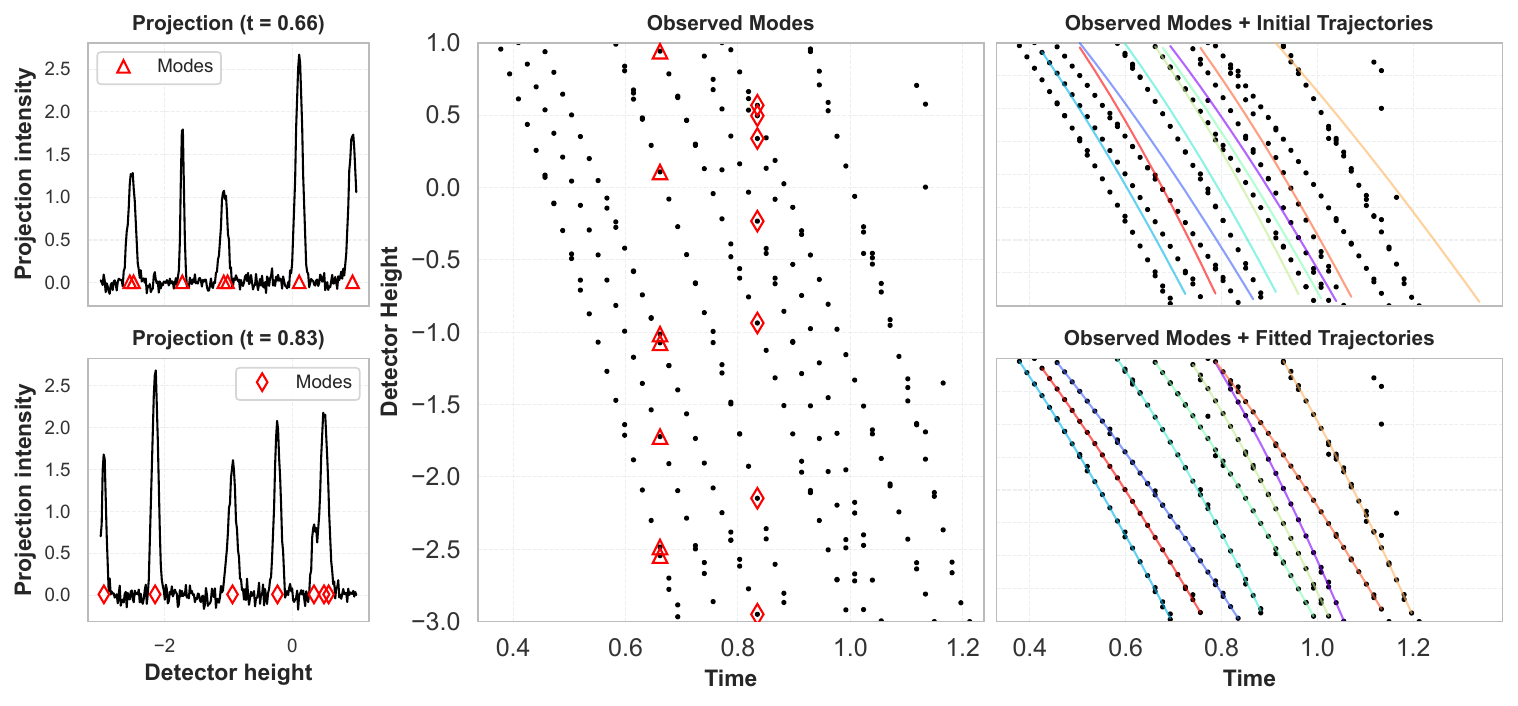}
    \caption{Visualization of the mode detection process and the subsequent trajectory optimization.}
    \label{fig:modes}
\end{figure}

\begin{table}[t]
\centering
\caption{Experiment hyperparameters.}
\label{tab:hyperparameters}
\small
\begin{tabular}{@{}llc@{}}
\toprule
\textbf{Group} & \textbf{Parameter} & \textbf{Value} \\
\midrule
\multicolumn{3}{@{}l}{\emph{Acquisition geometry}} \\
 & Source position & $(-1,\,-1)$\,m\, \\
 & Number of receivers & 256 \\
 & Receiver line & $x = 4$, $y \in [-3,\,1]$ \\
 & Number of projections $M_t$ & 128 \\
 & Time window & $[0,\,2]$\,s\, \\
 & Noise & Gaussian, SNR $=20$\,dB (per frame) \\
\midrule
\multicolumn{3}{@{}l}{\emph{Simulation}} \\
 & Number of particles $N$ & 9 \\
 & Ejection point $\vecOp{\mu}_0$ & $(1,\,1)$\,m\, \\
 & Acceleration $\vecOp{a}_0$ & $(0,\,-9.81)$\,m\,s$^{-2}$ \\
 & Angular velocity $\Theta_n$ & $\mathcal{U}[3,\,6]$ \,rad s$^{-1}$\, \\
 & Amplitude $\alpha_n$ & $15 + 5(n - 1) + \mathcal{N}(0,1)$ \\
 & Initial velocity $\vecOp{v}_n$ (before rejection) & $v_x \sim \mathcal{N}(1,\,0.5^2)$, $v_y \sim \mathcal{N}(1,\,2^2)$ \\
 & Trajectory separation $\delta_{\mathrm{sep}}$ & 0.2 \\
 & Max.\ overlap fraction $\kappa$ & 0.4 (relaxed), 0.2 (stringent) \\
 & Size (geometric mean of $\mathrm{diag}\,\matOp{U}_n$) & log-uniform on $[16,\,58]$ \\
 & Anisotropy ratio (max/min of $\mathrm{diag}\,\matOp{U}_n$) & $\mathcal{U}[1,\,4]$ proposal (at most $58/16$ after rejection) \\
 & Off-diagonal entries of $\matOp{U}$ & $\mathcal{N}(0,\,3^2)$ \\
\midrule
\multicolumn{3}{@{}l}{\emph{Reconstruction (GMM-CT)}} \\
 & Trajectory trials $N_\textnormal{traj}$ (Stage~1) & $4N$ \\
 & Rotation \& morphology trials $N_\textnormal{morph}$ (Stage~2) & $N$ \\
 & Velocity initialization & $v_x \sim \mathcal{N}(1,\,0.5^2)$, $v_y \sim \mathcal{N}(1,\,2^2)$ \\
 & Initial amplitude $\alpha^{(0)}$ & 10 \\
 & Initial precision factor $\matOp{U}^{(0)}$ & $45\,I_2$ (isotropic) \\
 & Velocity parameterization & $(\log v_x,\, v_y)$ \\
 & Huber threshold $\delta$ in \eqref{eq:Hausdorff_surrogate} and \eqref{eq:morphology_loss} & $0.3$ \\
 & Optimizer & L-BFGS, then Newton refinement \\
 & Max.\ iterations & 5000 \\
 & Gradient tolerance & $10^{-8}$ \\
 & Function tolerance (Newton) & $10^{-10}$ \\
\bottomrule
\end{tabular}
\end{table}

Figure~\ref{fig:modes} provides a visual demonstration of the Stage~1 trajectory optimization procedure. The left column shows the detected modes at the two considered acquisition times. To detect modes in the presence of noise, $N$ (here $9$) one-dimensional Gaussians were fit to the data independently at each time. Fitting for the particles not present in the sinogram at that time was handled by allowing the optimizer to dampen the surplus estimates to zero, while those present provide observations of the modes via the means of the non-zero fitted estimates. This detection process is repeated across all times, providing the complete mode time series in the middle column. The final column visualizes the random start initialization in the top panel that led to the optimal fitted trajectory estimate $\matOp{\Gamma}_P^*$ in the bottom. Figure~\ref{fig:modes} also shows that the Gaussian-based mode detection approach can lead to spurious ``modes'' (see around height $0\le y \le 0.8$ at times $1.1\le t \le 1.2$), yet due to the model-based mapping in Proposition \ref{prop:traj_to_modes_map} and suitability of the surrogate Hausdorff loss function \eqref{eq:Hausdorff_surrogate} to the task, the optimization method is robust to these outliers. 

Figure~\ref{fig:init_recon_soln} shows a state-space visualization of the Stage~1 projection-space solution in Figure \ref{fig:modes} at the considered acquisition times. Once a trajectory estimate is assigned to a Gaussian, it is held fixed for the remainder of the optimization (as characterized by \eqref{eq:ParticleMorphoFixed}), and all subsequent fitting is performed with respect to rotation and morphology. This shows why accurate trajectory recovery is crucial to the performance of the rest of the algorithm: if a modeled path is biased, it is impossible to tune the remaining parameters to fully compensate for this. It is also why the trajectory loss \eqref{eq:Hausdorff_surrogate} is formulated in terms of the parametric trajectories rather than per-frame assignments. Because the predicted modes at all acquisition times are generated by the same $N$ trajectories, \eqref{eq:Hausdorff_surrogate} couples all times through $\matOp{\Gamma}_P$, whereas assigning detected modes to particles independently at each time step (for example, frame-by-frame matching followed by curve fitting) ignores this coupling and propagates per-frame matching errors into the fitted trajectories.

\begin{figure}[!htbp]
    \centering
    \includegraphics[width=\linewidth]{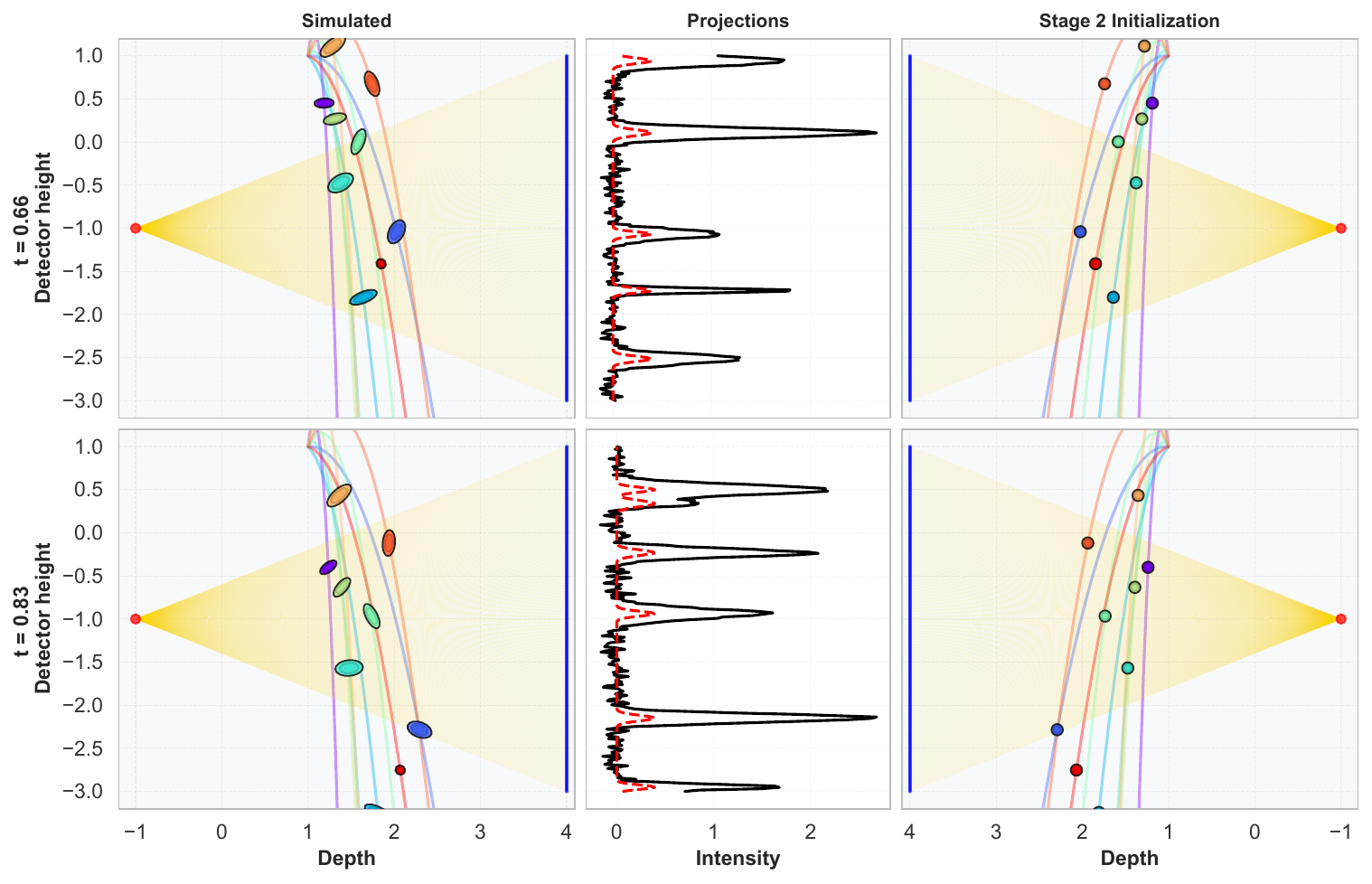}
    \caption{The simulated GMM state (left), the estimated state after Stage~1 (right, mirror imaged), and the respective projections (center). Post trajectory optimization, the predicted projections are centered and assigned to the observed ones, leaving the time fluctuation in the integral mass as the remaining criterion to fit to.}
    \label{fig:init_recon_soln}
\end{figure}

\begin{figure}[!htbp]
    \centering
    \includegraphics[width=\linewidth]{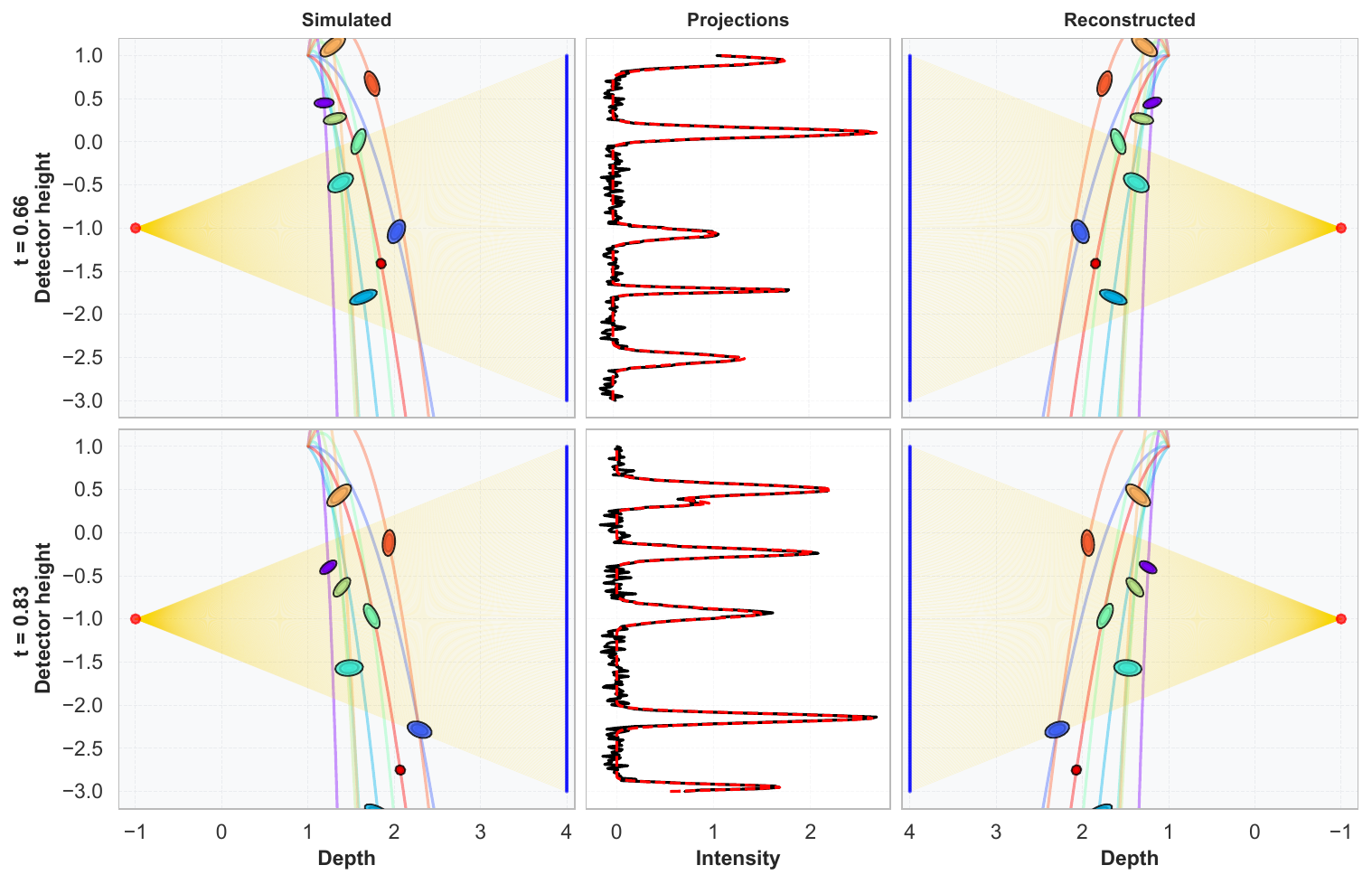}
    \caption{The simulation and corresponding reconstructed output at the sampled acquisition times.}
    \label{fig:optimal_joint_reconstruction}
\end{figure}

Once the trajectory estimates are constructed, the task becomes one of fitting for the particle rotations and morphologies by targeting the remaining mass discrepancy in the projections. Figure~\ref{fig:optimal_joint_reconstruction} shows the simulated GMM and the Stage~2 reconstructed solution at the sampled acquisition times, along with the combined plot of the observed and predicted projections used to inform the fitting.

At $t=0.83$, the two highest particles in the projection window partially obscure one another, which mixes their projections. Despite this, in this simulation they were sufficiently observable to enable their reconstruction. Such success is generally not guaranteed without specifying a relationship between the acquisition sample rate, the particle observability (both in terms of the extent to which it is obscured and its time in view) and the particle-specific angular velocity, which is unknown a priori. Even for particles that are fully unobstructed (i.e. in full view), such a formula is challenging. A sufficient projection rate could hypothetically be specified in terms of a maximum angular velocity, yet Figures~\ref{fig:init_recon_soln}--\ref{fig:optimal_joint_reconstruction} show that particles falling closer to the source in the fan beam geometry setup spend less time in the projection window than those with identical trajectories falling further from it. This suggests a further dependence between trajectory and sampling rate, with the former also unknown in advance. This is also without accounting for scenarios where observability is obscured by the particle interactions, making the derivation of such a formula practically impossible without introducing assumptions that significantly restrict the range of motions and particle interactions we have considered.

\begin{figure}[!htbp]
    \centering
    \includegraphics[width=\linewidth]{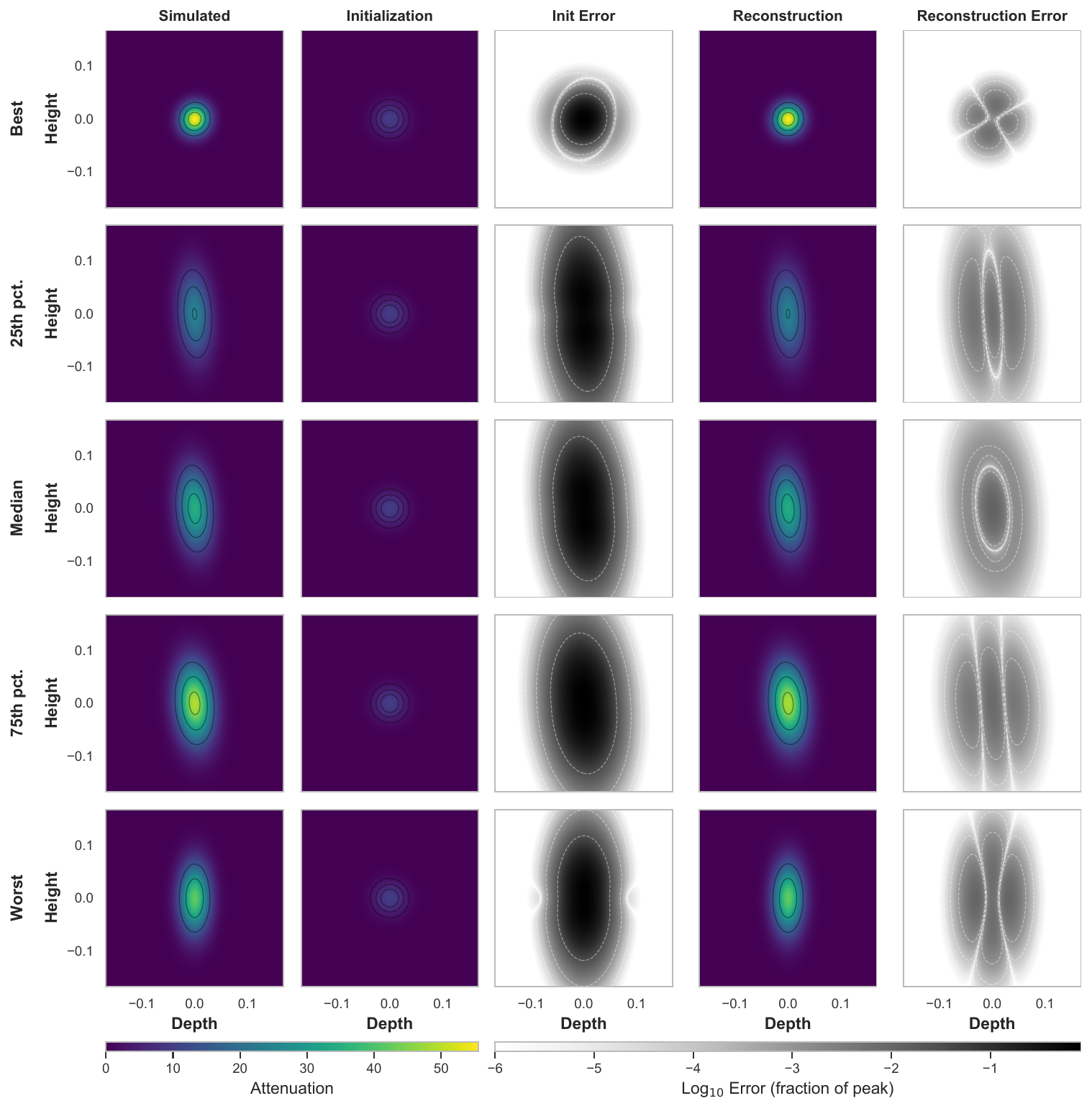}
    \caption{Rigid-body reconstruction performance on a subset of particles for the $N=9$ experiment, measured by the $\log_{10}$ relative absolute error \eqref{eq:log10}. The estimates are rotated so that their principal axes coincide with those of their assigned simulated particle.
    }
    \label{fig:optimal_morphological_reconstruction}
\end{figure}

To demonstrate that the method can reconstruct particles in such scenarios with partial obscuring, the max overlap fraction in $\mathcal{P}_{\vecOp{\eta}}$ was set to the relaxed value $\kappa = 0.4$ (with $\delta_{\mathrm{sep}} = 0.2$), which permits vertical launch velocities as close as $0.25$ (Section~\ref{sec:data_sim}). With $N=9$ particles traversing the field of view within about one second, the projections are crowded, and Section~\ref{sec:baseline} shows that this is close to the particle density at which GMM-CT stops being reliable in this geometry (see also Remark~\ref{rem:limits}). Over the repeated random simulations in the next section, some of the datasets may be unrecoverable by any method, and these cannot be identified without inspecting each one. Our model also does not account for collisions between particles as they fall. The overlap constraint therefore also deters particles from consistently occupying the same point in space, which would excessively limit their identifiability and subsequent chances of reconstruction, while also being physically impossible.

\begin{remark}\label{rem:limits}
    The primary limitation on the reconstructive performance of the method is in identifiability. The sample size and time duration are contributing factors to this, but they alone do not characterize it. For example, a $2$-particle simulation over an arbitrarily long time window is already impossible to reconstruct if one of the particles obstructs the other throughout the acquisition. For this reason, one should not look solely at the scale of $N$ as the means of gauging the limitations of the method. Tables~\ref{tab:gmm_ct_ablation_1} and~\ref{tab:gmm_ct_ablation_2} are consistent with this: forcing the particles apart by reducing $\kappa$ from $0.4$ to $0.2$ reduced the median error of GMM-CT for $N=9$ from $0.38$ to $0.27$, with no meaningful change for $N=5$.
\end{remark}

To show that the reconstruction is not just successful at the sampled state space times in Figures~\ref{fig:init_recon_soln}--\ref{fig:optimal_joint_reconstruction},
in Figure \ref{fig:optimal_morphological_reconstruction} we also plot the rigid-body reconstructions for a subset of particles. Each estimate is first rotated so that its principal axes coincide with those of its assigned simulated particle, so the comparison measures the recovered shape and attenuation but not the orientation; orientation and motion errors enter the spatio-temporal error \eqref{eq:spatio=temporal error} below. The accuracy was measured using the relative error 
\be
    E(\vecOp{x}) := \log_{10}\biggl(\dfrac{|\rho_n(\vecOp{x}) - \widehat{\rho}_n(\vecOp{x})|}{\|\rho_n\|_{L^\infty}}\biggr),\label{eq:log10}
\ee
with values truncated below at $10^{-6}$. In particular, the particles with the most, least and quartile accuracy in the batch are selected. All reconstructions offer orders-of-magnitude improvements in relative accuracy upon their initializations, with output errors in an approximate range of $[10^{-5}, 10^{-3}]$. From a practical CT-automation viewpoint, the rigid body reconstruction accuracy that Figure~\ref{fig:optimal_morphological_reconstruction} visualizes is likely to be the measure that is of most relevance.

\subsection{Baseline comparison and sensitivity analysis}\label{sec:baseline}
We close this section with two further types of numerical experiment: a comparison of GMM-CT with two baselines that fit the projection data directly, and a test of its performance under various projection sparsity and SNR scenarios. The experiments are based on repeated random simulations, but unless specified, all of the simulation and initialization models remain the same as the single reconstruction case study.

Returning to the original notation of \eqref{eq:f} of the task of recovering the function $f:\mathbb{R}^2\times[0, t_\textnormal{max}]\to\mathbb{R}$, in these experiments we measure the reconstruction error using the relative spatio-temporal $L_2$ error
\be
    \varepsilon(\matOp{\Gamma}) := \dfrac{\bigl(\int_0^{t_\textnormal{max}}\int_{\Lambda}|f(\vecOp{x}, t) - \widehat{\rho}(\vecOp{x}, t)|^2d\vecOp{x}dt\bigr)^{1/2}}{\bigl(\int_0^{t_\textnormal{max}}\int_{\Lambda}|f(\vecOp{x}, t)|^2d\vecOp{x}dt\bigr)^{1/2}},\label{eq:spatio=temporal error}
\ee
where $\Lambda = [-1, 5]\times[-3, 3]$ is the modeled spatial domain of interest, $t_\textnormal{max} = 2$ as per Table~\ref{tab:hyperparameters}, and
\bea
    \widehat{\rho}(\vecOp{x}, t) := \sum_{n=1}^N\alpha_n\matOp{A}[\matOp{U}_n, \vecOp{\Theta}_n, \vecOp{\eta}_n](t)\rho_0(\vecOp{x}), \qquad\rho_0(\vecOp{x}) = e^{-\vecOp{x}^\top\vecOp{x}}
\eea
is the in-motion Gaussian mixture resulting from \eqref{eq:NthParticle} configured with a Gaussian template function.

We choose \eqref{eq:spatio=temporal error} as the error metric for several reasons. Primarily, it offers a comprehensive measure of the total spatial error across all experiment times. This is the litmus test for a reconstruction method in this setting: while it is the rigid-body reconstruction that is typically of primary applied interest, recovering the motion has been central to this work and study of its accuracy is therefore relevant within this context. Furthermore, by measuring the error in terms of the physical reconstruction $\widehat{\rho}$ instead of the parameter discrepancies, scenarios of multiple solutions --- should they arise --- are accounted for. Lastly, a further advantage of choosing \eqref{eq:spatio=temporal error} is it enables isotropic particles to be included in the experiments: choosing to measure error in terms of the parameters would mean ruling out simulating particles that were ``too isotropic'', which would possess multiple (possibly infinitely-many) solutions in angular velocity. Equation \eqref{eq:spatio=temporal error} therefore allows a much broader class of datasets to be tested.

Comparisons were conducted against the two following baseline methods: \emph{Direct Huber} (DH), in which one attempts to fit all parameters directly on the projections using the Huber loss \eqref{eq:Huber_loss}; and \emph{Decoupled LS + Huber} (LS-H), in which the decoupling \eqref{eq:ParticleTrajec}--\eqref{eq:ParticleMorpho} is performed, but with MSE loss applied to the raw projections to estimate the trajectories in the first stage, and Huber loss applied in the second. The first approach was chosen simply because it is a typical initial strategy one would likely pursue for the problem, while the second was chosen to study the utility of the trajectory loss \eqref{eq:Hausdorff_surrogate} once a two-stage strategy has been chosen.

For each choice of $N$ from a range of values, a dataset was randomly simulated according to \eqref{eq:simulation_model} and DH, LS-H and GMM-CT were run independently on the dataset, with the respective accuracies measured using \eqref{eq:spatio=temporal error}. This procedure was performed for $10$ randomly-simulated datasets, with performance assessed by collating the results across simulations and analyzing the resulting summary statistics.

A commonality in all of the methods is that they have the potential to diverge, leading to scenarios where solutions are computed by seeking to minimize the sinogram discrepancy under divergent behavior. For trajectories, this biases towards solutions for which particle estimates practically never appear in the sinogram, thus making their contribution essentially zero. For morphology, projection discrepancy in a non-convergent context is minimized by attenuating the estimates to zero, in which case $(\Theta_n, \matOp{U}_n)$ can take any value and have minimal effect on the error.

To handle the errors resulting from divergent runs, values of \eqref{eq:spatio=temporal error} were truncated at $0.9999$. This is because for $\varepsilon \ge 1$, choosing $\widehat{\rho}\equiv0$ (i.e. performing no reconstruction at all) offers better accuracy. To report the statistics, the median was chosen over the mean, since its value is unaffected by the truncation (unless at least $50\%$ of the reconstructions are deemed to have diverged, in which case the median is the truncation value).

\begin{table*}[htbp]
\centering
\caption{Baseline comparison ($\text{SNR} = 20\,\text{dB}$, max overlap fraction $\kappa=0.4$). Reported values are the median relative $L_2$ error of the spatio-temporal density \eqref{eq:spatio=temporal error} $\pm$ standard deviation across 10 random seeds. Errors are capped at $0.9999$ (the error of an empty reconstruction); runs with error $\geq 0.9999$ or a non-finite error are counted as diverged (div.).}
\label{tab:gmm_ct_ablation_1}
\vspace{2mm}
\resizebox{\linewidth}{!}{%
\begin{tabular}{cccc}
\toprule
\textbf{$N$ Gaussians} & \makecell{Direct Huber} & \makecell{Decoupled LS + Huber} & \makecell{GMM-CT} \\
\midrule
$N = 1$ & $0.5511 \pm  0.4995$ \small(5/10 div.) & $0.9999 \pm  0.4985$ \small(6/10 div.) & $0.0436 \pm  0.1602$ \\
$N = 2$ & $0.6664 \pm  0.2208$ & $0.6113 \pm  0.2585$ \small(1/10 div.) & $0.2351 \pm  0.1273$ \\
$N = 5$ & $0.4251 \pm  0.1520$ & $0.4834 \pm  0.0893$ & $0.1127 \pm  0.1129$ \\
$N = 9$ & $0.4496 \pm  0.0941$ & $0.5002 \pm  0.1040$ & $0.3792 \pm  0.2705$ \small(1/10 div.) \\
\bottomrule
\end{tabular}
}
\end{table*}

Table~\ref{tab:gmm_ct_ablation_1} reports the errors from a comparative study where the max overlap fraction hyperparameter was set to $0.4$, allowing trajectories with significant overlap to be sampled by $\mathcal{P}_{\vecOp{\eta}}$. Simulating this way thus enabled greater entanglement between particles in state-space and consequently in their projections, resulting in more difficult reconstruction problems. Under this configuration, it is also possible that some of the simulated experimental datasets were simply impossible to fully reconstruct; for example, the fact that accuracy improves for all methods for $N=5$ in comparison to $N=2$ in Table~\ref{tab:gmm_ct_ablation_1} suggests there may have been some particularly unfavorable datasets sampled for $N=2$.

Table~\ref{tab:gmm_ct_ablation_1} shows that GMM-CT has the lowest median error for every $N$, by a wide margin for $N\le 5$. The margin narrows for $N=9$, where the error of GMM-CT increases sharply and varies widely across datasets, so that with ten datasets the difference from DH is within one standard deviation. We hypothesize that this is because, in the presence of an increasingly saturated observation window, the performance of the mode-based trajectory fitting deteriorates, compromising the performance of the whole pipeline. Figure~\ref{fig:per_seed_relaxed} supports this hypothesis, where for smaller $N$, GMM-CT on average recovers trajectory estimates that are more accurate than DH and LS-H, providing the subsequent Stage~2 fitting process with a more accurate foundation to optimize from. However, $N=9$ is the first configuration for which the average trajectory RMSE is approximately the same for all methods, suggesting that at this particle density the trajectory stage extracts little more information from the data than the initialization already provides. 

Through the positive correlation between trajectory RMSE and $\varepsilon(\matOp{\Gamma})$ in the left panel, Figure~\ref{fig:per_seed_relaxed} also demonstrates the impact Stage~1 has on Stage~2 in GMM-CT. For the other two methods, the per-$N$ bimodality in the trajectory errors (which for DH are synonymous with the full reconstruction) suggests that the success or failure of their reconstruction depends on the favorability of the initialization. In contrast, GMM-CT generally exhibits more stable performance, suggesting less dependence between its initialization and output. With that said, on the one occasion GMM-CT diverged ($N=9$), the right panel of Figure~\ref{fig:per_seed_relaxed} indicates that this is attributable to a failure of trajectory recovery.

\begin{figure}[!htbp]
    \centering
    \includegraphics[width=\linewidth]{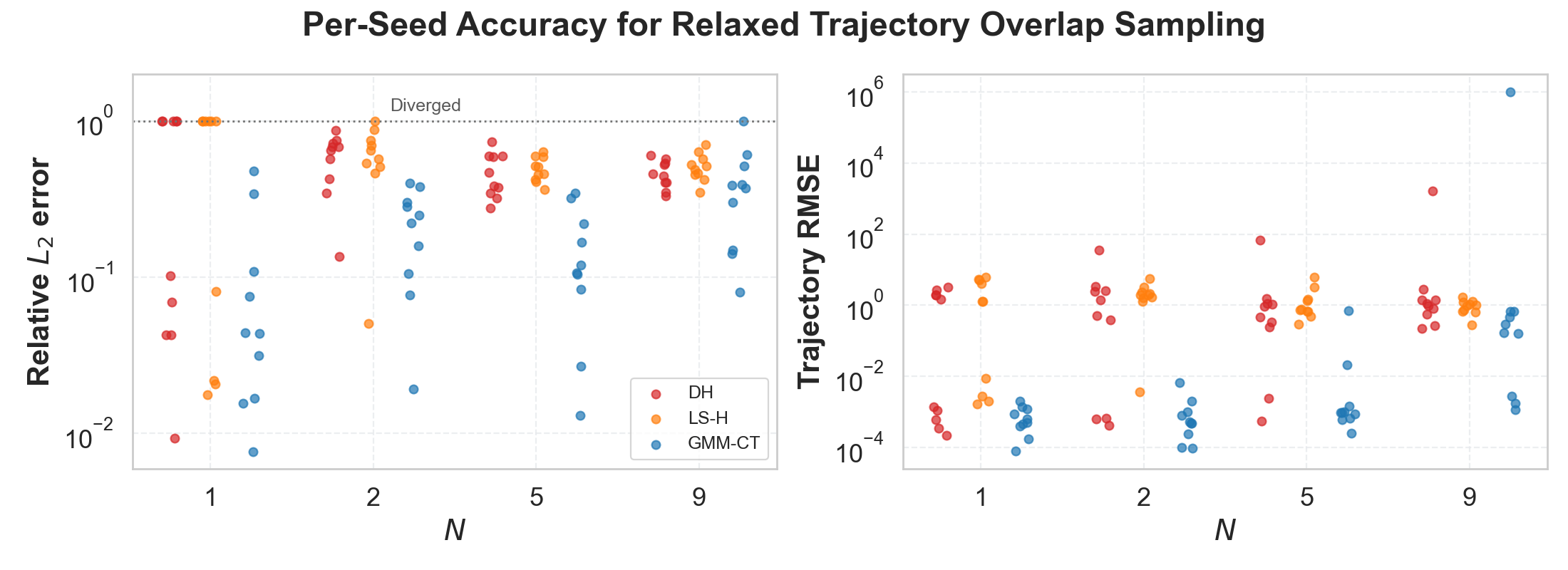}
    \caption{The per-seed error values for each method (left), and the RMSE over all trajectories and all times for each method (right). The max overlap fraction for these results was $\kappa=0.4$, allowing close proximity of particles within the acquisition window.
    }
    \label{fig:per_seed_relaxed}
\end{figure}

\begin{table}[htbp]
\centering
\caption{Mean reconstruction time (s) per simulation, averaged over runs.}
\label{tab:reconstruction-times}
\resizebox{0.56\linewidth}{!}{%
\begin{tabular}{rccc}
\toprule
$N$ & Direct Huber & Decoupled LS + Huber & GMM-CT \\
\midrule
1 & 0.44\, & 0.22\, & 0.88\, \\
2 & 2.85\, & 1.58\, & 1.89\, \\
5 & 59.40\, & 26.74\, & 17.71\, \\
9 & 439.19\, & 169.09\, & 104.22\, \\
\bottomrule
\end{tabular}\label{tab:mean_times}
}
\end{table}

Configuring the applicable values of $N_\textnormal{traj}$ and $N_\textnormal{morph}$ in the benchmarks in a fair way was a challenging aspect of the experiment, since the associated procedures fundamentally differ in their approaches and hence in their complexity (e.g. a prediction in trajectory optimization in GMM-CT is $\mathcal{O}(NM_td)$, while a prediction based on the full projection data in DH is $\mathcal{O}(NM_rM_td^2)$). For $N_\textnormal{traj}$ and $N_\textnormal{morph}$ fixed for GMM-CT according to Table~\ref{tab:hyperparameters}, the results in Table~\ref{tab:gmm_ct_ablation_1} and Figure~\ref{fig:per_seed_relaxed} were computed by running DH with $N_\textnormal{traj}$ multi-starts (chosen favorably towards DH, since $N_\textnormal{traj} > N_\textnormal{morph}$), and LS-H under the same configuration as GMM-CT, but which has a different complexity since the Stage~1 optimization for trajectories is based on the entire projection set. For both baselines, the corresponding wall clock times were comparatively far higher than GMM-CT once $N\ge5$ (Table~\ref{tab:mean_times}). 

While all three methods have the potential to be parallelized, efficient reconstruction in online applications may be crucial from an operational perspective. One could also exploit the overhead consumed by DH and instead apply it to GMM-CT to gain a favorable advantage. To test how much benefit this additional overhead was to DH --- the closest competitor to GMM-CT --- we ran DH again, with $N_\textnormal{traj}$ modified to elicit a roughly equal wall clock time ($\approx 18 \,\textrm{s}$ for $N=5$ and $\approx 105 \,\textrm{s} $ for $N=9$). For $N=5$ the new median $\varepsilon\pm\textnormal{std dev}$ was $0.5893 \pm  0.1080$ and for $N=9$ the new median $\varepsilon\pm\textnormal{std dev}$ was $0.5213 \pm 0.1180$, representing increases in median error for DH of $\approx 39\%$ and $\approx 16\%$ respectively. Equalizing the wall clock time therefore widens the gap between GMM-CT and DH in Table~\ref{tab:gmm_ct_ablation_1}; we did not repeat this test for LS-H.

To test the extent to which identifiability was affecting performance across all $N$, the comparative experiment was run again with the same seeds but with the more stringent overlap constraint $\kappa = 0.2$. For an example run, Figure~\ref{fig:state-sinogram-sparse} shows the corresponding state space snapshots and full sinogram, with particle morphologies and trajectories more distinguishable than when simulating under the previous conditions. Under this constraint the median error of GMM-CT decreased for $N=9$ and was essentially unchanged for $N=5$, while DH and LS-H worsened in three of the four $N\ge5$ cases (Table~\ref{tab:gmm_ct_ablation_2}). One possible explanation for this is that, due to the reliance the benchmark methods have on a favorable random initialization, the methods suffered more in the presence of more rejections in $\mathcal{P}_{\vecOp{\eta}}$, ultimately sampling from a prior distribution that is further from the ground truth. On the other hand, the reliance GMM-CT has on initialization is weaker (Figure~\ref{fig:per_seed_relaxed}), and it is better able to recover multiple trajectories, particularly when they are more distinguishable. If this explanation is correct, GMM-CT should also be less sensitive than the baselines to a mismatch between the initialization distribution and the true distribution of trajectories. As in the previous experiment, the baselines were run with comparatively generous computational budgets, so the advantage of GMM-CT in Table~\ref{tab:gmm_ct_ablation_2} may be underestimated.

\begin{figure}[!htbp]
    \centering
    \includegraphics[width=\linewidth]{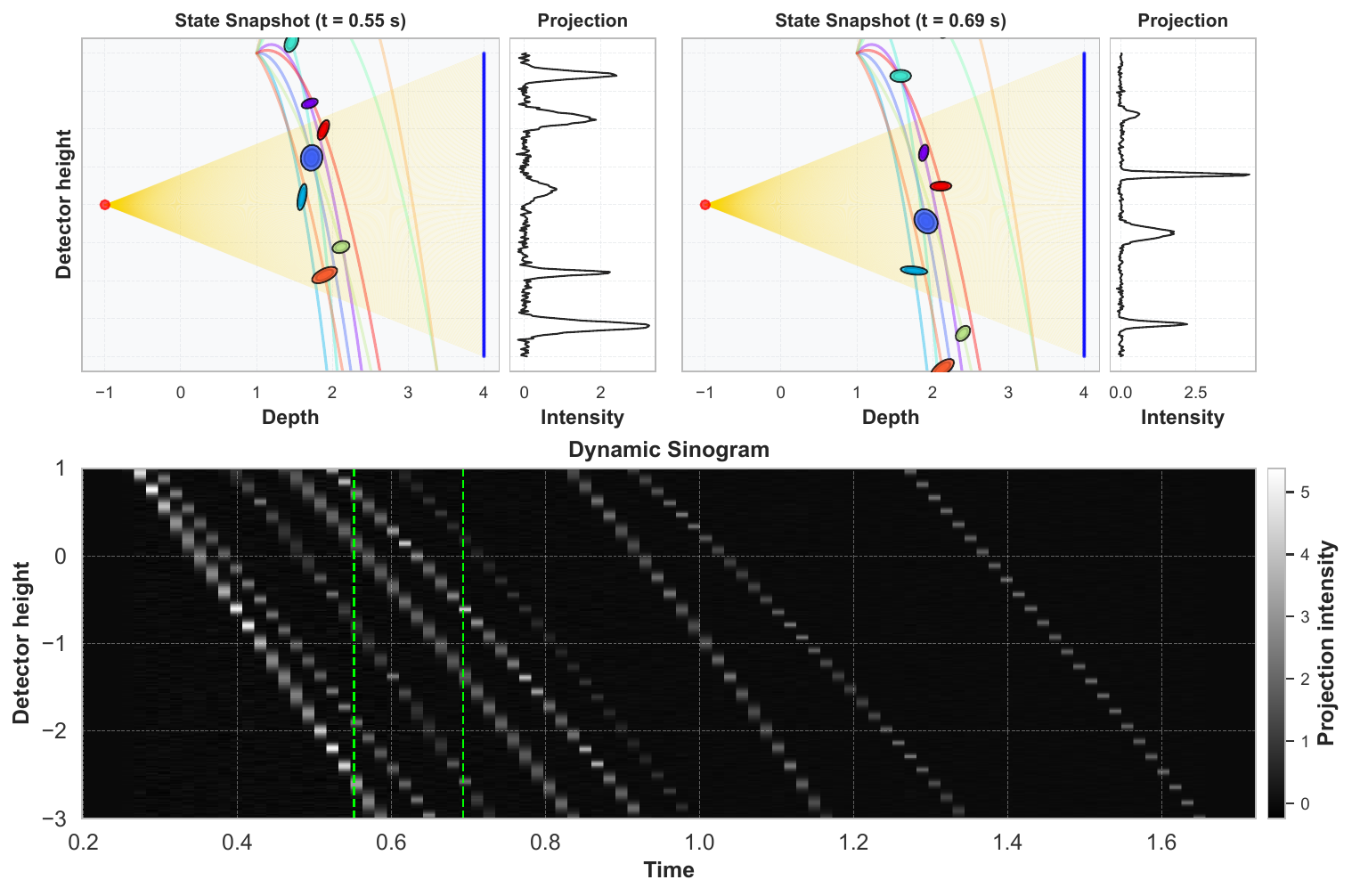}
    \caption{Simulated data for an $N=9$ experiment generated with the stringent overlap constraint $\kappa=0.2$. In contrast to Figure~\ref{fig:state-sinogram-N5}, all trajectories and morphologies are clearly distinguishable.}
    \label{fig:state-sinogram-sparse}
\end{figure}

\begin{table*}[t]
\centering
\caption{The same configuration as that used in Table \ref{tab:gmm_ct_ablation_1}, but with the more stringent overlap constraint $\kappa=0.2$. For $N=1$ the overlap constraint is inactive, so this row coincides with Table~\ref{tab:gmm_ct_ablation_1}.}
\label{tab:gmm_ct_ablation_2}
\vspace{2mm}
\resizebox{0.85\linewidth}{!}{%
\begin{tabular}{cccc}
\toprule
\textbf{$N$ Gaussians} & \makecell{Direct Huber} & \makecell{Decoupled LS + Huber} & \makecell{GMM-CT} \\
\midrule
$N = 1$ & $0.5511 \pm  0.4995$ \small(5/10 div.) & $0.9999 \pm  0.4985$ \small(6/10 div.) & $0.0436 \pm  0.1602$ \\
$N = 2$ & $0.6664 \pm  0.2208$ & $0.6113 \pm  0.2585$ \small(1/10 div.) & $0.2351 \pm  0.1273$ \\
$N = 5$ & $0.5306 \pm  0.1372$ & $0.4666 \pm  0.1000$ & $0.1055 \pm  0.1109$ \\
$N = 9$ & $0.4683 \pm  0.0661$ & $0.5186 \pm  0.1360$ & $0.2678 \pm  0.2534$ \\
\bottomrule
\end{tabular}
}
\end{table*}

Lastly, Figure~\ref{fig:multiple Gaussians} presents a sensitivity analysis of GMM-CT on data simulated with $\kappa=0.2$, including $N=3$ and $N=15$ in addition to the values above. Increasing the number of projections $M_t$ generally improved accuracy up to the largest value tested, $2^7$ (we attribute the few $N$-specific increases between $M_t$ values to the randomness of the simulations), and with $M_t = 2^3$ more than half of the runs diverged for every $N$. In the noise study, accuracy was markedly worse at $\mathrm{SNR}=10$\,dB for all $N$, while above $20$\,dB the error changed little with the SNR. The particle count $N$ had the largest effect: the error generally increased with $N$, and for $N=15$ the median error remained above $0.7$ for every $M_t$ and SNR tested, suggesting saturation of identifiability.

\begin{figure}[ht]
    \centering
    \includegraphics[width=\linewidth]{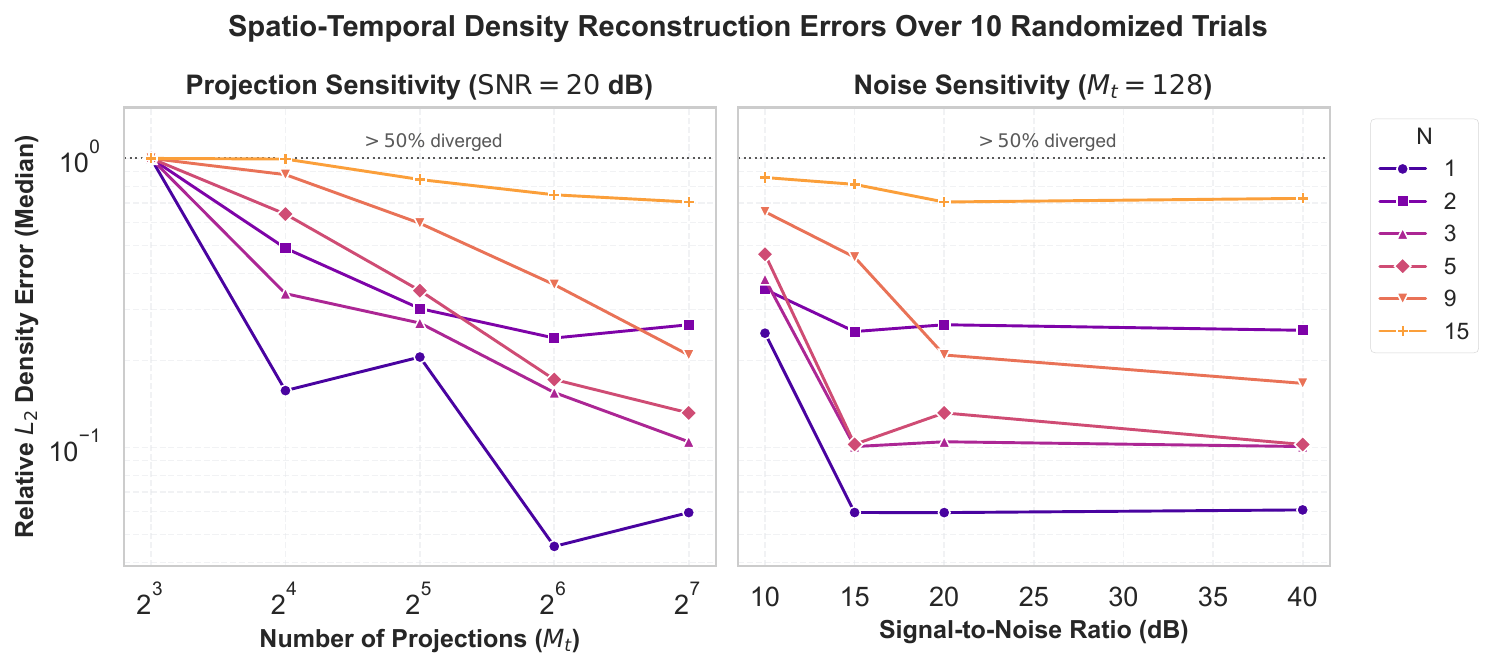}
    \caption{Sensitivity analysis for GMM-CT: median spatio-temporal error $\varepsilon(\matOp{\Gamma})$ over $10$ independent random simulations with $\kappa=0.2$, for different numbers of projections (left, $\mathrm{SNR}=20$\,dB) and SNRs (right, $M_t=128$). The dotted line marks the truncation value $0.9999$, which the median reaches when more than half of the runs are deemed to have diverged.}
    \label{fig:multiple Gaussians}
\end{figure}

\FloatBarrier
\section{Conclusion}\label{sec:conclusion}
Using an informative parametric spatiotemporal model, we have developed a theoretical framework and an accompanying algorithm (GMM-CT) for automated CT reconstruction of Gaussian-mixture particles undergoing unknown rigid-body motion along parabolic trajectories. To maximize generality, the model has been designed to operate on linear combinations of a generic class of particles $(\rho_n)$ given by affine transformations of a canonical template function $\rho_0$. Due to its closed form under the ray transform, $\rho_0$ was chosen as an isotropic Gaussian function, leading to a reconstruction framework based on GMMs. As detailed in Remark~\ref{rem:general_template}, however, the Gaussian is not essential. The method requires only an affine-covariant template projection, together with an equivariant localization feature (the projected mode or centroid) for the trajectory stage. The template projection may be available in closed form, as a precomputed one-dimensional profile for any radially symmetric template, or analytically for standard tomographic bases such as Kaiser--Bessel blobs and B-splines.

We validated GMM-CT on a 2D, 9-particle heterogeneous case study with two levels of observability, and compared it with two baselines that fit the projection data directly over a range of particle counts $N$. GMM-CT had the lowest median error for every $N$ tested, by a wide margin for $N\le5$ and within one standard deviation for $N=9$, and for $N\ge5$ it was also the fastest (Table~\ref{tab:mean_times}). Tightening the observability constraint on the simulated trajectories improved the performance of GMM-CT but generally degraded that of the competitors. We hypothesize that this is because the ground-truth trajectories are drawn with the separation constraint, whereas the multi-start initializations are not, so the initializations resemble the truth less closely. GMM-CT is less sensitive to this mismatch because it recovers trajectories from projection-domain mode paths, which are easier to distinguish when particles are better separated. This explanation is plausible but was not tested directly.

To handle noise, we used a Gaussian-based mode detection method. It occasionally returned false positives, but its output was accurate enough for trajectory recovery, and in the sensitivity study (Figure~\ref{fig:multiple Gaussians}) the error of GMM-CT changed little with the SNR above $20$\,dB and degraded markedly only at $10$\,dB. For the subsequent fitting, Stage~2 \eqref{eq:ParticleMorpho} uses a robust (Huber) projection-domain loss, so we expect it to degrade gracefully with higher noise levels, as observed in the static setting \cite{zickert2022joint}. There, dominant, high-amplitude Gaussians were recovered with errors comparable to the noise-free case, while the weakest components were lost first. The Stage~1 trajectory phase, which relies on detecting and matching projection-domain modes, is the most noise-sensitive step. The experience of \cite{zickert2022joint} suggests that, in addition to robust peak detection, denoising or decomposing the sinogram before mode extraction may be needed for extra stability.

We also expect the decoupled estimator to be advantageous in this low-information, noisy regime. Direct minimization of the joint objective \eqref{eq:inverse_problem_Gamma} is high-dimensional and highly non-convex, and therefore depends on many random multi-starts to escape spurious local minima. Decoupling replaces it with lower-dimensional, better-conditioned subproblems: morphology-independent trajectory recovery, followed by a full projection morphology fit. This reduces the reliance on restarts, which is consistent with the per-seed results in Figure~\ref{fig:per_seed_relaxed}, where the errors of GMM-CT depend less on the initialization than those of the baselines.

Both the reconstruction method and the model used are formulated for applications in Euclidean space of arbitrary dimension. The model assigns each particle its own angular velocity, ballistic trajectory and morphology, with the role of the motion crucial to enabling reconstruction. In particular, the rotation allows different views of the particles to be captured as they traverse the gantry. However, this motion alone is not enough to provide a reconstruction; if the viewing angle and projection times coincide in a way that does not provide a full view, or particles obscure one another, an identifiability issue arises and solutions can be poor or unstable if no constraints are imposed. In general, success depends on a relationship between the acquisition sample rate and the particle motion, which is unknown a priori.

The present study has several limitations. All experiments are two-dimensional and synthetic, and the data were generated with the same forward model, motion model and sampling as are used in the reconstruction, with the ejection point and acceleration known exactly (Section~\ref{sec:data_sim}); robustness to model and discretization mismatch therefore remains to be established. The number of particles $N$ is assumed known. Each particle is represented by a single Gaussian, which captures its size, orientation and aspect ratio but not finer shape. Finally, accuracy depends strongly on how crowded the field of view is: GMM-CT degrades at $N=9$ when the overlap constraint is relaxed and fails for $N=15$ in our geometry (Figure~\ref{fig:multiple Gaussians}). In the application of Section~\ref{sec:SubSurfCharCT}, the number of particles in view is set by the feed rate, so this limit constrains the achievable throughput; quantifying it for a physical system is part of future work.

Several extensions of the current features are natural. We expect the internal components to extend to 3D with modest modification (e.g. developing a mode detection method for a two-dimensional detector), making 3D experiments a natural next step. Other directions include estimating $N$ and reconstructing deformable bodies, for which the morphology parameters $(\alpha_n, \matOp{U}_n)$ become time-dependent. For the application, the most important extension is to use the flexibility of GMMs to represent each particle by several Gaussians that share one rigid motion. The Gaussians within a particle then overlap by design, but they share one trajectory and one angular velocity, so the number of motion parameters grows with the number of particles rather than with the number of Gaussians, and the trajectory stage could track the projected centroid of each particle instead of the modes of its components (Remark~\ref{rem:general_template}). For example, one could transform the projections of a benchmark dataset such as the Walnut dataset \cite{hamalainen2015walnut_data, hamalainen2015walnut_arxiv} to emulate the motion and then compare GMM-CT against the available ground truth. This differs from the experiments here, which involve multiple interacting particles: the walnut would have a single trajectory and rotation to recover, so the task is mainly one of reconstructing the object's structure rather than jointly estimating the motion and morphology of multiple objects. We anticipate that competitive performance would require $N > 100$ Gaussians grouped about a shared center of mass, which would drive the rotation and projectile motion. This requires some modification of the rotation model and an additional offset parameter for each Gaussian relative to the center of mass.

\FloatBarrier

\ack{The authors acknowledge the use of large language models for technical editing and linguistic refinement of this manuscript. Additionally, Anthropic Claude was utilized as a support tool in the development of the numerical code used for the simulations. The authors take full responsibility for the integrity of the final manuscript and correctness of the associated software.}

\funding{No specific funding was received.}

\data{No experimental data were created in this study. The code used to generate the simulated data and all results reported in this article is openly available at \url{https://github.com/dwb26/gmm-ct} and archived on Zenodo at \url{https://doi.org/10.5281/zenodo.24724146}.}

\bibliographystyle{iopart-num}
\bibliography{References}

@article{PaperA,
  author  = {Giovanni S. Alberti and Habib Ammari and Francisco Romero and Timoth{\'e}e Wintz},
  title   = {Dynamic Spike Superresolution and Applications to Ultrafast Ultrasound Imaging},
  journal = {SIAM Journal on Imaging Sciences},
  volume  = {12},
  number  = {3},
  pages   = {1501--1527},
  year    = {2019},
  doi     = {10.1137/18M1174775}
}

@article{PaperB,
  author  = {Martin Holler and Alexander Schl{\"u}ter and Benedikt Wirth},
  title   = {Dimension Reduction, Exact Recovery, and Error Estimates for Sparse Reconstruction in Phase Space},
  journal = {Applied and Computational Harmonic Analysis},
  volume  = {70},
  pages   = {101631},
  year    = {2024},
  doi     = {10.1016/j.acha.2024.101631}
}

@article{PaperC,
  author  = {Bernhard Schmitzer and Klaus P. Sch{\"a}fers and Benedikt Wirth},
  title   = {Dynamic Cell Imaging in {PET} With Optimal Transport Regularization},
  journal = {IEEE Transactions on Medical Imaging},
  volume  = {39},
  number  = {5},
  pages   = {1626--1635},
  year    = {2020},
  doi     = {10.1109/TMI.2019.2953773}
}

@article{PaperD,
  author  = {Martin Burger and Hendrik Dirks and Carola-Bibiane Sch{\"o}nlieb},
  title   = {A Variational Model for Joint Motion Estimation and Image Reconstruction},
  journal = {SIAM Journal on Imaging Sciences},
  volume  = {11},
  number  = {1},
  pages   = {94--128},
  year    = {2018},
  doi     = {10.1137/16M1084183}
}

@article{PaperE,
  author  = {Kristian Bredies and Marcello Carioni and Silvio Fanzon and Francisco Romero},
  title   = {A Generalized Conditional Gradient Method for Dynamic Inverse Problems with Optimal Transport Regularization},
  journal = {Foundations of Computational Mathematics},
  volume  = {23},
  number  = {3},
  pages   = {833--898},
  year    = {2023},
  doi     = {10.1007/s10208-022-09561-z}
}

@incollection{PaperF,
  author    = {Andreas Hauptmann and Ozan {\"O}ktem and Carola-Bibiane Sch{\"o}nlieb},
  title     = {Image Reconstruction in Dynamic Inverse Problems with Temporal Models},
  booktitle = {Handbook of Mathematical Models and Algorithms in Computer Vision and Imaging},
  editor    = {Ke Chen and Carola-Bibiane Sch{\"o}nlieb and Xue-Cheng Tai and Laurent Younes},
  publisher = {Springer},
  address   = {Cham},
  pages     = {1--31},
  year      = {2021},
  doi       = {10.1007/978-3-030-03009-4_83-1}
}

@article{TEMsim2011,
	author = {Hans Rullg{\aa}rd and Lars-G{\"o}ran {\"O}fverstedt and Sergei Masich and Bertil Daneholt and Ozan {\"O}ktem},
	doi = {10.1111/j.1365-2818.2011.03497.x},
	journal = {Journal of Microscopy},
	number = {3},
	pages = {234--256},
	title = {Simulation of transmission electron microscope images of biological specimens},
	volume = {243},
	year = {2011}}

@article{Louis:2011aa,
	author = {Alfred K. Louis},
	doi = {10.1088/0266-5611/27/6/065010},
	journal = {Inverse Problems},
	number = {6},
	pages = {065010},
	title = {Feature reconstruction in inverse problems},
	volume = {27},
	year = {2011}}

@article{Yuan:2025aa,
	author = {Yuan Yuan and Zhaojian Li and Bin Zhao},
	doi = {10.1145/3713070},
	journal = {ACM Computing Surveys},
	number = {7},
	pages = {1--34 (Article No.: 167)},
	title = {A Survey of Multimodal Learning: Methods, Applications, and Future},
	volume = {57},
	year = {2025}}

@article{Yin:2024aa,
	author = {Shukang Yin and Chaoyou Fu and Sirui Zhao and Ke Li and Xing Sun and Tong Xu and Enhong Chen},
	doi = {10.1093/nsr/nwae403},
	journal = {National Science Review},
	number = {12},
	pages = {nwae403},
	title = {A survey on multimodal large language models },
	volume = {11},
	year = {2024}}

@article{Liang:2024aa,
	author = {Paul Pu Liang and Amir Zadeh and Louis-Philippe Morency},
	doi = {10.1145/3656580},
	journal = {ACM Computing Surveys},
	number = {10},
	pages = {1--42 (Article No.: 264)},
	title = {Foundations \& Trends in Multimodal Machine Learning: Principles, Challenges, and Open Questions},
	volume = {56},
	year = {2024}}

@article{Oktem:Adler:2022aa,
	author = {Jonas Adler and Sebastian Lunz and Olivier Verdier and Carola-Bibiane Sch{\"o}nlieb and Ozan {\"O}ktem},
	doi = {10.1088/1361-6420/ac28ec},
	journal = {Inverse Problems},
	number = {7},
	pages = {075006},
	title = {Task adapted reconstruction for inverse problems},
	volume = {38},
	year = {2022}}

@article{zickert2022gaussian,
  title={{Gaussian} mixture model decomposition of multivariate signals},
  author={Zickert, Gustav and Yarman, Can Evren},
  journal={Signal, Image and Video Processing},
  volume={16},
  number={2},
  pages={429--436},
  year={2022},
  doi = {10.1007/s11760-021-01961-y},
  publisher={Springer}
}

@article{zickert2022joint,
  title={Joint {Gaussian} dictionary learning and tomographic reconstruction},
  author={Zickert, Gustav and {\"O}ktem, Ozan and Yarman, Can Evren},
  journal={Inverse Problems},
  volume={38},
  number={10},
  pages={105010},
  year={2022},
  doi = {10.1088/1361-6420/ac8bee},
  publisher={IOP Publishing}
}

@article{knauer2011directed,
  title={The directed {Hausdorff} distance between imprecise point sets},
  author={Knauer, Christian and L{\"o}ffler, Maarten and Scherfenberg, Marc and Wolle, Thomas},
  journal={Theoretical Computer Science},
  volume={412},
  number={32},
  pages={4173--4186},
  year={2011},
  doi = {10.1016/j.tcs.2011.01.039},
  publisher={Elsevier}
}

@article{taha2015efficient,
  title={An efficient algorithm for calculating the exact Hausdorff distance},
  author={Taha, Abdel Aziz and Hanbury, Allan},
  journal={IEEE transactions on pattern analysis and machine intelligence},
  volume={37},
  number={11},
  pages={2153--2163},
  year={2015},
  doi = {10.1109/TPAMI.2015.2408351},
  publisher={IEEE}
}

@article{charpiat2005approximations,
  title={Approximations of shape metrics and application to shape warping and empirical shape statistics},
  author={Charpiat, Guillaume and Faugeras, Olivier and Keriven, Renaud},
  journal={Foundations of Computational Mathematics},
  volume={5},
  number={1},
  pages={1--58},
  year={2005},
  doi = {10.1007/s10208-003-0094-x},
  publisher={Springer}
}

@article{lai2023toward,
  title={Toward the scientific interpretation of geophysical well logs: Typical misunderstandings and countermeasures},
  author={Lai, Jin and Wang, Guiwen and Fan, Qixuan and Zhao, Fei and Zhao, Xin and Li, Yuhang and Zhao, Yidi and Pang, Xiaojiao},
  journal={Surveys in Geophysics},
  volume={44},
  number={2},
  pages={463--494},
  year={2023},
  doi = {10.1007/s10712-022-09746-9},
  publisher={Springer}
}

@article{sun2025enhanced,
  title={Enhanced Lithology Classification Using an Interpretable {SHAP} Model Integrating Semi-Supervised Contrastive Learning and Transformer with Well Logging Data},
  author={Sun, Youzhuang and Pang, Shanchen and Li, Hengxiao and Qiao, Sibo and Zhang, Yongan},
  journal={Natural Resources Research},
  volume={34},
  number={2},
  pages={785--813},
  year={2025},
  doi = {10.1007/s11053-024-10452-z},
  publisher={Springer}
}

@misc{lithology_rock_types,
    author = {PetroWiki},
    title = {Lithology and rock type determination},
    doi = {10.2118/PW0558},
    publisher = {Society of Petroleum Engineers (SPE)}
}

@article{yamada2024liobia,
  title={{LiOBIA}: Object-based cuttings image analysis for automated lithology evaluation},
  author={Yamada, Tetsushi and Di Santo, Simone and Bondabou, Karim and Prashant, Ajeet and Di Daniel, Andrea and Su, Laura and Francois, Matthias and Ouaaba, Khalid and Lockyer, Daniel and Prioul, Romain},
  journal = {Petrophysics},
  volume = {65},
  number = {4},
  pages={624--648},
  year={2024},
  doi = {10.30632/PJV65N4-2024a14}
}

@book{morton1993development,
  title={Development geology reference manual},
  editor={Morton-Thompson, Diana and Woods, Arnold M},
  year={1993},
  series = {AAPG Methods in Exploration Series},
  volume = {10},
  publisher={American Association of Petroleum Geologists}
}

@inproceedings{caja2019image,
  title={Image processing and machine learning applied to lithology identification, classification and quantification of thin section cutting samples},
  author={Caja, Miguel {\'A}ngel and Pe{\~n}a, Andrea Carolina and Campos, Jos{\'e} Rafael and Garc{\'\i}a Diego, Laura and Tritlla, Jordi and Bover-Arnal, Telm and Mart{\'\i}n-Mart{\'\i}n, Juan Diego},
  booktitle={SPE Annual Technical Conference and Exhibition},
  pages={D022S083R001},
  year={2019},
  doi = {10.2118/196117-MS},
  organization={SPE}
}

@article{Sadeghnejad2021DRPReview,
  author  = {Sadeghnejad, Saeid and Enzmann, Frieder and Kersten, Michael},
  title   = {Digital Rock Physics, Chemistry, and Biology: Challenges and Prospects of Pore-Scale Modelling Approach},
  journal = {Applied Geochemistry},
  year    = {2021},
  volume  = {131},
  pages   = {105028},
  doi     = {10.1016/j.apgeochem.2021.105028}
}

@article{Berg2017IndustrialDRP,
  author  = {Berg, Carl Fredrik and Lopez, Olivier and Berland, H{\aa}vard},
  title   = {Industrial Applications of Digital Rock Technology},
  journal = {Journal of Petroleum Science and Engineering},
  year    = {2017},
  volume  = {157},
  pages   = {131--147},
  doi     = {10.1016/j.petrol.2017.06.074}
}

@article{Andra2013DRP1,
  author  = {Andr{\"a}, Heiko and Combaret, Nicolas and Dvorkin, Jack and Glatt, Erik and Han, Junehee and Kabel, Matthias and Keehm, Youngseuk and Krzikalla, Fabian and Lee, Minhui and Madonna, Claudio and Marsh, Mike},
  title   = {Digital Rock Physics Benchmarks---Part I: Imaging and Segmentation},
  journal = {Computers \& Geosciences},
  year    = {2013},
  volume  = {50},
  pages   = {25--32},
  doi     = {10.1016/j.cageo.2012.09.005}
}

@article{Andra2013DRP2,
  author  = {Andr{\"a}, Heiko and Combaret, Nicolas and Dvorkin, Jack and Glatt, Erik and Han, Junehee and Kabel, Matthias and Keehm, Youngseuk and Krzikalla, Fabian and Lee, Minhui and Madonna, Claudio and Marsh, Mike},
  title   = {Digital Rock Physics Benchmarks---Part II: Computing Effective Properties},
  journal = {Computers \& Geosciences},
  year    = {2013},
  volume  = {50},
  pages   = {33--43},
  doi     = {10.1016/j.cageo.2012.09.008}
}

@article{Madonna2012DRP,
  author  = {Madonna, Claudio and Almqvist, Bjarne S. G. and Saenger, Erik H.},
  title   = {Digital Rock Physics: Numerical Prediction of Pressure-Dependent Ultrasonic Velocities Using Micro-CT Imaging},
  journal = {Geophysical Journal International},
  year    = {2012},
  volume  = {189},
  number  = {3},
  pages   = {1475--1482},
  doi     = {10.1111/j.1365-246X.2012.05437.x}
}

@inproceedings{Lenormand2007Cuttings,
  author    = {Lenormand, Roland and Egermann, Patrick and Bouillot, Jean},
  title     = {Porosity and Permeability from Drill Cuttings},
  booktitle = {Proceedings of the 3rd EAGE North African/Mediterranean Petroleum and Geosciences Conference and Exhibition},
  year      = {2007},
  publisher = {European Association of Geoscientists \& Engineers},
  pages     = {cp--16},
  doi       = {10.3997/2214-4609.20146518}
}

@article{Rassenfoss2011DigitalRocks,
  author  = {Rassenfoss, Stephen},
  title   = {Digital Rocks Out to Become a Core Technology},
  journal = {Journal of Petroleum Technology},
  year    = {2011},
  volume  = {63},
  number  = {5},
  pages   = {36--41},
  doi     = {10.2118/0511-0036-JPT}
}

@article{Cong2023ThinInterbedded,
  author  = {Cong, Richao and Yang, Ruiyue and Li, Gensheng and Huang, Zhongwei and Gong, Yanjin and Jing, Meiyang and Lu, Meiquan},
  title   = {Geomechanical Properties of Thinly Interbedded Rocks Based on Micro- and Macro-Scale Measurements},
  journal = {Rock Mechanics and Rock Engineering},
  year    = {2023},
  volume  = {56},
  number  = {8},
  pages   = {5657--5675},
  doi     = {10.1007/s00603-023-03360-w}
}

@article{Chung2019MicroCTCement,
  author  = {Chung, Sang-Yeop and Kim, Ji-Su and Stephan, Dietmar and Han, Tong-Seok},
  title   = {Overview of the Use of Micro-Computed Tomography (Micro-CT) to Investigate the Relation Between the Material Characteristics and Properties of Cement-Based Materials},
  journal = {Construction and Building Materials},
  year    = {2019},
  volume  = {229},
  pages   = {116843},
  doi     = {10.1016/j.conbuildmat.2019.116843}
}

@article{Sanchez2017EDXASTomo,
  author  = {Sanchez, Dario Ferreira and Simionovici, Alexandre S. and Lemelle, Laurence and Cuartero, Vera and Mathon, Olivier and Pascarelli, Sakura and Bonnin, Anne and Shapiro, Russell and Konhauser, Kurt and Grolimund, Daniel and Bleuet, Pierre},
  title   = {2D/3D Microanalysis by Energy Dispersive X-ray Absorption Spectroscopy Tomography},
  journal = {Scientific Reports},
  year    = {2017},
  volume  = {7},
  pages   = {16453},
  doi     = {10.1038/s41598-017-16345-x}
}

@article{Wang2020MicroCTGeomet,
  author  = {Wang, Yuan and Miller, James D.},
  title   = {Current Developments and Applications of Micro-CT for the 3D Analysis of Multiphase Mineral Systems in Geometallurgy},
  journal = {Earth-Science Reviews},
  year    = {2020},
  volume  = {211},
  pages   = {103406},
  doi     = {10.1016/j.earscirev.2020.103406}
}

@article{Rahmani2016MultiscaleShaleCuttings,
  author  = {Rahmani, Reza and Ferrell, Ray E. and Smith, John Rogers},
  title   = {Multiscale Imaging of Fixed-Cutter-Drill-Bit-Generated Shale Cuttings},
  journal = {SPE Reservoir Evaluation \& Engineering},
  year    = {2016},
  volume  = {19},
  pages   = {196--204}
}

@article{Mudde2008TimeResolvingCT,
  author  = {Mudde, R. F. and Alles, J. and van der Hagen, T. H. J. J.},
  title   = {Feasibility Study of a Time-Resolving X-ray Tomographic System},
  journal = {Measurement Science and Technology},
  year    = {2008},
  volume  = {19},
  number  = {8},
  pages   = {085501},
  doi     = {10.1088/0957-0233/19/8/085501}
}

@article{Graas2024TimeResolvedFluidizedBeds,
  author  = {Graas, A. B. and Wagner, E. C. and van Leeuwen, T. and van Ommen, J. R. and Batenburg, K. J. and Lucka, F. and Portela, L. M.},
  title   = {X-ray Tomography for Fully-3D Time-Resolved Reconstruction of Bubbling Fluidized Beds},
  journal = {Powder Technology},
  year    = {2024},
  volume  = {434},
  pages   = {119269},
  doi     = {10.1016/j.powtec.2023.119269}
}

@inproceedings{fathi2025bi,
  title={Bi-level optimization and implicit differentiation as a framework for optimal experimental design in tomography},
  author={Fathi, Hamid and Skorikov, Alexander and van Leeuwen, Tristan},
  booktitle={10th International Conference on Scale Space and Variational Methods in Computer Vision ({SSVM 2025})},
  editor = {Bubba, Tatiana A. and Gaburro, Romina and Gazzola, Silvia and Papafitsoros, Kostas and Pereyra, Marcelo and Sch\"onlieb, C.-B.},
  pages={123--135},
  year={2025},
  doi = {10.1007/978-3-031-92369-2_10},
  organization={Springer}
}

@misc{Bafna:2025aa,
	archiveprefix = {arXiv},
	author = {Mihir Bafna and Bowen Jing and Bonnie Berger},
    doi = {10.48550/arXiv.2509.01038},
	eprint = {2509.01038},
	primaryclass = {q-bio.BM},
	title = {Learning residue level protein dynamics with multiscale {Gaussians}},
	year = {2025}}

@misc{qu2025gem,
	archiveprefix = {arXiv},
	author = {Qu, Huaizhi and Wang, Xiao and Zhang, Gengwei and Peng, Jie and Chen, Tianlong},
    doi = {10.48550/arXiv.2509.25075},
	eprint = {2509.25075},
	primaryclass = {cs.CV},
	title = {{GEM}: {3D} {Gaussian} Splatting for Efficient and Accurate {Cryo-EM} Reconstruction},
	year = {2025}}

@misc{chen2025cryosplat,
	archiveprefix = {arXiv},
	author = {Chen, Suyi and Ling, Haibin},
    doi = {10.48550/arXiv.2508.04929},
	eprint = {2508.04929},
	primaryclass = {eess.IV},
	title = {{CryoSplat}: {Gaussian} Splatting for {Cryo-EM} Homogeneous Reconstruction},
	year = {2025}}

@misc{Zhang2026GaussianSTEM,
	archiveprefix = {arXiv},
	author = {Zhang, Beiyuan and Li, Hesong and Shao, Ruiwen and Fu, Ying},
	doi = {10.48550/arXiv.2604.04693},
	eprint = {2604.04693},
	primaryclass = {cs.CV},
	title = {{3D} {Gaussian} Splatting for Annular Dark-Field Scanning Transmission Electron Microscopy Tomography Reconstruction},
	year = {2026}}

@misc{Halacheva2025GaussianVLM,
  author  = {Halacheva, Anna-Maria and Zaech, Jan-Nico and Wang, Xi and Paudel, Danda Pani and Van Gool, Luc},
  title   = {{GaussianVLM}: Scene-Centric {3D} Vision-Language Models using Language-Aligned {Gaussian} Splats},
  journal = {arXiv preprint},
  year    = {2025},
  primaryclass = {cs.CV},
  eprint  = {2507.00886},
  doi = {10.48550/arXiv.2507.00886},
  archivePrefix = {arXiv}
}

@article{Kerbl2023GaussianSplatting,
  author  = {Kerbl, Bernhard and Kopanas, Georgios and Leimk{\"u}hler, Thomas and Drettakis, George},
  title   = {3D Gaussian Splatting for Real-Time Radiance Field Rendering},
  journal = {ACM Transactions on Graphics},
  year    = {2023},
  volume  = {42},
  number  = {4},
  pages   = {1--14},
  doi     = {10.1145/3592433}
}

@misc{Gao2025NeRFReview,
	archiveprefix = {arXiv},
	author = {Gao, Kyle and Gao, Yina and He, Hongjie and Lu, Dening and Xu, Linlin and Li, Jonathan},
    doi = {10.48550/arXiv.2210.00379},
	eprint = {2210.00379},
	primaryclass = {cs.CV},
	title = {{NeRF}: Neural Radiance Field in {3D} Vision: A Comprehensive Review (Updated Post-{Gaussian} Splatting)},
	year = {2026}}

@book{burnham2002model,
  title={Model selection and multimodel inference: a practical information-theoretic approach},
  author={Burnham, Kenneth P and Anderson, David R},
  year={2002},
  publisher={Springer}
}

@article{stoica2004model,
  title={Model-order selection: a review of information criterion rules},
  author={Stoica, Petre and Selen, Yngve},
  journal={IEEE signal processing magazine},
  volume={21},
  number={4},
  pages={36--47},
  year={2004},
  publisher={IEEE}
}

@book{natterer2001mathematics,
  title={The Mathematics of Computerized Tomography},
  author={Natterer, Frank},
  series={Classics in Applied Mathematics},
  volume={32},
  publisher={Society for Industrial and Applied Mathematics},
  address={Philadelphia, PA},
  year={2001},
  note={Reprint of the 1986 original},
  doi={10.1137/1.9780898719284}
}

@book{kak2001principles,
  title={Principles of Computerized Tomographic Imaging},
  author={Kak, Avinash C. and Slaney, Malcolm},
  series={Classics in Applied Mathematics},
  publisher={Society for Industrial and Applied Mathematics},
  address={Philadelphia, PA},
  year={2001},
  doi={10.1137/1.9780898719277}
}

@book{natterer2001mathematical,
  title={Mathematical Methods in Image Reconstruction},
  author={Natterer, Frank and W{\"u}bbeling, Frank},
  series={SIAM Monographs on Mathematical Modeling and Computation},
  publisher={Society for Industrial and Applied Mathematics},
  address={Philadelphia, PA},
  year={2001},
  doi={10.1137/1.9780898718324}
}

@book{herman2009fundamentals,
  title={Fundamentals of Computerized Tomography: Image Reconstruction from Projections},
  author={Herman, Gabor T.},
  edition={2},
  series={Advances in Pattern Recognition},
  publisher={Springer},
  address={London},
  year={2009},
  doi={10.1007/978-1-84628-723-7}
}

@book{helgason1999radon,
  title={The Radon Transform},
  author={Helgason, Sigurdur},
  edition={2},
  series={Progress in Mathematics},
  volume={5},
  publisher={Birkh{\"a}user},
  address={Boston, MA},
  year={1999},
  doi={10.1007/978-1-4757-1463-0}
}

@article{lewitt1992alternatives,
  title={Alternatives to voxels for image representation in iterative reconstruction algorithms},
  author={Lewitt, Robert M.},
  journal={Physics in Medicine and Biology},
  volume={37},
  number={3},
  pages={705--716},
  year={1992},
  doi={10.1088/0031-9155/37/3/015}
}

@article{matej1996practical,
  title={Practical considerations for {3-D} image reconstruction using spherically symmetric volume elements},
  author={Matej, Samuel and Lewitt, Robert M.},
  journal={IEEE Transactions on Medical Imaging},
  volume={15},
  number={1},
  pages={68--78},
  year={1996},
  doi={10.1109/42.481442}
}

@article{entezari2012box,
  title={A box spline calculus for the discretization of computed tomography reconstruction problems},
  author={Entezari, Alireza and Nilchian, Masih and Unser, Michael},
  journal={IEEE Transactions on Medical Imaging},
  volume={31},
  number={8},
  pages={1532--1541},
  year={2012},
  doi={10.1109/TMI.2012.2191417}
}

@book{bracewell2000fourier,
  title={The Fourier Transform and Its Applications},
  author={Bracewell, Ronald N.},
  edition={3},
  publisher={McGraw-Hill},
  address={Boston, MA},
  year={2000}
}

@misc{hamalainen2015walnut_data,
  author       = {H{\"a}m{\"a}l{\"a}inen, Keijo and Harhanen, Lauri and Kallonen, Aki and Kujanp{\"a}{\"a}, Antti and Niemi, Esa and Siltanen, Samuli},
  title        = {{Tomographic X-ray data of a walnut}},
  month        = feb,
  year         = 2015,
  publisher    = {Zenodo},
  version      = {1.0},
  doi          = {10.5281/zenodo.1254206},
  url          = {https://doi.org/10.5281/zenodo.1254206}
}

@online{hamalainen2015walnut_arxiv,
  author       = {H{\"a}m{\"a}l{\"a}inen, Keijo and Harhanen, Lauri and Kallonen, Aki and Kujanp{\"a}{\"a}, Antti and Niemi, Esa and Siltanen, Samuli},
  title        = {Tomographic X-ray data of a walnut},
  year         = {2015},
  eprint       = {1502.04064},
  archivePrefix = {arXiv},
  primaryClass = {math.NA},
  url          = {https://arxiv.org/abs/1502.04064}
}

@article{huber1964,
  author    = {Huber, Peter J.},
  title     = {Robust Estimation of a Location Parameter},
  journal   = {The Annals of Mathematical Statistics},
  volume    = {35},
  number    = {1},
  pages     = {73--101},
  year      = {1964},
  publisher = {Institute of Mathematical Statistics},
  doi       = {10.1214/aoms/1177703732}
}

@misc{Singh2025GaussianET,
  author        = {Singh, Nalini M. and Chien, Tiffany and McCray, Arthur R. C. and Ophus, Colin and Waller, Laura},
  title         = {A {G}aussian Parameterization for Direct Atomic Structure Identification in Electron Tomography},
  year          = {2025},
  eprint        = {2512.15034},
  archivePrefix = {arXiv},
  doi           = {10.48550/arXiv.2512.15034}
}

@misc{ISO13322-2,
  author       = {{International Organization for Standardization}},
  title        = {{ISO} 13322-2:2021 {P}article size analysis -- {I}mage analysis methods -- {P}art 2: {D}ynamic image analysis methods},
  year         = {2021},
  howpublished = {Geneva: ISO}
}

@misc{Sun2024EmbodiedGS,
  author        = {Abou-Chakra, Jad and Rana, Krishan and Dayoub, Feras and S{\"u}nderhauf, Niko},
  title         = {Physically Embodied {G}aussian Splatting: A Realtime Correctable World Model for Robotics},
  year          = {2024},
  eprint        = {2406.10788},
  archivePrefix = {arXiv},
  doi           = {10.48550/arXiv.2406.10788}
}

@misc{Kong2024Survey3DGSRobotics,
  author        = {Zhu, Siting and Wang, Guangming and Kong, Xin and Kong, Dezhi and Wang, Hesheng},
  title         = {{3D} {G}aussian Splatting in Robotics: A Survey},
  year          = {2024},
  eprint        = {2410.12262},
  archivePrefix = {arXiv},
  doi           = {10.48550/arXiv.2410.12262}
}

@article{Bouaziz2025GaussianNeuralFields,
  author  = {Bouzidi, Abdelaziz and Laga, Hamid and Wannous, Hazem and Sohel, Ferdous},
  title   = {{GNF}: {G}aussian Neural Fields for Multidimensional Signal Representation and Reconstruction},
  journal = {Computer Graphics Forum},
  year    = {2025},
  doi     = {10.1111/cgf.70232}
}

@misc{Yao2026NeRF,
  author        = {Yao, Mingyuan and Huo, Yukang and Ran, Yang and Tian, Qingbin and Wang, Ruifeng and Wang, Haihua},
  title         = {Neural Radiance Field-based Visual Rendering: A Comprehensive Review},
  year          = {2024},
  eprint        = {2404.00714},
  archivePrefix = {arXiv},
  doi           = {10.48550/arXiv.2404.00714}
}

@article{Chen2023GMMCryoEM,
  author  = {Chen, Muyuan and Schmid, Michael F. and Chiu, Wah},
  title   = {Improving resolution and resolvability of single-particle cryo{EM} structures using {G}aussian mixture models},
  journal = {Nature Methods},
  volume  = {21},
  number  = {1},
  pages   = {37--40},
  year    = {2024},
  doi     = {10.1038/s41592-023-02082-9}
}

\end{document}